\documentclass[11pt,reqno]{amsart}
\usepackage{times}
\usepackage{amsmath,amsfonts,amstext,amssymb,amsbsy,amsopn,amsthm,eucal}
\usepackage{txfonts}
\usepackage{dsfont}
\usepackage{graphicx}   
\usepackage{hyperref}
\usepackage{color}
\usepackage{verbatim}

\newcommand{\Ric}{\text{Ric}}

\newcommand{\Vol}{\text{Vol}}
\newcommand{\diam}{\text{diam}}

\newcommand{\dN}{\mathds{N}}

\newcommand{\dQ}{\mathds{Q}}
\newcommand{\dR}{\mathds{R}}

\newcommand{\dZ}{\mathds{Z}}

\newcommand{\cA}{\mathcal{A}}
\newcommand{\cB}{\mathcal{B}}

\newcommand{\cH}{\mathcal{H}}

\newcommand{\cK}{\mathcal{K}}

\newcommand{\cM}{\mathcal{M}}

\newcommand{\cV}{\mathcal{V}}

\newtheorem{theorem}{Theorem}[section]

\newtheorem{proposition}[theorem]{Proposition}
\newtheorem{lemma}[theorem]{Lemma}
\newtheorem{corollary}[theorem]{Corollary}
\theoremstyle{definition}
\newtheorem{definition}[theorem]{Definition}
\theoremstyle{remark}
\newtheorem{remark}{Remark}[section]
\theoremstyle{remark}
\newtheorem{example}{Example}[section]
\theoremstyle{remark}

\theoremstyle{remark}
\newtheorem{question}{Question}[section]
\theoremstyle{remark}

\begin{document}

\title{Fundamental Groups and the Milnor Conjecture}

\author{Elia Bru\`e, Aaron Naber and Daniele Semola}

\date{\today}
\maketitle

\begin{abstract}
It was conjectured by Milnor in 1968 that the fundamental group of a complete manifold with nonnegative Ricci curvature is finitely generated.  The main result of this paper is a counterexample, which provides an example $M^7$ with $\Ric\geq 0$ such that $\pi_1(M)=\dQ/\dZ$ is infinitely generated.

There are several new points behind the result.  The first is a new topological construction for building manifolds with infinitely generated fundamental groups, which can be interpreted as a smooth version of the fractal snowflake.  The ability to build such a fractal structure will rely on a very twisted gluing mechanism.  Thus the other new point is a careful analysis of the mapping class group $\pi_0\text{Diff}(S^3\times S^3)$ and its relationship to Ricci curvature.  In particular, a key point will be to show that the action of $\pi_0\text{Diff}(S^3\times S^3)$ on the standard metric $g_{S^3\times S^3}$ lives in a path connected component of the space of metrics with $\Ric>0$. 
\end{abstract}

\tableofcontents

\section{Introduction}\label{s:Intro}

The study of the structure of the fundamental group $\pi_1(M)$ of a manifold with lower Ricci curvature bounds has received a good deal of attention, and at this point its structural properties are very well understood.  Before discussing the results of this paper let us focus for a moment on some background about what is known.\\

One of the earliest results in the analysis of spaces with lower Ricci curvature bounds is by John Milnor \cite{Milnor}.  Milnor used an early version of volume comparison by Bishop \cite{Bishop} in order to show that if $M^n$ has nonnegative Ricci curvature, then any finitely generated subgroup of the fundamental group $\pi_1(M)$ has polynomial growth.  These results led Milnor to conjecture that the fundamental group need automatically be finitely generated.\\

The importance of the polynomial growth condition to the inherent structure of a group became clear when Gromov \cite{Gromovgroups} proved that any finitely generated polynomial growth group must be almost nilpotent, that is, must have a nilpotent subgroup of finite index.  Combining this with Milnor's result we see that any finitely generated subgroup of $\pi_1(M)$, when $M$ has nonnegative Ricci curvature, is itself almost nilpotent.  Wilking \cite{Wilking} gives a form of converse to this statement, where by building on the work of Wei \cite{Wei} he can show that for any finitely generated almost nilpotent group there exists a manifold with nonnegative Ricci curvature which has this group as its fundamental group.\\

In the context of lower sectional curvature one could do more.  Gromov proved in \cite{Gromovalmostflat} that the local fundamental group \footnote{The image $\pi_1(B_{\epsilon(n)}(p))\to \pi_1(B_{1}(p))$.} is always generated by a uniformly finite number of generators.  This gave the first real hints toward finite generation.  The next major breakthrough on relating the structure of the fundamental group with geometry came from Fukaya and Yamaguchi \cite{FukayaYamaguchi}.  They proved that on a space with lower sectional curvature bounds the local fundamental group is almost nilpotent.  This influential work gave the first real structure theory for the fundamental group.  A subtle point in their work is that the index of the nilpotent subgroup of the local fundamental group was not uniformly controlled.  This point was resolved in the work of Kapovitch, Petrunin and Tuschmann \cite{KapovitchPetruninTuschmann}.  Fukaya and Yamaguchi went on to conjecture in \cite{FukayaYamaguchi} that in the nonnegative sectional context a manifold should have almost abelian fundamental group with the index of the abelian subgroup dimensionally bounded.  An interesting example of Wei \cite{Wei} shows this conjecture cannot hold for manifolds with nonnegative Ricci curvature, though the conjecture remains open for spaces with nonnegative sectional curvature. \\

The results and techniques of Fukaya and Yamaguchi were extended to the context of lower Ricci bounds by Kapovitch and Wilking \cite{KapovitchWilking}.  Among the important applications of this extension was to understand that for a manifold with nonnegative Ricci curvature, a finitely generated subgroup of the fundamental group has a dimensionally bounded number of generators.  A result by Colding and Naber \cite{ColdingNaberholder} proves that the isometry group of a limit of spaces with lower Ricci curvature bounds is a Lie group, and combining this with their structure, Kapovitch and Wilking \cite{KapovitchWilking} are able to give a fairly comprehensive understanding of the fundamental group in the compact case.  In \cite{Wilking} Wilking was able to show how a counterexample to the Milnor conjecture must arise from an abelian action. \\

In low dimensions the Milnor conjecture has been resolved.  At its heart this is because one can prove much stronger rigidities in these contexts, and control much more than just the fundamental group.  In dimension two Cohn-Vossen \cite{Cohn-Vossen} proved that if $M^2$ satisfies $\Ric\geq 0$ and is noncompact, then $M$ is flat or diffeomorphic to $\dR^2$.  In particular, that $M^2$ has finitely generated fundamental group is an easy consequence.  In dimension three the first major result was by Schoen-Yau \cite{SchoenYau}, where they proved that if $\Ric>0$ for a noncompact $M^3$, then it is diffeomorphic to $\dR^3$.  Their proof was unique in comparison to the techniques used in other papers being cited, and relied heavily on minimal surface theory.  Their program was expanded on by Liu \cite{Liu}, who was able to prove that if $M^3$ satisfies $\Ric\geq 0$ then $M^3$ is either diffeomorphic to $\dR^3$ or its universal cover isometrically splits.  The Milnor conjecture is again an easy consequence in this context.  Recently Pan \cite{Pan3d} has given a distinct proof in the three dimensional case.\\

In addition to the broad points of progress mentioned above, let us also mention some of the more specific lines of attack which have had success over the years.  The most rigid result is in the completely noncollapsed case, that is when $\Vol(B_r(p))\geq v r^n$ for all large $r$.  In this case Li \cite{Li} showed that the fundamental group is uniformly finite.  Anderson \cite{Anderson} generalized this to show that if $b_1(M)\geq k$ and $\Vol(B_r(p))\geq v r^{n-k}$, then again $M^n$ has finitely generated fundamental group.  On the opposite end of rigidity, Sormani \cite{Sormanilinear} studied manifolds with minimal growth.  In particular, if a space satisfies small diameter growth $\diam \,\partial B_r\leq \epsilon(n) r$ for all large $r$, then she showed that the fundamental group of $M^n$ is finitely generated.  More recently, Pan \cite{Pancone} has extended these techniques in order to show that if the universal cover of $M^n$ has a unique metric tangent cone at infinity, then the Milnor conjecture holds and $M^n$ has finitely generated fundamental group.  See also \cite{Sormaniloop,SormaniWei1,SormaniWei2,Wu,Pan3,PanWei,Wang}, for many other interesting directions and related results.  \\

\vspace{.5cm}

\subsection{Main Results on Fundamental Groups}

The results of Gromov \cite{Gromovalmostflat}, Fukaya-Yamaguchi \cite{FukayaYamaguchi}, Kapovitch-Wilking \cite{KapovitchWilking} and Wilking \cite{Wilking} thus tell us that the fundamental group $\pi_1(M)$ of a manifold with nonnegative Ricci curvature is well understood, and locally it is uniformly finitely generated.  In particular, even if $\pi_1(M)$ were infinitely generated then necessarily all finitely generated subgroups are $C(n)$-uniformly finitely generated.  The first main result of this paper is to build such an example, and in particular we can take the fundamental group to be the rationals:\\

\begin{theorem}[Infinitely Generated Fundamental Group]\label{t:main_milnor}
	Let $\Gamma\leq \dQ/\dZ \subseteq S^1$ be any subgroup.  Then there exists a smooth complete manifold $(M^7,g)$ with $\pi_1(M)=\Gamma$ and such that $\Ric\geq 0$.
\end{theorem}

We will outline the constructions in Sections \ref{s:outline_milnor_construction} and \ref{s:outline_milnor:inductive} more carefully, however let us begin with a very rough picture of the space and its properties.  There are several topological methods to build spaces with infinitely generated fundamental groups, with the dyadic solenoid complement being a geometrically popular method.  The constructions of this paper are quite distinct.  \\

We will not directly build $M$, instead we will focus on constructing the universal cover $\tilde M$ with the appropriate group action by $\Gamma$.  The overall structure of $\tilde M$, with respect to a basepoint $\tilde p\in \tilde M$, will in many ways mimic that of a fractal snowflake, see Section \ref{s:outline_milnor_construction}.  The ability to build such a fractal structure will rely on a very twisted gluing mechanism.  As we move up in scales we can study the local group $\Gamma_r\equiv \big\langle\gamma: d(\tilde p,\gamma\cdot\tilde p)\leq r\big\rangle\leq \Gamma$, which will jump one generator at a time at scales $r_j$ with $\Gamma_j\equiv\Gamma_{r_j}=\langle\gamma_j,\Gamma_{j-1}\rangle$.  Note that the local group will always be generated by a single action, what jumps is what this generator will be.  At the scales $r_j$ when the local group increases the space will look very close to $S^3\times \dR^4$ with the generating $\gamma_j$ acting by a rotation on both the $S^3$ factor and the $\dR^4$ factor.  \\

A major subtlety of the construction of $\tilde M$ will occur between two of the scales $r_j$ and $r_{j+1}$.  At the bottom scale the generating $\gamma_j$ action will rotate both the $\dR^4$ factor and the $S^3$ factor, while at the top scale the same $\gamma_{j}$ only rotates the $S^3$ factor.  Geometrically the space may look like $S^3\times \dR^4$ at both the $r_j$ and $r_{j+1}$ scales, however one should view these two copies of $S^3\times \dR^4$ quite distinctly.  In particular, the two $3$-spheres in $S^3\times \dR^4 = S^3\times C(S^3)$ will necessarily mix together in order to change the behavior of the action.  We will see this behavior is closely connected to the mapping class group of $S^3\times S^3$.\\ 

As this point is of some independent interest it is worth discussing it briefly, we refer the reader to Section \ref{s:mappingS3S3} for a more in depth discussion.  Let $\cM_0(S^3\times S^3)\equiv \Big\{[g]: g\sim \phi^*g:\phi\in \text{Diff}_0(S^3\times S^3)\Big\}$ represent the space of smooth Riemannian metrics modulo diffeomorphisms which are isotopic to the identity.  From the perspective of gluing and topology a diffeomorphism which is not isotopic to the identity is a highly twisted object, and thus it is good to distinguish between those which are and are not connected to the identity by a continuous path.  We can let $\cM^+_0(S^3\times S^3)\equiv\{[g]\in \cM_0:\Ric>0\}$ be the subset of metrics with positive Ricci curvature.  Note that there is a canonical action of the mapping class group $\pi_0\text{Diff}(S^3\times S^3)$ on these spaces given by $[\phi]\cdot [g]= [\phi^*g]$. One of the main technical lemmas of this paper is that this action of the mapping class group $\pi_0\text{Diff}(S^3\times S^3)$ on the standard metric $g_{S^3\times S^3}$ lives in a connected component of $\cM^+_0(S^3\times S^3)$:\\

\begin{lemma}[Mapping Class Group and Ricci Curvature on $S^3\times S^3$]\label{l:main_mapping_class}
	Let $g_0=g_{S^3\times S^3}$ be the standard metric on $S^3\times S^3$.  Then given $\phi\in {\rm Diff}(S^3\times S^3)$\, there exists a smooth family $g_t$ of metrics  with $\Ric_{g_t}>0$ such that $g_0$ is the standard metric and $g_1=\phi^*g_0$.  That is, the orbit $\pi_0{\rm Diff}(S^3\times S^3)\cdot [g_{S^3\times S^3}]$ of the mapping class group lives in a connected component of $\cM^+_0(S^3\times S^3)$, the space of metrics with strictly positive Ricci curvature.
\end{lemma}
\begin{remark}
Observe that if $\phi\in \text{Diff}_0(S^3\times S^3)$ is isotopic to the identity then the above is trivial as one can take $g_t=\phi_t^*g_{S^3\times S^3}$ to all differ from the standard metric by diffeomorphisms.  If $[\phi]\in \pi_0\text{Diff}(S^3\times S^3)$ is not the trivial element, the above is of course much more subtle.
\end{remark}

An equivariant version of the above will be one of the driving mechanisms allowing us to untwist our actions as the scale increases and slide Euclidean rotations of $S^3\times\dR^4$ at scale $r_j$ to spherical rotations of $S^3\times \dR^4$ at scale $r_{j+1}$.  This process will be described in detail in the next Sections.\\

Geometrically we will have at large scales that $M$ typically looks like a cone over a lens space $\approx C(S^3_s/\dZ_k)$ for some sphere size $s\leq 1$ and $k\in \dN$.  As the scale increases the size of spheres $S^3_s$ will decrease until $M$ is close to a ray, and when the ray opens again $M$ will become close to a potentially different lens space $\approx C(S^3_s/\dZ_{k'})$.  This process will repeat indefinitely, and in the case of a $\dQ/\dZ$-fundamental group one can arrange it so that a cone over every possible lens space occurs infinitely often.  In particular, the tangent cones of $M$ at infinity will include $C(S^3_s/\dZ_k)$ for every choice of $k\in \dN$ and $0\leq s\leq 1$.  It is important to note that the basepoint for the tangent cone at infinity may not always be the cone point itself.  In addition to these lens space tangents, by blowing up at the scale of the actions when $k\to \infty$ we will also see tangent cones at infinity of the form $\dR^3\times S^1$.   \\

We are left with the following open question:

\begin{question}
	If $M^n$ satisfies $\Ric\geq 0$ with $n=4,5$, or $6$, then is $\pi_1(M)$ finitely generated?
\end{question}

The techniques of this paper need to be extended to work in lowest dimensions, and so the above are important open questions.  Additionally, our examples are quite collapsed in nature.  The issue of finite generation is still open in the noncollapsed setting:

\begin{question}
	If $(M^n,g,p)$ satisfies $\Ric\geq 0$ with the universal cover $\tilde M$ noncollapsed, i.e. $\Vol(B_r(\tilde p))\geq v r^n>0$ for all $r>0$, then is $\pi_1(M)$ finitely generated?
\end{question}

\vspace{.5cm}

\section*{Acknowledgements}
The first author would like to express gratitude for the financial support received from Bocconi University. The author acknowledges the support provided by the Giorgio and Elena Petronio Fellowship at the Institute for Advanced Study, where a part of this work was conducted. Additionally, he is grateful to the Northwestern University for the hospitality during his visit. 
\\

The second author was funded by NSF Grant 1809011 during much of this work.\\

The last author was supported by the European Research Council (ERC), under the European's Union Horizon 2020 research and innovation programme, via the ERC Starting Grant  “CURVATURE”, grant agreement No. 802689, while he was employed at the University of Oxford until August 2022. He was supported by the Fields Institute for Research in Mathematical Sciences with a Marsden Fellowship from September 2022 to December 2022. He is currently supported by the FIM-ETH Z\"urich with a Hermann Weyl Instructorship. He is grateful to these institutions, to the Northwestern University and to the Institute for Advanced Study in Princeton, for the excellent working conditions during the completion of this work.

\vspace{.5cm}

\section{Geometric and Topological Outline for Theorem \ref{t:main_milnor}}\label{s:outline_milnor_construction}

The focus of this Section is to describe in broad strokes the Example of Theorem \ref{t:main_milnor}.  In particular, we will begin in Section \ref{ss:outline_milnor:topology_milnor} by outlining the main topological ingredients in the construction, and then in Section \ref{ss:outline_milnor:geometry_milnor} we want to describe the large scale geometry of the construction.  Both of these discussions are meant to help draw an intuitive picture as a preamble to the careful construction given in the next Section.\\

\subsubsection{\bf Decomposing the group $\Gamma\leq \dQ/\dZ$}\label{ss:outline_milnor_construction:decomposition} Let us begin by choosing in $\Gamma\leq \dQ/\dZ\subseteq S^1$ a nested sequence of
finitely generated subgroups $\{e\}=\Gamma_{-1}\leq \Gamma_0\leq \Gamma_1\leq \cdots$ which generate $\Gamma$ in the sense that for every $\gamma\in \Gamma$ we have that $\gamma\in \Gamma_j$ for some $j$ sufficiently large.  For instance such a sequence of subgroups may be built using that $\Gamma$ is countable and choosing an enumeration.  A finitely generated subgroup $\Gamma_j\leq \dQ/\dZ$ is necessarily finite and generated by a single element $\gamma_j\in \Gamma_j$.  In this way we can write 
\begin{align}\label{e:outline_milnor:Gamma}
	\Gamma_j = \big\langle \gamma_j, \Gamma_{j-1}\big\rangle \text{ and $\exists!$ minimal $k_j\in \dN$ such that } \gamma_j^{k_j}=\gamma_{j-1}\, .
\end{align}
It will be convenient to adopt the notation $k_{\le j}\equiv k_0\cdot k_1\cdots\cdot k_j$, for $j\in \mathbb{N}$ and we shall denote by $|\gamma|$ the order of any $\gamma\in\Gamma$.  Notice that, with this notation, $|\gamma_j|=k_{\le j}$.  There is no harm in assume that $k_j>1$ for each $j$, as otherwise $\Gamma_{j}=\Gamma_{j-1}$. \\  

\begin{example}
Let $p$ be a prime and $\Gamma = \langle 1,p^{-1},p^{-2},\ldots\rangle\leq \dQ/\dZ$ be the set of rationals which can be written as a finite series $\gamma = \sum a_i p^{-i}$ with $0\leq a_i<p-1$.  In this case we let $\gamma_i = p^{-i}$, so that $k_i=p$ for all $i$.  We have that $\Gamma_j=\{\gamma = \sum_1^j a_i p^{-i}\}$.  $\qed$	
\end{example}

\begin{example}
Let $\Gamma = \dQ/\dZ$.  Let us choose $k_j$ to cyclically evaluate at the primes, that is
\begin{align}
	\{k_j\} = 2;2,3;2,3,5;2,3,5,7\ldots
\end{align} 
and let $\gamma_j\equiv \frac{1}{k_{\leq j}}$ .  Thus $\Gamma_j$ is the set of all rationals whose denominators are products of primes up to some order and power.  We can take subsequences of $\{k_j\}$ which converge to any prime or to $\infty$. $\qed$	
\end{example}

\begin{example}\label{ex:Gamma=Q_all_integers}
Let $\Gamma = \dQ/\dZ$.  Let us choose $k_j$ to cyclically evaluate at all the integers, that is
\begin{align}
	\{k_j\} = 2;2,3;2,3,4;2,3,4,5\ldots
\end{align} 
and let $\gamma_j\equiv \frac{1}{k_{\leq j}}$ .  We can take subsequences of $\{k_j\}$ which converge to any element of $\dN$ or to $\infty$. $\qed$	
\end{example}

Let us make a few useful observations about the induced structure.  For each $\gamma\in \Gamma$ we can then uniquely write it as

\begin{align}
	\gamma = \prod_{j} \gamma_j^{a_j}\, , \text{ such that }a_j<k_j\, ,
\end{align}
where at most a finite number of $a_j$ are nonvanishing.  Note that there is the short exact sequence $0\to \Gamma_j\to \Gamma\to \Gamma/\Gamma_j\to 0$.  This does not split as a group splitting of course, however the choice of basis builds for us a splitting of sets
\begin{align}\label{e:Gamma_splitting}
	&\Gamma = \Gamma_j\oplus \Gamma/\Gamma_j\, ,\text{ given by }\notag\\
	&\gamma = \gamma_{\leq j}\cdot\gamma_{>j} = \prod_{i\leq j} \gamma_i^{a_i}\cdot \prod_{i>j} \gamma_i^{a_i}\, .
\end{align}

\begin{remark}
It is possible, and helpful, to include into the discussion the case where $\Gamma$ is finitely generated, or equivalently $\Gamma=\Gamma_j$ for some $j$.  This is more in line with how our inductive construction in Section \ref{s:outline_milnor:inductive} will proceed.  However our main focus is of course on the case where $\Gamma$ is not finitely generated.
\end{remark}

\vspace{.3cm}

\subsection{Topological Outline of $(\tilde M,\tilde p,\Gamma)$}\label{ss:outline_milnor:topology_milnor}

Let us open with the topological construction of $\tilde M$ with its group action by $\Gamma$.  We will not worry in this subsection about geometry or preserving Ricci curvature.  Indeed the viewpoint we will take in Section \ref{s:outline_milnor:inductive} when we carefully construct our space will be quite different, however the point of view we use here is particularly convenient for understanding the global structure of our space.\\

\subsubsection{Identifying $\tilde M$ with a Directed Graph} In order to visualize the space it is helpful to build the following directed graph $(V,E)$ of vertices and directed edges.  We should think of each vertex as a copy of $S^3\times D^4\approx S^3\times \dR^4$. 
If a $v^a$ is a given vertex we will sometimes write $S^3\times D^4_a$ in order to explicitly understand that the copy of $S^3\times D^4$ we are staring at is the one represented by $v^a$.\\  

A directed edge $E_{ab}$ will represent for us a gluing.  Given a vertex $v^a\approx S^3\times D^4_a$ note that its boundary is a single $S^3\times S^3$.  This boundary will be glued into the target vertex $v^b\approx S^3\times D^4_b$ by removing a smaller disk $S^3\times D^4_b\setminus (S^3\times D^4_{ab})$ with $D^4_{ab}\subseteq D^4_b$ and choosing a gluing map $\phi_{ab}:\partial (S^3\times D^4_a)\to \partial(S^3\times D^4_{ab})$.  Note that we can identify $\phi_{ab}:S^3\times S^3\to S^3\times S^3$.  The exact choices of $D_{ab}$ and $\phi_{ab}$ will be discussed after the enumeration of the vertices and edges is complete, however it is worth pointing out that the $\phi_{ab}$ will be a nontrivial element of the mapping class group, with the goal of twisting our underlying action.  We see from the above that we should expect each vertex to be the base of at most one directed edge, although it may be the target of multiple edges.\\

In order to enumerate our vertices $V=\{v^a\}$ it is convenient to decompose them as a disjoint union $V=\cup_j V_j$ as follows.  Recall that we will have a global $\Gamma$ action on the end manifold, and so for each of our subgroups $\Gamma_j$ let $V_j=\{v^a_j\}$ represent those vertices whose associated $S^3\times D^4_{a_j}$ will be preserved under the $\Gamma_j$ action.  In this case we will always have that $\Gamma_j\leq S^1$ is induced by the $(1,k_{\le j-1})$-action, which is to say that the generator $\gamma_j$ will (left) Hopf rotate the $\dR^4$ factor by $2\pi/k_j=2\pi k_{\le j-1}/|\gamma_j|$ and will (left) Hopf rotate the $S^3$ factor by $2\pi/|\gamma_j|=2\pi/(k_0\cdots k_j)$.  We will use the notation
\begin{align}
	\theta\cdot_{(a,b)}(g_1,g_2) = (a\theta\cdot g_1,b\theta\cdot g_2)\, ,
\end{align}
where 
$\theta\cdot g$ denotes the (left) Hopf rotation of $S^3$ by angle $\theta$.  We refer to $\cdot_{(a,b)}$ as the $(a,b)$-Hopf rotational action. We will often be viewing $\dR^4=C(S^3)$, and hence the Hopf rotation of $S^3$ naturally induces a rotation on $\dR^4$. Analogous considerations hold for $D^4$, which we view as the ball centered at the origin of $\dR^4$.\\  

Identifying the vertices $v^a_{j}$ with glued copies of $S^3\times D^4_{a_j}$ in the end manifold $(\tilde M,\tilde p, \Gamma)$, we see that we should expect an induced $\Gamma/\Gamma_j$ action on $V_j$.  In fact, this action will be a transitive and free action.  Hence we will enumerate $V_j$ by identifying it directly with
\begin{align}
	V_j \equiv \big\{v^a_j : a\in \Gamma/\Gamma_j\big\} \, .
\end{align}
Thus as a set we have identified $V_j$ with $\Gamma/\Gamma_j$, and from the point of view of our construction we will have one $\Gamma_j$-preserved $S^3\times D^4$ neighborhood per element of $\Gamma/\Gamma_j$.\\

In order to build our directed graph we also need our edges.  Each vertex $v^a_j$ will be the base one of edge, and so we can we view the edges as a map $E:V\to V$.  We will see that each vertex $v^a_j$ will be the target of $k_j$ edges, and indeed on each $V_j\subset V$ our edge map is given by
\begin{align}
	E:V_j\to V_{j+1}\,  \text{ by }\,  E[v^a_j] \equiv v^a_j/\Gamma_{j+1}\, .
\end{align}
In particular, as $|\Gamma_j/\Gamma_{j-1}|=k_j$ we have that if $v^b_j\in \Gamma/\Gamma_j$ then there are exactly $k_{j}$ elements $v^a_{j-1}\in \Gamma/\Gamma_{j-1}$ for which $v^b_j=v^a_{j-1}/\Gamma_j$, as claimed.\\

\subsubsection{The Gluing Maps}\label{ss_gluingtopo}

We have now identified our set of vertices 
\begin{align}
	V=\bigoplus_j V_j =\bigoplus_j \Gamma/\Gamma_j\, ,
\end{align}
and our edges by $E[v^a_j] = v^a_j/\Gamma_{j+1}$.

Let us fix $v^b_j\approx S^3\times D^4_b$ and let $v^a_{j-1}\in \Gamma/\Gamma_{j-1}$ be the $k_j$ vertices such that $v^a_{j-1}/\Gamma_j = v^b_j$.  As discussed in the last subsection, associated to each $v^a_{j-1}\approx S^3\times D^4_a$ there is a disk $D^4_{ab}\subseteq D^4_b$ and a gluing map $\phi_{ab}:S^3\times S^3\to S^3\times S^3$ identifying the boundaries $\partial (S^3\times D^4_{a})$ with $\partial(S^3\times D^4_{ab})$.  Let us discuss these disks and mappings.\\

To begin, note that in the equivalence class $[v^b_j]\in \Gamma/\Gamma_j$ there is a distinguished element with $v^b_j\in \Gamma/\Gamma_{j-1}$. More precisely, if
\begin{equation}
[v^b_j]=\prod_{i>j} \gamma_i^{a_i} \in \Gamma/\Gamma_j\, ,
\end{equation}
then we can write
\begin{equation}
v^b_j\equiv e\cdot\prod_{i>j} \gamma_i^{a_i}=\prod_{i>j} \gamma_i^{a_i}\in \Gamma/\Gamma_{j-1}\, .
\end{equation}

Then we can identify the elements of $v^a_{j-1}\in V_{j-1}$ for which $v^a_{j-1}/\Gamma_j = v^b_j$ as the collection $\{\gamma_j^a v^b_j\}\in V_{j-1}$ for $a=0,\ldots ,k_j-1$.  Now consider the disk $S^3\times D^4_{b}$ and let $D_{0b}= B_r(x^0)\subseteq D_b$ be any ball $0\not\in B_{2r}(x^0)$ which is not too close to the origin, and for which $r<<k_j^{-1}$ .  Let $x^a$ be the rotation of $x^0$ by angle $2\pi a/k_j$, and hence $D_{ab}=B_r(x^a)$ is the rotation of $D_{0b}$ by angle $2\pi a/k_j$.  Note that this is a set $\{D_{ab}\}$ of $k_j$ disjoint balls in $D^4_b$.  By definition the set $\bigcup_a S^3\times D^4_{ab}$ is invariant under the action by $\Gamma_j$.\\

Now we need to define the gluing maps $\phi_{ab}:\partial (S^3\times D^4_{a})\to \partial(S^3\times D^4_{ab})$, that is maps $\phi_{ab}:S^3\times S^3\to S^3\times S^3$.  The challenge is that we need the gluing maps to respect the $\Gamma_j$ actions, and the identity map does not do this.  More specifically, if we consider the glued space
\begin{align}
	\Bigg(S^3\times \big(D^4_b\setminus\bigcup D^4_{ab}\big)\Bigg) \bigcup_{\phi_{ab}}S^3\times D^4_{a}\, ,
\end{align}
then we want a well defined $\Gamma_j$ action.  The $\Gamma_j$ action should restrict to the $(1,k_{\le j-1})$ action on each $S^3\times\big(D_b\setminus\bigcup D_{ab}\big)$, so that in particular the $\Gamma_{j-1}$ action on $S^3\times \big(D_b\setminus\bigcup D_{ab}\big)$ should be the $(1,0)$-action.  However, this same $\Gamma_{j-1}$ action should be the $(1,k_{\le j-2})$ action on each glued copy of $S^3\times D^4_{a}$.  If we unwind this, this is telling us we need a diffeomorphism $\phi_{j}:S^3\times S^3\to S^3\times S^3$ such that
\begin{align}
	\phi_j\Big(\theta \cdot_{(1,k_{\le j-2})}(g_1,g_2)\Big) = \theta\cdot_{(1,0)}\phi_j(g_1,g_2)\, .
\end{align}
In fact we have such a diffeomorphism, see Section \ref{ss:equivariant_metric:examples}.  This diffeomorphism is not isotopic to the identity, a point which causes some trouble on the geometric side of the gluing procedure, see Section \ref{ss_Step2description} for more on this.  For the topological picture we will now define $\phi_{ab}:\partial (S^3\times D^4_{a})\to \partial(S^3\times D^3_{ab})$ by
\begin{equation}
	\phi_{ab}(g_1,g_2)=\gamma_j^a\cdot\phi_j(g_1,g_2)\, .\notag\\ 
\end{equation}
Note that the gluing maps are built precisely so that the $\Gamma_{j-1}$ action on $S^3\times D^4_{a}$ extends to an action of $\Gamma_j$ on the glued space $\Bigg(S^3\times \big(D^4_b\setminus\bigcup D^4_{ab}\big)\Bigg) \bigcup_{\phi_{ab}}S^3\times D^4_{a}$.

\vspace{.3cm}

\subsubsection{Construction of $(\tilde M, \Gamma)$}

Let us build our global space $\tilde M$ as follows.  Having defined our collection of vertices $\{v^a\}\in V=\bigoplus V_j =\bigoplus \Gamma/\Gamma_j$ let us first consider the disjoint collection
\begin{align}
	\bigcup_j\bigcup_{V_j} S^3\times D^4_{b_j}\, ,
\end{align}
where we have assigned to each vertex $v^b_j\in V_j$ a copy of the disk cross a sphere.  Observe that there is a free action of $\Gamma$ on this space, where if $\gamma\in \Gamma$ then its action on $S^3\times D^4_{b_j}$ may be understood in the following manner.  Recall that $\Gamma$ has a set splitting as in \eqref{e:Gamma_splitting} defined by our choice of basis, so that each $\gamma\in \Gamma$ can be written $\gamma=\gamma_{\leq j}\cdot\gamma_{>j}$ where
\begin{align}
	\gamma_{\leq j}=\prod_{i\leq j}\gamma_i^{a_i}\in \Gamma_j\, ,\text{ and }\gamma_{> j}=\prod_{i> j}\gamma_i^{a_i}\in \Gamma/\Gamma_j\, .
\end{align}

Note there is the defined action of $\gamma_{\leq j}\in \Gamma_j\leq S^1$ on $S^3\times D^4_{b_j}$ induced by the $(1,k_{\le j-1})$ action which Hopf rotates $D^4$ at speed $2\pi/k_j$ and Hopf rotates $S^3$ at speed $2\pi/|\gamma_j|=2\pi/(k_0\cdots k_j)$.  Additionally, $\gamma_{>j}\in \Gamma/\Gamma_j$ naturally acts on $b_j\in \Gamma/\Gamma_j$.  This will tell us that $\gamma\cdot:S^3\times D^4_{b_j}\to S^3\times D^4_{\gamma_{>j}\cdot b_j}$.  This action is almost given by the composition of the $\gamma_{\leq j}$ action on $S^3\times D^4_{b_j}$ and the identification of $S^3\times D^4_{b_j}$ with $S^3\times D^4_{\gamma_{>j}\cdot b_j}$, however there is one slight subtle point.  Recall the splitting $\Gamma=\Gamma_j\oplus \Gamma/\Gamma_j$ is a splitting of sets however and not groups, so to understand the extension of the action of $\Gamma_j$ to $\Gamma$ let us identify $b_j\in \Gamma/\Gamma_j$ as
\begin{align}
	b_j=\prod_{i>j}\gamma_i^{b_{ij}}\, ,
\end{align}
where as usual $0\leq b_{ij}<k_i$.  Consider the product (in $\Gamma$)
\begin{align}
	\gamma_{>j}\cdot b_j = \prod_{i> j}\gamma_i^{a_i}\cdot\gamma_i^{b_{ij}} = \gamma_{j}^{c_j}\cdot \prod_{i> j}\gamma_i^{c_i} \in \Gamma\, ,
\end{align}
where $0\leq c_i<k_i$.  Note that $\prod_{i> j}\gamma_i^{c_i}\in \Gamma/\Gamma_j$ is the natural product of $\gamma_{>j}$ and $b_j$ in $\Gamma/\Gamma_j$, and either $c_j=0$ or $c_j=1$.  Then the action of $\gamma$ on $S^3\times D^4_{b_j}$ is given by the composition of the action $\gamma_j^{c_j}\gamma_{\leq j}:S^3\times D^4_{b_j}\to S^3\times D^4_{b_j}$ and the identity map $S^3\times D^4_{b_j}\to S^3\times D^4_{\gamma_{>j}\cdot b_j}$.  The gluing maps in the previous subsection were built precisely to make this extend to a global action of $\Gamma$.\\

Now for each directed edge $E_{ab_j}$ let us remove the corresponding ball $D_{ab_j}\subseteq D_{b_j}$ as in the last subsection:
\begin{align}
	\bigcup_j\bigcup_{V_j} S^3\times \big(D^4_{b_j}\setminus \bigcup_{E_{a b_j}} D^4_{a b_j}\big)\, .
\end{align}
Observe that the action of $\Gamma$ restricts to an action of the above.  Finally let us observe that for each directed edge $E_{ab_j}$ we have defined the corresponding gluing maps $\phi_{ab_j}:\partial(S^3\times D^4_{a})\to \partial(S^3\times D^4_{ab_j})$ which were built precisely to commute with the above action of $\Gamma$.  Thus we arrive at our end space
\begin{align}\label{eq:tildeM}
	\tilde M \equiv \Bigg(\bigcup_j\bigcup_{V_j} S^3\times \big(D^4_{b_j}\setminus \bigcup_{E_{a b_j}} D^4_{a b_j}\big)\Bigg){ \Bigg/_{\{\phi_{ab_j}\in E\}}}\, , 
\end{align}
together with its free action by $\Gamma$.  In \eqref{eq:tildeM} it is understood that the boundary components $\partial(S^3\times D^4_{a})$ and $\partial(S^3\times D^4_{ab_j})$ are identified according to the directed edges $E_{ab_j}\in E$ and via the diffeomorphisms $\phi_{ab_j}$.  We refer to Section \ref{s:outline_milnor:inductive} for an alternative (but equivalent) approach to the definition of the total space $\tilde{M}$ which has a more geometric flavor.\\

\vspace{.3cm}

\subsection{Geometric Outline of $(\tilde M,\tilde p,\Gamma)$}\label{ss:outline_milnor:geometry_milnor}

We described in the previous subsection the topological construction of the universal cover $\tilde M$ from Theorem \ref{t:main_milnor} together with its free action by $\Gamma$.  In this subsection we want to understand the broad geometry of $\tilde M$.  We will mostly concern ourselves with a rough Gromov Hausdorff picture of what is happening, with only some mild comments toward the finer geometric and topological points.  In the next subsection we will introduce a precise inductive construction that will put the pictures of this subsection and the last together more comprehensively.  The geometric viewpoint will have a different flavor than the topological construction, as we will focus ourselves more locally as we move up in scale.  This will also be the convenient viewpoint for the inductive construction in Section \ref{s:outline_milnor:inductive}.\\

For a chosen basepoint $\tilde p\in \tilde M$ let us look at the ball $B_r(\tilde p)$ and consider the local group $\Gamma_r \equiv \langle \gamma\in \Gamma: d,\tilde p, \gamma\cdot \tilde p)\leq r\rangle$ generated by those actions which move $\tilde p$ at most $r>0$.  The local groups $\Gamma_r\subseteq \Gamma$ will then necessarily be monotone increasing, and there will be discrete radii $r_j$ at which the local group jumps.  We will have for $r_j\leq r<r_{j+1}$ that $\Gamma_{r} = \Gamma_j \leq \dQ/\dZ\subseteq S^1$ as in \eqref{e:outline_milnor:Gamma}.  In particular, at scale $r_j$ we will add one new generator $\Gamma_j = \langle\gamma_j,\Gamma_{j-1}\rangle$ to the local group.  As usual we will denote by $k_j$ the minimal integer for which $\gamma_j^{k_j}=\gamma_{j-1}\in \Gamma_{j-1}$ becomes the generator of $\Gamma_{j-1}$.  In this way the local group is always generated by a single element, and what is happening at scale $r_j$ is that this element is changing.\\

\subsubsection{Geometry at Scale $r_j$} Let us then roughly describe what $\tilde M$ looks like on the scales $r_j$, and then next we will even more roughly describe what happens to $\tilde M$ between scales $r_j$ and $r_{j+1}$.  At each scale $r_j$ the manifold $\tilde M$ will be Gromov-Hausdorff close to a ball in $S^3\times \dR^4$.  Indeed the space will be mostly diffeomorphic
and nearly isometric to $S^3\times \dR^4$ at scale $r_j$, however as in the gluing construction of Section \ref{ss_gluingtopo} there will be $k_j$ small balls around the local orbit $\Gamma_j\cdot \tilde p$ which will contain a good deal of topology at smaller scales.  It is worth pointing out that for $j$ large the sphere factor $S^3$ will have scale invariantly decreasing radius, so that from a Gromov-Hausdorff point of view the space is looking increasingly like $\dR^4$\\  

Note that there is a $T^2=S^1\times S^1$ action on $S^3\times \dR^4$.  The first $S^1$ acts freely on the $S^3$ factor by Hopf rotating.  The second $S^1$ acts on the $\dR^4=C(S^3)$ factor by Hopf rotating the unit sphere.  For $(a,b)\in \dZ\times \dZ$ there is an induced $S^1$ action on $S^3\times \dR^4$ through the homomorphic embedding $S^1\to S^1\times S^1$ given by $\theta\mapsto (a\theta, b\theta)$.  That is, the $(a,b)$-action of $S^1$ will Hopf rotate the spheres of $\dR^4=C(S^3)$ at speed $b$ and will Hopf rotate $S^3$ at speed $a$.  Note that if $a$ and $b$ are coprime then this is a free action.  The size of the $3$-sphere will be growing, but go to zero relative to $r_j$, and so from a pure Gromov-Hausdorff point of view the space will be close to $\dR^4$ at the scales $r_j$. \\

Now on the scale $r_j$ the action of the generator $\gamma_j\in \Gamma_j$ will look like a rotation of the $\dR^4$ factor by $2\pi/k_j$, and a Hopf rotation of the $S^3$ factor by $2\pi/|\gamma_j|=2\pi/(k_0\cdots k_j)$.  If we view $\Gamma_j\leq S^1$ then the action of $\Gamma_j$ is the one induced by the $(1,k_{\le j-1})$-$S^1$ action as above, where we recall that we set $k_{\le j}\equiv k_0\cdot k_1 \cdots k_j$.  Observe that $\Gamma_{j-1}$ is generated by $\gamma_{j-1}=\gamma_j^{k_j}$, and therefore it looks like a rotation of purely the $S^3$ factor.  The basepoint $\tilde p$ should {\it not} be viewed as the center of the rotation of $\gamma_j$ in $\dR^4$.  The center of the rotation $\approx S^3\times\{0\}$ will be a central $3$-sphere. 
The point $\tilde p$ should be viewed as a point of distance roughly $k_jr_j$ from the center of this rotation.  In this way $d(\tilde p, \gamma_j\cdot \tilde p)= r_j$ and the size of the orbit of the $\Gamma_j$ action is roughly $k_j r_j$.\\

\subsubsection{Geometry between Scales $r_j$ and $r_{j+1}$}\label{ss:geometryrjrj1}  
We have described that at scale $r_j$ the space is close to $S^3\times \dR^4$ and the local group $\Gamma_j$ looks primarily like a rotation of the $\dR^4$ factor.  Let us now discuss very roughly what happens between scales $r_j$ and $r_{j+1}$.  Observe that for the picture of the last paragraphs to hold, something substantial must have happened.  Indeed, let us consider the group $\Gamma_j$ at scales $r_j$ and $r_{j+1}$.  At both of these scales the space looks like $S^3\times \dR^4$, however the action of $\Gamma_j$ on the bottom $r_j$-scale rotates both factors, while on the top $r_{j+1}$-scale it rotates only the second factor.  In particular the action of the generator $\gamma_j$, which looks mostly like a rotation of $\dR^4$ on the bottom scale, has slid in to become just a rotation of $S^3$ on the top scale.  \\

The topological mechanism for this twisting was described in Section \ref{ss_gluingtopo}, namely we needed to glue these two copies of $\dR^4\times S^3\approx D^4\times S^3$ together by a boundary map $\phi_j:S^3\times S^3\to S^3\times S^3$ which is homotopically nontrivial, and which commutes with the action by untwisting
\begin{align}
	\phi_j\big(\theta\cdot_{(1,k_{\le j-1})}(g_1,g_2)\big)=\theta\cdot_{(1,0)}\phi_j(g_1,g_2)\, .
\end{align}
Let us give a different viewpoint here which is geometrically convenient.  Between scales $r_j$ and $r_{j+1}$ our space will be diffeomorphic to an annulus in $S^3\times \dR^4$, or equivalently diffeomorphic to an annulus $A_{r_j,r_{j+1}}(0)\subseteq C(S^3\times S^3)$.  Very roughly, we can view the metric on this annulus as $dr^2+r^2 g_r$, where $g_r$ is a family of metrics on $S^3\times S^3$.  We know that the top and bottom scales are very close to $S^3\times \dR^4$, and so to first approximation we can say that $g_{r_j}$ and $g_{r_{j+1}}$ are isometrically very close to a product of two spheres $S^3_\delta\times S^3_1$, where the subscript denotes the radius and we are viewing $0<\delta<<1$.  Note that $C(S^3_1)=\dR^4$ and the small sphere is playing the role of $S^3$ cross factor.  However we understand from Section \ref{s:equivariant_Ricci} that these isometries are very different, and indeed not even isotopic to one another.  That is, even if $g_{r_{j+1}}$ and $g_{r_j}$ are isometric, as tensors we do not have $g_{r_{j+1}}\approx g_{r_j}$ but instead have $g_{r_{j+1}}\approx \phi_j^*g_{r_j}$, where $\phi_j:S^3\times S^3\to S^3\times S^3$ is a diffeomorphism as above.  So although the geometry at scales $r_j$ and $r_{j+1}$ begins and ends at the same point, we should be interpreting these two copies of $S^3\times \dR^4$ very differently.  Step 2 of Section \ref{s:outline_milnor:inductive} will discuss this in greater detail, and see Section \ref{s:step2_proof} for the precise discussion.  Note for precision sake that the metric is not a cone metric at the beginning and end, and that the two product $3$-spheres at the top and bottom scales will be of very different size. \\

Geometrically, the rough description of the geometry of $dr^2+r^2 g_r$ on the region between scales $r_j$ and $r_{j+1}$ is as follows.  The metric $g_r$ on $S^3\times S^3$ begins at $r_j$ so that the space is isometrically very close to $S^3\times \dR^4=S^3\times C(S^3_1)$.  As the first sphere is very small, and indeed how small will be scale invariantly going to zero as $j$ increases, this is Gromov-Hausdorff close to $\dR^4$.  Then slowly in $r$ the metric will shrink the second $S^3$ factor, so that geometrically our space becomes a ray $\dR^+$.  Now the complicated twisting of the cross sections from Section \ref{s:equivariant_Ricci} will take place, however geometrically the space will look roughly like a ray this whole time.  Finally the metric will reexpand to $g_{r_{j+1}}$, which is again isometrically very close to $S^3\times \dR^4$, albeit a very different copy of $S^3\times \dR^4$.\\

\subsubsection{Transitioning from Scale $r_j$ to Scale $r_{j+1}$} After the action has been untwisted between scales $r_j$ and $r_{j+1}$, let us remark that there is an additional challenge when the next generator $\gamma_{j+1}$ enters the picture.  At scale $r_{j+1}$ we again look close to $S^3\times \dR^3$, however we then suddenly see $k_{j+1}$ copies of our original space appear.  Geometrically this will occur on very (scale-invariantly) small balls, and so from a broad geometrical viewpoint the space will still look Gromov Hausdorff close to $\dR^4$.  These new copies will be identified by the $\gamma_{j+1}$ action, as our local group has jumped.  Step 3 in Section \ref{s:outline_milnor:inductive} will deal with this issue with more care, and see Section \ref{s:Step3_proof} for the precise discussion. \\

\subsubsection{\bf Tangent Cones at Infinity of $\tilde M$ and $M$}\label{ss:tangentcones}

Let us consider a sequence of radii $s_j\to \infty$ and understand the limits of $(s_j^{-1}\tilde M,p,\Gamma)$ and $(s_j^{-1}M,p)$.  After passing to subsequences (and reindexing) we can break ourselves down into various cases depending on how $s_j$ compares to our naturally defined scales $r_j$ from before.

\subsubsection{The scales $s_j=r_j$}

Let us begin with the base case of understanding the sequence  $(r_j^{-1}\tilde M, p, \Gamma)$ on the universal cover.  We have determined that $\tilde M$ looks very close to $S^3\times \dR^4$ at these scales with (scale invariantly) shrinking sphere factor.  In particular, we have that geometrically the tangent cone at infinity along this sequence gives $r_{j}^{-1}\tilde M\to \dR^4$.  The action of $\gamma_j$ at scale $r_j$ is visible as a rotation by angle $2\pi/k_j$ of the $\dR^4$ factor with respect to a basepoint distance $k_j$ away. Therefore to understand the equivariant limit we need to break ourselves into two cases.  Namely, after passing to subsequences either $k_j$ converges or not.\\

\subsubsection{The scales $s_j=r_j$ with $k_j\to k<\infty$}

In this case the action of $\gamma_j$ looks like a rotation with respect to a point distance $kr_j$ away from $p$, and so we have that $(r_j^{-1}\tilde M, p, \Gamma)\to (\dR^4,p_\infty,\dZ_k)$ where $\dZ_k$ is acting by rotation around the origin and $p_\infty$ is a point distance $k$ from the origin.  We get that the quotient space 
\begin{align}
	(r_j^{-1}M,p)\to (C(S^3_1/\dZ_k),p_\infty)
\end{align}
limits to a cone over a lens space.  The basepoint $p_\infty$ of this limit is again a point distance $k$ from the cone point.

\subsubsection{The scales $s_j=r_j$ with $k_j\to\infty$}

In this case the action of $\gamma_j$ is looking increasingly like a translation by $\dZ$, and we get that $(r_j^{-1}\tilde M, p, \Gamma)\to (\dR^4,0,\dZ)$ where $\dZ$ acts by unit translation.  The quotient space in this case limits 
\begin{align}
	r_j^{-1}M\to \dR^3\times S^1\, .
\end{align}

\subsubsection{The scales $r_j<s_j<<k_jr_j$ with $k_j\to \infty$}.  In the case that $k_j\to k$ remains bounded there is no distinction between this case and the last.  Therefore, we are only concerned with the case where we have some subsequence for which $k_j\to \infty$.  In this situation note with $\frac{s_j}{r_j},\frac{k_jr_j}{s_j}\to \infty$ that our $\dZ$ action is looking increasingly like an $\dR$ action.  Our limit in this case becomes $(s_j^{-1}\tilde M, p, \Gamma)\to (\dR^4,0,\dR)$, where $\dR$ is acting by translation.  Our quotient space is therefore limiting
\begin{align}
	s_j^{-1}M\to \dR^3\, .
\end{align}

\subsubsection{The scales $s_j\approx k_jr_j$ when $k_j\to \infty$}  Note the action of $\gamma_j$ at these scales looks like a rotation by angle $2\pi/k_j$.  In particular, we get that $(s_j^{-1}\tilde M, p, \Gamma)\to (\dR^4,p_\infty,S^1)$, where $S^1$ is a rotation around the origin.  Our basepoint is now roughly distance $1$ from the center of the rotation.  In particular our quotient limit is given by
\begin{align}
	(r_j^{-1}M,p)\to (C(S^2_{1/2}),p_\infty)\, .
\end{align}

\subsubsection{The scales $k_jr_j<<s_j << r_{j+1}$ when $k_j\to k<\infty$}. We discussed that at scale $s_j\approx k_jr_j$ we have $s_j^{-1}\tilde M$ looks like $\dR^4=C(S^3_1)$.  As $\frac{s_j}{k_j r_j}$ increases our cross section sphere $S^3_s$ begins to decrease in radius until it looks like a half ray.  Therefore we get the possible limits $(s_j^{-1}\tilde M,p,\Gamma)\to (C(S^3_s),0,\dZ_k)$ for all $0\leq s\leq 1$.  In the case when $\frac{s_j}{k_jr_j}$ becomes sufficiently large we get that the limit is a half ray with the trivial action.  Our quotient limits in this range are therefore
\begin{align}
	(s_j^{-1}M,p)\to (C(S^3_s/\dZ_k),p_\infty)\, ,
\end{align}
for all $0\leq s\leq 1$.\\

\subsubsection{The scales $s_j\to r_{j+1}$}
As the scale $s_j$ continues to increase to $r_{j+1}$, we have that the half ray reopens up so that we again have $s_j^{-1}\tilde M\approx \dR^4$.  However, as it reopens the $\Gamma_j$ is now a trivial action.   As we approach scale $r_{j+1}$ a new $\gamma_{j+1}$ action appears and we repeat the above process. $\qed$\\

Let us make several quick observations about this process.  In the case $\Gamma=\dQ/\dZ$ we can choose $k_j$ so that every $k\in \dN$ appears infinitely often, see Example \ref{ex:Gamma=Q_all_integers}.  Consequently, all of the cones
\begin{align}
	M_\infty \equiv C(S^3_s/\dZ_{k})\, ,
\end{align}
appear as tangent cones at infinity for all $s\in [0,1]$ and $k\in \dN$.  Geometrically, when we start at a scale for which the space $M$ looks like a cone over a lens space $C(S^3/\dZ_{k})$, then as the scale increases the cross section shrinks so that the space looks like a half ray.  As the scale continues to increase the space expands to again look like a cone over a lens space $C(S^3/\dZ_{k'})$.  However, when it reopens it may appear to be a different lens space. \\  

The last point to remark on is that though every tangent cone at infinity is a metric cone, the pointed limit does not always have the cone point as the base point.  This is in agreement with \cite{Sormanilinear}, where we understand some tangent cones at infinity need to not be polar {\it with respect to} the base point.

\vspace{.5cm}

\section{Inductive Construction for Theorem \ref{t:main_milnor}}\label{s:outline_milnor:inductive}

Let us now describe our construction for Theorem \ref{t:main_milnor} in more technical detail.  The proof will be set up in an inductive fashion, where we will build a sequence of pointed manifolds $(M_j,p_j,\Gamma_j)$ with $\Ric_j\geq 0$ together with free uniformly discrete isometric actions by $\Gamma_j$.  This Section will begin with a description of the main properties of our inductive sequence $M_j$, together with how one proves Theorem \ref{t:main_milnor} once this sequence has been constructed.  The induction criteria will be such that building $(\tilde M,\Gamma)$ from the inductive sequence will be relatively straightforward.\\  

The remainder of this Section will then focus on proving the induction, namely on how to construct $M_{j+1}$ from $M_j$ in order to complete the induction proof.  The construction will boil down to three major steps, and in each step we will state one of our three main inductive Propositions.  These Propositions will be proved in remaining Sections of the paper, and thus the proof of Theorem \ref{t:main_milnor} will be complete in this Section modulo these main Propositions.\\  

Let us now set the stage for a precise statement of our induction criteria.  Recall that we have chosen as in \eqref{e:outline_milnor:Gamma} a sequence of finitely generated subgroups $\Gamma_j\leq \Gamma$ with $\Gamma_j = \langle \gamma_j,\Gamma_{j-1}\rangle$ which are all generated by a single element $\gamma_j$ such that $\gamma_j^{k_j}=\gamma_{j-1}$ .  From the point of view of the topological construction of Section \ref{s:outline_milnor_construction} we can view the sequence $(M_j, p_j,\Gamma_j)$ as the manifold obtained under the construction tree with $\Gamma = \Gamma_j$.\\

Our geometric construction will be based on a sequence of parameters $\epsilon_j\to 0$ and $\delta_j\to 0$.  We may begin by choosing any sequence $\epsilon_j\to 0$.  Indeed any sequence of constants $\epsilon_j<1$ will do, but in our description of the tangent cones at infinity of $\tilde M$ in Section \ref{ss:tangentcones} we have used that these constants tend to zero, which gives a slightly cleaner picture. Let $\delta_1<<1$ also be any small constant, the remaining $\delta_j$ will be chosen based on applications of our Inductive Propositions. We shall adopt the notation $A_{s_1,s_2}(p)$ to denote the annuls $B_{s_2}(p)\setminus B_{s_1}(p)$ for any $0\le s_1<s_2\le \infty$.\\

Our sequence $(M_j,p_j,\Gamma_j)$ will inductively be assumed to satisfy:\\

\begin{enumerate}
	\item[(I1)] There exists a free isometric action by $\Gamma_j$ on $M_j$ with $r_j\equiv d(p_j,\gamma_j\cdot p_j)$ and $\frac{r_j}{k_{j-1} r_{j-1}} >> 1 $.
	\item[(I2)] There exists an isometry $\Phi_j:U_j\subseteq M_j\to M_{j+1}$ with $B_{10 k_j r_j}(p_j)\subseteq U_j\subseteq B_{10^3 k_j r_j}(p_j)$ with $\Phi_j(p_j)=p_{j+1}$, where $U_j$ is $\Gamma_j$ invariant with $\Phi_j(x\cdot \gamma)=\Phi_j(x)\cdot \gamma$ for all $\gamma\in \Gamma_j\leq \Gamma_{j+1}$ .
	\item[(I3)] $M_j\setminus U_j$ is isometric to $S^3_{\delta_j r_j}\times A_{10^{2}k_jr_j,\infty}(0)\subseteq S^3_{\delta_j r_j}\times C(S^3_{1-\epsilon_j})$ \footnote{Observe that this is isometrically very close to $S^3\times \dR^4$.  Indeed, in our setup $U_j$ itself is very Gromov-Hausdorff close to $S^3\times \dR^4$.}.  The action of $\gamma_j$ in this domain rotates the cross section $S^3_{1-\epsilon_j}$ of the cone factor by $2\pi/k_j$ and the $S^3_{\delta_j r_j}$ factor by $2\pi/|\gamma_j|=2\pi/(k_0k_1\cdots k_j)$.
\end{enumerate} 
\vspace{.2cm}
\begin{remark}
	It follows from (I3) that the orbit of the action of $\Gamma_j$ has diameter roughly $k_j r_j$.
\end{remark}
\begin{remark}
	It will be clear from the construction that $\frac{r_{j+1}}{k_{j} r_{j}} \to \infty$.  That is, the scale of the action of the next generator $\gamma_{j+1}$ relative to the orbit of the previous generator $\gamma_j$ is tending to infinity.
\end{remark}

Before discussing more carefully the structure of the spaces $M_j$ above, let us quickly see that if such an inductive sequence as above can be built, then we are done.  Indeed, consider first the $\Gamma_i$-equivariant isometries $\Phi_{ji}=\Phi_j\circ\cdots\circ\Phi_i:U_i\to U_{ji}\equiv \Phi_{ji}(U_i)\subseteq M_j$.  We can take an abstract equivariant pointed Gromov-Hausdorff limit of the sequence $(M_j,p_j,\Gamma_j)$. However the setup is such that we can also simply define the direct limit
\begin{align}
	\tilde M\equiv\big\{ (x_j,x_{j+1},\ldots): x_{k+1}=\Phi_k(x_k) \text{ for all }k\geq j\big\}/\sim\, ,
\end{align}
where there is an equivalence relation $(x_{j},x_{j+1},\ldots)\sim (y_{j'},y_{j'+1},\ldots)$ if there exists $k\geq \max\{j,j'\}$ such that $x_k=y_k$.   By the equivariance of the isometries $\Phi_{i}$ we have that $\Gamma_j$ naturally acts on all sequences $(x_k,x_{k+1},\ldots)$ with $k\geq j$.  In particular there is an induced action of $\Gamma$ on $\tilde M$.  Note that $U_j\subseteq M_j$ all embed isometrically into $\tilde M$ and exhaust $\tilde M$, and the restriction of the $\Gamma_j$ action to $U_j\subseteq \tilde M$ is the expected action.   Thus $\tilde M$ is a smooth Riemannian manifold with $\Ric\geq 0$ and a free discrete isometric action by $\Gamma$, as claimed.

\vspace{.2cm}
\subsection{The Steps of the Inductive Construction}

We will break down this inductive construction into three steps.  Each will involve a Proposition which will form the main constructive ingredient in the step.  Our goal in this subsection is then to discuss these steps and state the Propositions.  We will then see how to finish the induction given these results.  Future sections will then be dedicated to proving each of these Propositions individually.\\

The first step will build our background model space $\cB(\epsilon,\delta)\approx S^3\times \dR^4$.  It will form the basis of both our base step of the induction, and also the underlying space for which previous induction manifolds $M_j$ will be glued into in order to form $M_{j+1}$.  From the point of view of the topological construction of Section \ref{s:outline_milnor_construction}, there will be one copy of a background model space per vertex in our construction tree.  The construction of the model space $\cB(\epsilon,\delta)$ is actually a fairly standard one, but it will help with the exposition to isolate it and discuss the role it plays.\\

The second step will deal with the action twisting described in Section \ref{ss:geometryrjrj1} .  Each $M_j$ looks like $S^3\times \dR^4$ at infinity with the action $\Gamma_j$ induced by the $(1,k_{\le j-1})$-Hopf $S^1$ action.  The first step in building $M_{j+1}$ is to equivariantly twist $M_j$ to a new manifold $\hat M_j$, so that after our twisting $\hat M_j$ again looks like $S^3\times \dR^4$ at infinity but this time the $\Gamma_j$ action is induced by the $(1,0)$-Hopf action.\\

The third step of the inductive construction is to take our twisted $\hat M_j$ and glue in $k_{j+1}$ copies into a new base manifold $\cB_{j+1}$.  The gluing is such that we have now extended the $\Gamma_j$ action on $\hat M_j$ to a $\Gamma_{j+1}=\langle \gamma_{j+1},\Gamma_j\rangle$ action on $M_{j+1}$ in the appropriate fashion.  From the point of view of the topological construction, there will be one gluing per directed edge in our construction tree.

\vspace{.25cm}

\subsubsection{\bf Step 1: The Background Model Space $\cB(\epsilon,\delta)$}

Our construction will begin by building a background manifold $\cB(\epsilon,\delta)$.  The space will both play the role of base step in the inductive construction, and additionally when we move from $M_j$ to $M_{j+1}$ the basis for our construction will be to glue in $k_{j+1}$ copies of $M_j$ into the background space $\cB_{j+1}$.  From the point of view of our construction tree in Section \ref{ss:outline_milnor:topology_milnor}, there will eventually be one copy of $\cB_j$ per vertex $v_j\in V_j = \Gamma/\Gamma_j$.\\ 

The construction of $\cB(\epsilon,\delta)$ is relatively straightforward, we will simply take $S^3\times \dR^4$ and slightly curve the $\dR^4$ factor in order to give it a slight cone angle.  The precise setup is the following:

\begin{proposition}[Step 1: The Model Space]\label{p:step1}
	For each $\delta>0$ and $1>\epsilon>0$ , there exists a smooth manifold $\cB^7=\cB(\epsilon,\delta)$ such that the following hold:
	\begin{enumerate}
		\item $(\cB^7,g_B,p)$ is a complete Riemannian manifold with $\Ric\geq 0$, it is diffeomorphic to $S^3\times \dR^4$.
		\item There exists $B_{10^{-3}}(p)\subseteq U\subseteq B_{10^{-1}}(p)$ such that $\cB\setminus U$ is isometric to $S^3_{\delta}\times A_{10^{-2},\infty}(0)\subseteq S^3_{\delta}\times C(S^3_{1-\epsilon})$ 
		\item There is an isometric $T^2=S^1\times S^1$ action on $\cB$ for which on $\cB\setminus U\approx S^3_{\delta}\times C(S^3_{1-\epsilon})$ the first $S^1$ acts on the $S^3_{\delta}$ factor by a globally free (left) Hopf rotation and the second $S^1$ acts on the cross sections $S^3_{1-\epsilon}$ of the cone factor by (left) Hopf rotation. 
		\item The $S^1$-action induced by the homomorphic embedding $S^1\ni \theta\mapsto (a\theta,b\theta)\in T^2$ is free whenever $(a,b)\in \dZ\times \dZ$ are coprime and $a\neq 0$.
	\end{enumerate}
\end{proposition}
\begin{remark}
	Thus for each $(a,b)\in \dZ\times \dZ$ we have the induced $(a,b)$-$S^1$ action given by the homomorphic embedding $S^1\ni \theta\mapsto (a\theta,b\theta)\in T^2$.
\end{remark}

{\bf Base Step: } Let us then define the base step of our induction as $M_1=\cB(\epsilon_1,\delta_1)$ as above.  We will equip $M_1$ with the the isometric group action of $\Gamma_1\leq S^1$, which is induced by the $(1,k_{0})$-action as above.  In particular, on $S^3_{\delta_1}\times S^3_{1-\epsilon_1}$ we have that the generator $\gamma_1$ will act by Hopf rotating $S^3_{\delta_1}$ by $2\pi/|\gamma_1| = 2\pi/(k_1k_0)$ and by Hopf rotating the cross section of $C(S^3_{1-\epsilon_1})$ by $2\pi/k_1=2\pi k_0/|\gamma_1|$.\\

\vspace{.3cm}

\subsubsection{\bf Step 2:  The Equivariant Mapping Class Group and Twisting the Geometry of $M_j$ at Infinity}\label{ss_Step2description}

By condition (I3) of the induction we know that outside some compact set $U_j$ our space $M_j\setminus U_j$ is isometric to $S^3_{\delta_j r_j}\times A_{10^{2}k_jr_j,\infty}(0)\subseteq S^3_{\delta_j r_j}\times C(S^3_{1-\epsilon_j})\approx S^3\times \dR^4$.  Further, we understand that in this region the action of the generator $\gamma_j\in \Gamma_j$ looks primarily like a rotation of the $\dR^4$ factor.  More precisely, it rotates the $\dR^4$ factor by $2\pi/k_j$ and it rotates the $S^3_{\delta_j r_j}$ factor by the much smaller $2\pi/|\gamma_j| = 2\pi/(k_0\cdots k_j)$. \\

In Step 3 we will be gluing $k_{j+1}$ copies of $M_j$ into a model space $\cB_{j+1}$, and in the gluing region we will again have that $\cB_{j+1}\approx S^3\times \dR^4$.  However, the action of $\Gamma_j$ on $\cB_{j+1}$ will look like a rotation of just the $S^3$ factor without any rotational bit on the $\dR^4$ factor.  Thus to accomplish the gluing we will need to modify $M_j$ at infinity into a new space $\hat M_j$, which will again look close to $S^3\times \dR^4$ but for which the action of $\gamma_j$ is now purely a rotation of the $S^3$ factor.\\

In order to address this problem let us first consider $S^3\times S^3$ with the standard metric $g_{S^3\times S^3}$, and let us recall that if $(a,b)\in \dZ\times \dZ$ then we have the $S^1$-isometric action $\cdot_{(a,b)}:S^3\times S^3\times S^1\to S^3\times S^3$ which acts by $a$ times the (left) Hopf rotation on the first $S^3$ and $b$ times the (left) Hopf rotation on the second $S^3$.  The following will provide for us how the cross sections of our new space $\hat M_j$ will be twisting.  It will be proved in Section \ref{s:equivariant_Ricci}:\\

\begin{theorem}[Equivariant Mapping Class Group on $S^3\times S^3$]\label{t:equivariant_mapping_class_S3xS2}
	Let $g_0=g_{S^3\times S^3}$ be the standard metric on $S^3\times S^3$, and let $k\in \dZ$.  Then there exist a diffeomorphism $\phi:S^3\times S^3\to S^3\times S^3$ and a family of metrics $(S^3\times S^3,g_t)$ such that
	\begin{enumerate}
		\item $\Ric_t>0$ for all $t\in [0,1]$
		\item The $S^1$-action $\cdot_{(1,k)}$ on $S^3\times S^3$ is an isometric action for all $g_t$.
		\item $g_1 = \phi^*g_0$ with $\phi\big(\theta\cdot_{(1,k)}(s_1,s_2)\big)= \theta\cdot_{(1,0)}\phi(s_1,s_2)$ .
	\end{enumerate}
\end{theorem}
\begin{remark}
	The diffeomorphism $\phi:S^3\times S^3\to S^3\times S^3$ will represent a nontrivial element of the mapping class group.  
\end{remark}
\begin{remark}	
	In Section \ref{s:mappingS3S3} we will discuss how to connect $g_{S^3\times S^3}$ and $\phi^*g_{S^3\times S^3}$ for {\it any} element of the mapping class group $[\phi]\in \pi_0\text{Diff}(S^3\times S^3)$.  However for an arbitrary element of the mapping class group we cannot necessarily keep track of the behavior of an isometric action. 
\end{remark}

\vspace{.2cm}

The above tells us that we can find an $S^1$-invariant family of metrics with positive Ricci curvature which (from an isometric point of view) start and end at the classical $S^3\times S^3$, however the beginning and ending $S^1$ actions are quite distinct.  Our main use of the above will be to build the following neck region, which will be used to alter $M_j$ to $\hat M_j$:

\begin{proposition}[Step 2: Twisting the Action]\label{p:step2}
	Let $\epsilon,\hat \epsilon,\delta>0$ with $k\in \dZ$.  Then there exist $\hat \delta(\epsilon,\hat \epsilon,\delta,k)>0$ and $R(\epsilon,\hat \epsilon,\delta,k)>1$ and a metric space $X$ with an isometric and free $S^1$ action such that
	\begin{enumerate}
		\item $X$ is smooth away from a single three sphere $S^3_\delta\times \{p\}\in X$ with $\Ric _X\geq 0$.
		\item There exists $B_{10^{-3}}(p)\subseteq U\subseteq B_{10^{-1}}(p)\subseteq X$ which is isometric to $S^3_\delta\times B_{10^{-2}}(0)\subseteq S^3_\delta\times C(S^3_{1-\epsilon})$ , and under this isometry the $S^1$ action on $U$ identifies with the $(1,k)$-Hopf action. 
		\item There exists $B_{10^{-1}R}(p)\subseteq \hat U\subseteq B_{10R}(p)\subseteq X$ s.t. $X\setminus \hat U$ is isometric to $S^3_{\hat \delta R}\times A_{R,\infty}(0)\subseteq S^3_{\hat \delta R}\times C(S^3_{1-\hat \epsilon})$, and under this isometry the $S^1$ action on $X\setminus \hat U$ identifies with the $(1,0)$-Hopf action. 
	\end{enumerate}
\end{proposition}

{\bf Constructing $\hat M_j$:} Before moving on to Step 3, let us see how the above will be used as part of our induction process. Thus let us assume we have constructed $M_j$ as in (I1)-(I3) with sphere radius $\delta_j$.  Recall by (I3) that outside of a compact subset we have that $M_j$ is isometric to $S^3_{\delta_j r_j}\times C(S^3_{1-\epsilon_j})$, and the action of $\gamma_j$ Hopf rotates the $S^3_{1-\epsilon_j}$ factor by $2\pi/k_j$ and the $S^3_{\delta_j r_j}$ factor by $2\pi/|\gamma_j|$.  Observe that if we consider the $(1,k_{\le j-1})$-Hopf $S^1$-action on $S^3_{\delta_j r_j}\times C(S^3_{1-\epsilon_j})$, then $\Gamma_j\subseteq S^1$ can be viewed as a subaction.\\

Now with any $\hat\epsilon_j>0$, the precise constant will be chosen later, we have for $R_j=R_j(\epsilon_j,\hat\epsilon_j,\delta_j,k_{\le j-1})$ and $\hat\delta_j=\hat\delta_j(\epsilon_j,\hat\epsilon_j,\delta_j,k_{\le j-1})$ the existence of $X_j$ as in Proposition \ref{p:step2}, where we chose $k=k_{\le j-1}=k_0\cdot k_1\cdots k_{j-1}$ in the application of the Proposition.  We can rescale $X_j\to r_j X_j$ by $r_j$ so that it is isometric to $S^3_{\delta_j r_j}\times C(S^3_{1-\epsilon_j})$ on a region $U$ containing $B_{r_j}(p)$, and it is isometric to $S^3_{\hat \delta_{j}R_j r_j}\times C(S^3_{1-\hat \epsilon_{j}})$ on a region $X_j\setminus \hat U$ containing the annulus $A_{R_j r_j,\infty}(p)$ .  Further, there is a free isometric $S^1$ action on $X_j$ which looks like the $(1,k_{\le j-1})$ action on $U$ and the $(1,0)$ action on $X_j\setminus \hat U$.  In particular, by condition (2) in Proposition \ref{p:step2} and the inductive assumption (I3) there is an induced $\Gamma_j$ action on $X_j$ and an open annulus of $U\subseteq X_j$ which is equivariantly isometric to an open annulus in $M_j\setminus U_j$.\\

We can thus glue $X_j$ to $M_j$ in order to produce the space $\hat M_j$.  The space $\hat M_j$ is now isometric to $S^3_{\hat \delta_{j}R_j r_j}\times C(S^3_{1-\hat \epsilon_{j}})$ outside of some compact set $V_j\subseteq \hat M_j$, and the $\Gamma_j$ action is a pure Hopf rotation on the $S^3_{\hat \delta_{j}R_j r_j}$ factor on $\hat M_j\setminus V_j$.\\

\vspace{.3cm}

\subsubsection{\bf Step 3: Gluing Construction}\label{ss:step 3 outline}

The third step of the construction involves extending the action of $\Gamma_j$ to an action of $\Gamma_{j+1}$ in order to move from the manifold $M_j$ to the next step of the induction $M_{j+1}$.  This will occur by taking $k_{j+1}$ copies of the twisted space $\hat M_j$, constructed in the second step, and gluing them into a model space $\cB_{j+1}\approx \cB(\epsilon_{j+1},\delta_{j+1})$ constructed in the first step.  From the point of view of the topological construction in Section \ref{s:outline_milnor_construction}, there is one gluing per directed edge in our construction tree.\\

Recall that a model space $\cB(\epsilon,\delta)$ is isometric to an annulus in $S^3_{\delta}\times C(S^3_{1-\epsilon})$ outside of a compact set, and recall that the induction manifolds $\hat M$ are isometric to annuli in $S^3_{\delta}\times C(S^3_{1-\hat\epsilon})$ outside of a compact set.  We will therefore outline our gluing constructions purely in terms of annuli, which is where the gluing will take place.  If we can accomplish this with the correct behaviors, we can then glue our model space $\cB_{j+1}$ and inductive manifolds $\hat M_j$ directly into our glued space and finish the inductive construction of $M_{j+1}$.\\

Let us first outline the gluing strategy without worrying about smoothness or Ricci curvature.  We will end with Proposition \ref{p:step3}, which will state the end construction in a smooth Ricci preserving manner.  We describe this in some generality, with the understanding that we will be applying it as above afterwards.  So let $\cA' \equiv S^3_{\delta}\times C(S^3_{1-\epsilon})$ and let $\hat \cA = S^3_\delta\times B_1(0)\subseteq S^3_\delta\times C(S^3_{1-\hat\epsilon})$ with $\Gamma\leq S^1$ a finite group generated by a single element $\gamma$ whose order is divisible by $k$.  Let $\hat\Gamma$ be the group generated by $\hat\gamma\equiv\gamma^k$.  Consider the action of $\Gamma$ on $\cA'$ induced by the $(1,|\gamma|/k)$-Hopf action.  Thus $\gamma$ Hopf rotates the $S^3_{1-\epsilon}$ factor by $2\pi/k$ and the $S^3_\delta$ factor by $2\pi/|\gamma|$.  Let us also consider the action of $\hat\Gamma$ on $\hat \cA$ obtained by just rotating the $S^3_\delta$ factor by $2\pi/|\hat\gamma|$ . \\

Consider $k$ copies of the annulus {$\hat\cA^a \equiv \hat\cA \times \{a\} $} with $a=0,\ldots,k-1$, and note that $\partial\hat\cA^a = S^3_\delta\times S^3_{1-\hat\epsilon}$ isometrically.  Our goal is to glue in these $k$ copies into $\cA'$ such that there is an induced $\Gamma$ action on the glued space.  We will want that $\hat\Gamma$ restricts to the usual actions on both $\cA'$ and the glued copies of $\hat \cA$.  To be more precise let $x\in C(S^3_{1-\epsilon})$ be a point which is distance is $10^2k$ from the origin.  Let $x^a\in C(S^3_{1-\epsilon})$ with $a=0,\ldots, k-1$ be the $k$ points  obtained by Hopf rotating $x^0=x$ by $2\pi a/k$. \\

Consider each of the domains $S^3_\delta\times B_{1}(x^a)\subseteq \cA'$, and note that their boundaries are diffeomorphic (and nearly isometric) to $S^3_{\delta}\times S^3_{1}$.  Note that the $\hat \Gamma$ action restricts to actions on each of these domains, while the $\Gamma$ action simply restricts to an isometry between potentially different pairs of domains.  We will want to glue $\hat \cA^0, \ldots, \hat \cA^{k-1}$ into the space

\begin{align}
	\cA'\setminus \Big(\bigcup_a S^3_{\delta}\times B_{1}(x^a)\Big)\, .
\end{align}

In order to perform the gluing we need to define the gluing diffeomorphisms
\begin{align}
	\varphi^a: \partial \hat \cA^a\to S^3_{\delta}\times \partial B_{1}(x^a)\, .
\end{align} 

Recalling that $\partial \hat \cA^0=S^3_{\delta}\times S^3_{1}$ and $S^3_{\delta}\times\partial B_{1}(x)$ is nearly isometric to $S^3_\delta\times S^3_{1}$, let us first choose an almost isometry $\varphi^0:\partial \hat \cA^0\to S^3_\delta\times\partial B_{ 1}(x)$ which is the identity on the first sphere factor.  In particular, it follows that $\varphi^0$ commutes with the natural $\hat\Gamma$ actions on each of these spaces.  Let us then define $\varphi^a: \partial \hat \cA^a\to S^3_\delta\times\partial B_{1}(x^a)$ by

\begin{align}
	\varphi^a(y,a) = \gamma^a\cdot \varphi^0(y,0)\, , \quad y\in \hat \cA\, ,
\end{align}
for $a=0, \ldots, k-1$.
Note that we could naturally extend the above maps for any $a\in \dZ$.  However, we would have that $\varphi^{k}:\partial \hat \cA^0\to S^3_\delta\times\partial B_{1}(x^0)$ would not be the same mapping as $\varphi^0$.  Indeed, we see that $\varphi^{k}= \gamma^{k}\cdot \varphi^0 = \hat\gamma\cdot \varphi^0$.  To understand the implications of this consider the glued space
\begin{align}
	\tilde \cA \equiv \Big(\cA'\setminus \bigcup_a S^3_\delta\times B_{1}(x^a)\Big)\bigcup_{\varphi^a} \hat \cA^a\, ,
\end{align}
where we have plucked out the $k$ domains $S^3_\delta\times B_{1}(x^a)$ and plugged in the new annular regions $\hat \cA^a$.  The new space $\tilde \cA$ is still isometrically of the form $S^3_\delta\times C(S^3_{1-\epsilon})$ near the origin and infinity.  The effect of the gluing maps is that the $\hat\Gamma$ action on $\hat\cA$ extends to a $\Gamma$ action on $\tilde \cA$.  To understand this action, we need to describe the action of $\gamma$ on $\bigcup_{a} \hat \cA^a$. The latter is given by
\begin{equation}
	\begin{split}
	&\gamma\cdot (y,a) = (y,a+1)\, , \quad a=0, \ldots , k-2\, 
	\\
	&\gamma\cdot (y,k-1) = (\hat \gamma\cdot y,0)
	\end{split}
\end{equation}
for every $y\in \hat \cA$.  In particular, the action of $\hat\Gamma$ restricts to the expected action on each piece of the gluing.\\

The main Proposition of this step is to show that, up to some altering of constants, the above construction can be smoothed to preserve nonnegative Ricci curvature:

\begin{proposition}[Step 3:  Action Extension]\label{p:step3}
	Let $\epsilon,\epsilon',\delta>0$ with $0<\epsilon-\epsilon'\le \frac{1}{10^2 }\epsilon$, and let $\hat\Gamma\leq \dQ/\dZ\subseteq S^1$ be a finite subgroup with $\Gamma = \langle \gamma, \hat\Gamma\rangle$ such that $\hat\gamma\equiv \gamma^k$ is the generator of $\hat\Gamma$.  Then for $\hat\epsilon\leq \hat\epsilon(\epsilon,\epsilon')$ there exists a pointed space $(\tilde \cA,p)$, isometric to a smoooth Riemannian manifold with $\Ric\ge 0$ away from $k+1$ three spheres, with an isometric and free action by $\Gamma$ such that
	\begin{enumerate}
		\item There exists a $\Gamma$-invariant set $B_{10^{-1}}(p)\subseteq U'\subseteq B_{10}(p)$ which is isometric to $S^3_\delta\times B_{1}(0)\subseteq S^3_\delta\times C(S^3_{1-\epsilon'})$ and such that $\Gamma$ is induced by the $(1,|\gamma|/k)$-Hopf action on $S^3_\delta\times S^3_{1-\epsilon'}$ ,
		\item There exists a $\Gamma$-invariant set $B_{10^3 k}(p)\subseteq U\subseteq B_{10^5 k}(p)$ such that $\tilde \cA\setminus U$ is isometric to $A_{10^4k,\infty}(0)\times S^3_\delta\subseteq S^3_\delta\times C(S^3_{1-\epsilon})$ and such that $\Gamma$ is induced by the $(1,|\gamma|/k)$-Hopf action on $S^3_\delta\times S^3_{1-\epsilon}$ 
		\item There exist $\hat\Gamma$-invariant sets $S^3_\delta\times B_{2^{-1}}(x^a)\subseteq V^a\subseteq S^3_\delta\times B_{2}(x^a)$ with $d(S^3_\delta\times \{x^a\},S^3_\delta\times \{p\})=10^2 k$ which are isometric to $S^3_\delta\times B_{1}(0)\subseteq S^3_\delta\times C(S^3_{1-\hat\epsilon})$ and such that $\hat\Gamma$ is induced by the $(1,0)$-Hopf action on $S^3_\delta\times S^3_{1-\hat\epsilon}$ .
	\end{enumerate}
\end{proposition}
\begin{remark}
	It is important to observe that $\hat\epsilon(\epsilon,\epsilon')$ depends on the choices of $\epsilon$ and $\epsilon>\epsilon'$, however it does not depend on the choice of $\delta$.  
\end{remark}

{\bf Constructing $M_{j+1}$:}  Let us now apply Proposition \ref{p:step3} in order to finish the construction of $M_{j+1}$.  Let us take in the above $\Gamma=\Gamma_{j+1}$ and $\hat\Gamma=\Gamma_j$, and let us choose $\epsilon=\epsilon_{j+1}$ with $\epsilon'=\epsilon_{j+1}\cdot \frac{99}{100}$.  Recall that the construction of $\hat M_j$ in Section 2 depended on a choice of $\hat\epsilon_j$, which had not yet been fixed.  Let us now use Proposition \ref{p:step3} in order to choose $\hat\epsilon_j = \hat\epsilon_j(\epsilon_{j+1})$.  From this we now have from Proposition \ref{p:step2} a well defined $R_j$ and $\hat\delta_j$.  Finally let us now choose $\delta=\hat\delta_j$ in the application of Proposition \ref{p:step2}, so that we have built the space $\tilde \cA_j$.  After rescaling $\tilde \cA_j\to (R_jr_j) \tilde \cA_j$ by $R_jr_j$ observe that there exists $U\subseteq \tilde \cA_j$ which is isometric to $S^3_{\hat\delta_j R_j r_j}\times B_{R_jr_j}(0)\subseteq S^3_{\hat\delta_j R_j r_j}\times C(S^3_{1-\epsilon'})$, and also observe that the domains $V^a$ are isometric to $ S^3_{\hat\delta_j R_j r_j}\times B_{R_jr_j}(0)\subseteq S^3_{\hat\delta_j R_j r_j}\times C(S^3_{1-\hat\epsilon_{j}})$.\\

Finally, let us consider the base model $\cB_{j+1}=\cB(\epsilon', \hat\delta_j R_j r_j)$ from Proposition \ref{p:step1}.  We see we can glue it isometrically into $U\subseteq \tilde \cA_j$.  Additionally we can isometrically glue $\hat M_j$ into each $V^a\subseteq \tilde \cA_j$.  The resulting space is $M_{j+1}$.  If we define $p_{j+1}=p_j^0$ to be the basepoint of the copy of $M_j$ glued into $V^0$, then we can define $r_{j+1}\equiv d(p_{j+1},\gamma_{j+1}\cdot p_{j+1})$ and $\delta_{j+1}$ through the formula $\delta_{j+1} r_{j+1}\equiv \hat\delta_j R_j r_j$.  This completes the induction step of the construction.  In particular, we have proved Theorem \ref{t:main_milnor} up to the proofs of Propositions \ref{p:step1}, \ref{p:step2}, and \ref{p:step3}. $\qed$\\

\vspace{.5cm}


\section{Preliminaries}

In our constructions we will exploit the well known expressions for the Ricci curvature in various setups.  We will recall and record some of them here with the relevant sources.

\vspace{.3cm}

\subsection{Riemannian Submersions}\label{ss:prelim:submersions}

The first special case we recall is that of a Riemannian submersion with totally geodesics fibers.  Our setup is that we have Riemannian manifolds $(M^n,g)$ and $(B,g_b)$ together with a Riemannian submersion
\begin{align}
	\pi: M\stackrel{F}{\longrightarrow} B\, .
\end{align}

We will assume throughout this section that the fibers $F_x\equiv \pi^{-1}(x)$ are totally geodesic submanifolds of $M$.\\

Throughout we will let $U,V,..$ denote vertical vector fields on $M$, so $U,V\in TF\equiv \cV\subseteq TM$, and we will let $X,Y,..$ denote horizontal vector fields on $M$, so $X,Y\in T^\perp F\equiv \cH\subseteq TM$.  Though we have assumed the fibers are totally geodesic, there is still a remaining piece of structure, namely the integrability tensor defined by
\begin{equation}
A_{E_1}E_2:=\mathcal{H}\nabla_{\mathcal{H}E_1}\mathcal{V}E_2+\mathcal{V} \nabla_{\mathcal{H}E_1}\mathcal{H}E_2\, ,
\end{equation}
where our notation $\cV E$ and $\cH E$ denote the projections of $E$ to the corresponding subspaces, see \cite[Definition 9.20]{Besse}.  Recall that if $X,Y$ are horizontal vector fields then
\begin{equation}
A_{X}Y=\frac{1}{2}\mathcal{V}[X,Y]\, .
\end{equation}

For the proposition below we refer the reader to O' Neil \cite{Oneill} (see also \cite[Proposition 9.36]{Besse}):

\begin{proposition}[Ricci curvature for Riemannian submersions]\label{prop:Rictotgeo}
Let $\pi:(M,g)\to (B,g_B)$ be a Riemannian submersion with totally geodesic fibers $F$. Then 
\begin{align}
\Ric_M(U,V)=&\, \Ric_F(U,V)+(AU,AV)\, ,\\
\Ric_M(U,X)=&\, \left(\text{div}_BA[X],U\right) \, ,\\
\Ric_M(X,Y)=&\, \Ric_B(X,Y)-2(A_{X},A_{Y})\, ,
\end{align}
where $ \Ric_F$ stands for the Ricci curvature of the fiber with the induced Riemannian metric and $\Ric_B$ is the Ricci curvature of the base, understood as a horizontal tensor on $M$.  
\end{proposition}
\begin{remark}
	Note that in the above proposition we have the explicit expressions
\begin{align}
&(AU,AV):=\sum_ig(A_{X_i}U,A_{X_i}V)\, ,\notag\\
&(A_X,A_Y):=\sum_ig(A_{X}X_i,A_{Y}X_i)\, ,\notag\\
&\text{div}_BA:=\sum_i\left(\nabla_{X_i}A\right)(X_i,\cdot)\, ,
\end{align}
where $\{X_i\}$ is an orthonormal basis of the horizontal space.
\end{remark}

It is helpful to record how the Ricci curvature on the total space of the Riemannian submersion changes when we perform the so called \emph{canonical variation} of the metric, i.e. we define $g_t$ by leaving the horizontal distribution unchanged, the metric on the base unchanged, and scaling the metric on the fibers by a factor $t$. Below we shall assume again that the fibers are totally geodesic, see \cite[Proposition 9.70]{Besse}.

\begin{corollary}\label{cor:canonicalvariation}
Let $\pi:(M,g)\to (B,g_B)$ be a Riemannian submersion with totally geodesic fibers and let $g_t$ the Riemannian metric on $M$ obtained by scaling the fibers metrics with a factor $t$. Then
\begin{align}
\Ric_t(U,V)=& \, \Ric_F(U,U)+t^2(AU,AV)\, ,\\
\Ric_t(X,U)= & \, t \left(\text{div}_BA[X],U\right)\, , \\
\Ric_t(X,Y)=& \, \Ric_B(X,Y)-2t\left(A_X,A_Y\right)\, .
\end{align}
Above, $A$ denotes the integrability tensor of the Riemannian submersion $\pi:(M,g)\to (B,g_B)$.
\end{corollary}

\vspace{.3cm}

A particularly natural form of Riemannian submersion is obtained via principal bundles:

\begin{definition}[Riemannian Principal Bundle]\label{d:riemannian_principal_bundle}
	We call a principal $G$-bundle $P\stackrel{G}{\longrightarrow} B$ a Riemannian $G$-principal bundle if it is equipped with a Riemannian metric $g_P$ which is invariant under the $G$-action.
\end{definition}  

Observe that if $P$ is a Riemannian $G$-principal bundle then it well defines a metric $g_B$ on the base $B$ through the quotient, and as the horizontal distribution $\cH=(TG_x)^\perp$ is right invariant it well defines a principal connection $\xi\in \Omega^1(P;\mathfrak{g})$.  The remaining information is a family of right invariant metrics on the $G$ fibers, which may equivalently be viewed as a metric on the adjoint vector bundle over $B$.  Conversely, this triple of data well defines a metric $g_P$ on $P$ which is invariant under the $G$ action.\\

Consider the case when $G$ is simple and there exists a unique bi-invariant metric $\langle,\rangle$ on $G$ up to scaling.  Then given a metric $g_B$ on the base and $\xi$ a connection one form, we can write a Riemannian principal bundle structure on $P$ as
\begin{equation}\label{eq:Vilms}
g_P(X,Y)\equiv g_B(\pi_*[X],\pi_*[Y])+\lambda(x)\,\langle \xi[X],\xi[Y]\rangle_G\, ,
\end{equation} 
where $\lambda:B\to \dR^+$ determines the scaling of the fibers.  It follows from Vilms \cite{Vilms} that this metric has totally geodesic fibers iff $\lambda(x)=\lambda$ is a constant.

\vspace{.3cm}
\subsection{Riemannian Submersions and Circle Bundles}

Let us now restrict ourselves to the case of a Riemannian $S^1$-principal bundle, so that $\pi:M\to B$ is the total space of an $S^1$-principal bundle over $B$.  Note that if $(B,g_B)$ is a Riemannian manifold, then an $S^1$-invariant metric on $M$ is well defined by the additional data of a principal connection $\eta\in \Omega^1(M)$ and a smooth $f:B\to \dR^+$ which prescribes the length of the $S^1$ fiber above a point.  If $\partial_t$ is the invariant vertical vector field coming from the $S^1$ action, then we have the expressions
\begin{align}
	&\cH = \ker\eta\, ,\notag\\
	&\eta[\partial_t] = 1\, ,\notag\\
	&g(\partial_t,\partial_t) = f^2\, .
\end{align}
In the case of an $S^1$ bundle we have that $d\eta = \pi^*\omega$ where $\omega\in \Omega^2(B)$ is the curvature $2$-form, which relates to the integrability tensor $A$ on $M$ by 
\begin{align}
	A(X,Y) = -\frac{1}{2}\omega[X,Y]\,\partial_t\, .
\end{align}

The following proposition is borrowed from \cite[Lemma 1.3]{GilkeyParkTuschmann}, where it was used to show that any principal $S^1$ bundle $\pi:M\to B$ admits an $S^1$-invariant metric of positive Ricci curvature when the base $(B,g_B)$ has positive Ricci curvature and the total space has finite fundamental group. 

\begin{proposition}\label{prop:S1bundlewarped}
Let $M\stackrel{S^1}{\longrightarrow} B$ be a Riemannian $S^1$-principal bundle as above with $X$ a unit horizontal vector and $U=f^{-1}\partial_t$ a unit vertical vector. Then 
\begin{align}\label{eq:RicS^1warped}
\Ric(U,U)=&\,   -\frac{\Delta f}{f}+ \frac{f^2}{4}|\omega|^2\, ,\\
\Ric(U,X)=&\, \frac{1}{2}\left(-f\left(\text{div}_B\omega\right)(X)+3\omega[X,\nabla f]\,\right)   \, \\
\Ric(X,X)=&\, \Ric_B(X,X)-  \frac{f^2}{2}|\omega[X]\,|^2 -\frac{\nabla^2f(X,X)}{f}\, ,
\end{align}
where it is understood, when necessary, that we are identifying the horizontal vector field $X$ with an element of $TB$.
\end{proposition}

\vspace{.3cm}

\subsection{Doubly Warped Products}

A particularly common ansatz is that of the doubly warped factor.  Let $(N^{n_1}_1,h_1)$ and $(N^{n_2}_2,h_2)$ be Riemannian manifolds.  Consider the space $M\equiv (r_-,r_+)\times N_1\times N_2\subseteq \dR^+\times N_1\times N_2$ with the metric
\begin{align}\label{e:doubly_warped_metric}
	g\equiv dr^2+f_1(r)^2g_{N_1}+f_2(r)^2g_{N_2}\, .
\end{align}

Let $X_1\in TN_1$ and $X_2\in TN_2$ represent unit directions on $N_1$ and $N_2$ respectively.  Then the nonzero Ricci curvatures on $M$ are given by
\begin{align}
	\Ric(\partial_r,\partial_r) &= -n_1\frac{f_1''}{f_1}-n_2\frac{f_2''}{f_2}\notag\\
	\Ric(X_1,X_1) &= \Ric_{N_1}(X_1,X_1)-\frac{f_1''}{f_1}+\Big(\frac{f'_1}{f_1}\Big)^2-\frac{f'_1}{f_1}\Big(n_1\frac{f'_1}{f_1}+n_2\frac{f'_2}{f_2}\Big)\, ,\notag\\
	\Ric(X_2,X_2) &= \Ric_{N_2}(X_2,X_2)-\frac{f_2''}{f_2}+\Big(\frac{f'_2}{f_2}\Big)^2-\frac{f'_2}{f_2}\Big(n_1\frac{f'_1}{f_1}+n_2\frac{f'_2}{f_2}\Big)\, ,\notag\\
\end{align}
see for instance \cite{Petersen}, dealing with the case where $(N^{n_1}_1,h_1)$ and $(N^{n_2}_2,h_2)$ are isometric to standard spheres.

\vspace{.5cm}

\section{Step 1: The Base Model}

In this Section we complete the construction of the base model $\cB(\epsilon,\delta)$.  It is a straightforward construction, however as we have discussed previously the base models play the role of the $S^3\times \dR^4$ vertices in the topological construction.  As such they are the starting point for each new inductive step and it seems worthwhile to record their geometric properties.  Our main goal is to prove Proposition \ref{p:step1}, which we restate below for the ease of readability:

\begin{proposition}[Step 1: The Model Space]\label{p:step1:2}
	For each $\delta>0$ and $1>\epsilon>0$ , there exists a smooth manifold $\cB^7=\cB(\epsilon,\delta)$ such that the following hold:
\begin{enumerate}
	\item $(\cB^7,g_B,p)$ is a complete Riemannian manifold with $\Ric\geq 0$, it is diffeomorphic to $S^3\times \dR^4$.
	\item There exists $B_{10^{-3}}(p)\subseteq U\subseteq B_{10^{-1}}(p)$ such that $\cB\setminus U$ is isometric to $S^3_{\delta}\times A_{10^{-2},\infty}(0)\subseteq S^3_{\delta}\times C(S^3_{1-\epsilon})$ 
	\item There is an isometric $T^2=S^1\times S^1$ action on $\cB$ for which on $\cB\setminus U\approx S^3_{\delta}\times C(S^3_{1-\epsilon})$ the first $S^1$ acts on the $S^3_{\delta}$ factor by a globally free (left) Hopf rotation and the second $S^1$ acts on the cross sections $S^3_{1-\epsilon}$ of the cone factor by (left) Hopf rotation. 
	\item The $S^1$-action induced by the homomorphic embedding $S^1\ni \theta\mapsto (a\theta,b\theta)\in T^2$ is free whenever $(a,b)\in \dZ\times \dZ$ are coprime and $a\neq 0$.
\end{enumerate}
\end{proposition}

To build $\cB(\epsilon,\delta)$ let us start with the geometry of $S^{3}_\delta\times \dR^{4}$.  Note that if we view $\dR^{4}=C(S^{3}_1)$ then we can write this metric as
\begin{align}
	g_0 \equiv dr^2 +\delta^2 g_{S^{3}}+ r^2 g_{S^{3}}\, .
\end{align}
We will warp this metric by considering the ansatz
\begin{align}
	g \equiv dr^2 +\delta^2 g_{S^{3}}+ h(r)^2 g_{S^{3}}\, .
\end{align}

If we let $a,b,c$  denote the directions on the first $S^{3}$ factor and $i,j,k$ on the second $S^{3}$ factor, then we can compute the nonzero terms of the Ricci curvature of this ansatz as

\begin{align}\label{e:step1:Ricci}
	&\Ric_{rr} = -3\frac{h''}{h}\, ,\notag\\
	&\Ric_{aa} = \frac{2}{\delta^2}\, ,\notag\\
	& \Ric_{ii} = 2\frac{1-(h')^2}{h^2}-\frac{h''}{h}\, .
\end{align}

Let us now consider any smooth function $h(r)$ of the following form:

\begin{align}
	h(r) \equiv \begin{cases}
		r & \text{ if } r\leq 10^{-5}\, ,\notag\\
		h''<0\, , & \text{ if } 10^{-5}\leq r\leq 10^{-3}\, ,\notag\\
		10^{-4}+(1-\epsilon)(r - 10^{-4}) & \text{ if } r\geq 10^{-3}\, .\notag\\
	\end{cases}
\end{align}
We can build a function as above by smoothing out $h(r)=\min\{r,10^{-4}+(1-\epsilon)(r - 10^{-4})\}$.  Note these two linear functions intersect at $r=10^{-4}$.  Observe that any such function satisfies $|h'|\leq 1$ on $(0,+\infty)$ as $h''\leq 0$, and so we can compute
\begin{align}\label{e:step1:Ricci_compute}
	&\Ric_{rr} \geq 0\, ,\notag\\
	& \Ric_{aa} = \frac{2}{\delta^2}>0\, ,\notag\\
	&\Ric_{ii} \geq 0\, .
\end{align}

In particular, this metric has nonnegative Ricci curvature.  The $S^1\times S^1$ torus Hopf action on $S^{3}\times S^3$ induces an isometric action with respect to $g$ from our warping coordinates.  On the domain $\{r\geq 10^{-3}\}$ we can write this metric
\begin{align}
	dr^2 +\delta^2g_{S^3}+\big(10^{-4}+ (1-\epsilon)(r-10^{-4})\big)^2 g_{S^3}\, .
\end{align}

If we consider the domain $U=\{r\leq 10^{-2}-10^{-4}\}$ then we see that $\cB\setminus U$ is isometric to the annulus $S^3_\delta\times A_{10^{-2},\infty}\subseteq S^3_\delta\times C(S^3_{1-\epsilon})$, as claimed.  $\qed$

\vspace{.5cm}

\section{Equivariant Mapping Class Group and Ricci Curvature}\label{s:equivariant_Ricci}

Recall that we have equipped $S^3\times S^3$ with the left $(a,b)$-Hopf action given by 
\begin{equation}
	\theta\cdot_{(a,b)}(s_1, s_2)  = (a\theta \cdot s_1, b\theta\cdot s_2)\, ,\quad \theta \in S^1\, ,
\end{equation}
where $\theta\cdot s$ is the classical (left) Hopf action on $S^3$.  We will consider in this Section $S^3\times S^3$ with the standard metric $g_{S^3\times S^3}$ and the distinct $(1,k)$ and $(1,0)$ actions.  Our goal is to connect these two spaces by a family of smooth $S^1$-invariant Riemannian manifolds $(S^3\times S^3,g_t)$ with positive Ricci curvature which begin and end with the standard geometry but with these distinct actions.  The subtle point is that while $g_1$ is isometric to $g_0$, necessarily it cannot be equal to the standard metric as a tensor and we will have $g_1=\phi^*g_{S^3\times S^3}$ for some mapping class nontrivial diffeomorphism $\phi$.  The precise statement is the following, which is the main goal of this Section:\\

\begin{theorem}[Equivariant Mapping Class and Ricci]\label{t:equivariant_mapping_class:2}
Let $g_0=g_{S^3\times S^3}$ be the standard metric on $S^3\times S^3$, and let $k\in \dZ$.  Then there exist a diffeomorphism $\phi:S^3\times S^3\to S^3\times S^3$ and a family of metrics $(S^3\times S^3,g_t)$ so that
\begin{enumerate}
	\item $\Ric_t>0$ for all $t\in [0,1]$
	\item The $S^1$-action $\cdot_{(1,k)}$ on $S^3\times S^3$ is an isometric action for all $g_t$.
	\item $g_1 = \phi^*g_0$ with $\phi\big(\theta\cdot_{(1,k)}(s_1,s_2)\big)= \theta\cdot_{(1,0)}\phi(s_1,s_2)$ .
\end{enumerate}
\end{theorem}
\begin{remark}\label{r:equivariant_mapping_diffeo}
	We will show in Section \ref{ss:equivariant_metric:diffeo} that in the case $k=1$ we can pick $\phi=\phi_1$ to be the diffeomorphism specified in the Examples of Section \ref{ss:equivariant_metric:examples}.
\end{remark}
\begin{remark}
	The above Theorem is for the $(1,k)$-actions induced by the left Hopf actions.  Of course, one can equally well deal with the right actions or mapping $(k,1)$-actions to $(0,1)$-actions.
\end{remark}

The proof will be divided into various steps.  In Section \ref{ss:equivariant_metric:geometry_N} we will begin by studying the geometry of $N\equiv S^1\backslash S^3\times S^3$, that is the quotient of $S^3\times S^3$ by the $(1,k)$-Hopf action.  We will see that we can view $N$ itself as a Riemannian $S^3$-principal bundle over $S^2$, so that we have the viewpoint
\begin{align}
	S^3\times S^3 \stackrel{S^1}{\longrightarrow} N \stackrel{S^3}{\longrightarrow} S^2\, .
\end{align}
Each of these bundle structures are Riemannian submersions, however the connections will be nonstandard and the $S^3$ fibers will not have the standard geometry.  In particular, $N \stackrel{S^3}{\longrightarrow} S^2$ is an $S^3$ bundle over $S^2$ with twisted connection and totally geodesic fibers. We notice that the induced metric on the fibers will be right invariant but not biinvariant.\\

Our construction will now proceed as follows.  In the first step of the construction in Section \ref{ss:equivariant_metric:step1} we will construct a family $(N_t,h_t)=S^1\backslash(S^3\times S^3,g_t)$ with positive Ricci which begins at $N_0\equiv N$ as above, and ends at $N_{1/2}$.  Each $N_t$ will itself be the total space of a Riemannian $S^3$-principal bundle over $S^2$, where for $t\in [0,1/2]$ the connection will be fixed but the geometry of the $S^3$ fibers will round off so that $N_{1/2}$ will become an $S^3$ bundle over $S^2$ with small but round $S^3$ fibers.  This construction will essentially take place directly on $(S^3\times S^3,g_t)$, however it will be crucial in the remaining steps that we emphasize the geometry of $N_t$ during this process.\\

The second and third steps of the construction will focus on just changing the geometry of $(N_t,h_t)$.  For each $t\in [1/2,3/4]$ we will have in Section \ref{ss:equivariant_metric:step2} that $N_t$ is a Riemannian $S^3$-principal bundle with totally geodesic fibers which are isometric to spheres.  The connection of this bundle on $N_{1/2}$ is highly nontrivial, and in the second step we will vary this connection until we arrive at a flat connection on $N_{3/4}$.  The sphere fibers may shrink during this process in order to preserve the positive Ricci condition.  It is worth emphasizing that although the connection at the end is flat, it will look highly nontrivial in the original coordinates of $N=S^1\backslash S^3\times S^3$.  In the third step of Section \ref{ss:equivariant_metric:step3}, where we now have a flat connection with small round sphere fibers, we will increase the size of the sphere fiber until we arrive at $N_1=S^2\times S^3$ isometrically.\\

The last steps of the construction will view $S^3\times S^3$ as the total space of a Riemannian $S^1$-principal bundle over $N_t$.  This $S^1$ bundle structure will generate our family of actions, and in particular by definition will begin with the $(1,k)$ action on $S^3\times S^3$.  A metric $g_t$ on this total space is then well defined by an $S^1$-principal connection $\eta_t$ and fiber size functions $f_t:N_t\to \dR$, see Section \ref{ss:prelim:submersions}.  We will choose $\eta_t$ to be a Yang-Mills connection under the Coulumb gauge, which will allow us to show $\Ric_t\geq 0$.  In Step 5 of Section \ref{ss:equivariant_metric:step5} we will further vary the fiber size $f_t$ of the circle bundle so that we can push the Ricci curvature up to become strictly positive.  Recall the construction of $N_t$ ends with $N_1=S^2\times S^3$ isometrically.  Thus we will see that the Yang-mills connection must give us a total space which is isometric to $S^3\times S^3$ with the bundle action being the $(1,0)$ action.  In Section \ref{ss:equivariant_metric:finish_proof} we will put all of these ingredients together in order to complete the proof of Theorem \ref{t:equivariant_mapping_class:2}.

\vspace{.3cm}

\subsection{Examples: Equivariant Mappings $\phi_k:S^3\times S^3\to S^3\times S^3$}\label{ss:equivariant_metric:examples}

In this subsection let us first present some explicit examples of equivariant mappings which satisfy:

\begin{align}
	\phi_k\big(\theta\cdot_{(1,k)}(s_1,s_2)\big) = \theta\cdot_{(1,0)}\phi_k(s_1,s_2)\, .
\end{align}

We will see in Theorem \ref{t:equivariant_mapping_class:2} that we will be able to take these explicit mappings as our end diffeomorphisms.\\

\subsubsection{The Case $k=1$}\label{sss:equivariant_metric:examples:k=1}  Let us treat $k=1$ and $k>1$ separately.  In the case of $k=1$ we will view each $S^3$ as the corresponding Lie Group $\text{SU}(2)$ and so write $(s_1,s_2)\in S^3\times S^3$.  Let us explicitly define the mapping
\begin{align}
	\phi_1(s_1,s_2)=(s_1,s_1^{-1}s_2)\, .
\end{align}

It follows that the left $(1,1)$-action pushes forward to the left $(1,0)$ action, that is $\phi_1\big(\theta\cdot_{(1,1)}(s_1,s_2)\big) = \theta\cdot_{(1,0)}\phi_1(s_1,s_2)$ .\\

\subsubsection{The Case $k>1$}

The case $k>1$ is a little trickier. 
Let $u,z\in S^3$. We write $u=(u_1,u_2)$, $z=(z_1,z_2)$, where $u_1,u_2,z_1,z_2\in \mathbb C$. 
The diffeomorphism is given by
\begin{align}
	\phi_k\big(u_1,u_2,z_1,z_2\big) &\equiv \Big(u_1,u_2, \frac{1}{\sqrt{|u_1|^{2k} + |u_2|^{2k}}}(\bar u_1^k, -u_2^k)\cdot (z_1,z_2)\Big)\notag\\
	&=\Big(u_1,u_2, \frac{1}{\sqrt{|u_1|^{2k} + |u_2|^{2k}}}(\bar u_1^kz_1+u_2^k\bar z_2,-u_2^k\bar z_1+\bar u_1^k z_2)\Big)\, .
\end{align}
where $\cdot$ denotes the product of $S^3$ as Lie-group.

With this choice, we have the claimed equivariance property that
\begin{align}
	\phi_k\big(\theta\cdot_{(1,k)}(u_1,u_2,z_1,z_2)\big) = \theta\cdot_{(1,0)}\phi_k(u_1,u_2,z_1,z_2)\, .
\end{align}
Moreover, $\phi_k$ is equivariant with respect to the $S^3$-action on the second $S^3$-factor, i.e.
\begin{equation}
	\phi_k(u, z \cdot g) = \phi_k(u,z)\cdot g
	\, , \quad
	\text{for every $g\in S^3$}\, .
\end{equation}

Note this corresponds to the previous construction when $k=1$.

\vspace{.3cm}

\subsection{The Geometry of $N\equiv S^1\backslash S^3\times S^3$}\label{ss:equivariant_metric:geometry_N}

Let us begin with $S^3\times S^3$ endowed with the product Lie group structure.  Let $U_1,U_2,U_3$ be an orthonormal basis of right invariant vector fields on the first $S^3$ factor, and similarly let $V_1,V_2,V_3$ be an orthonormal basis of right invariant vector fields on the second $S^3$ factor.  Let $U_1$ and $V_1$ be the right invariant vector fields induced from the left Hopf actions.  We will write $U_j^*$ and $V_j^*$ to denote the dual basis of one forms.\\

Let us now define $(N,h)\equiv S^1\backslash S^3\times S^3$ to be the isometric quotient of $S^3\times S^3$ by the left $(1,k)$-Hopf action, which is a free and isometric action.  We have that $\pi_{(1,k)} : S^3 \times S^3 \to N$ is a principal $S^1$-bundle, and it is endowed with the Yang-Mills principal connection
\begin{equation}
  \eta_0 := \frac{1}{1 +k^2} \Big(U_1^* + k V_1^* \Big)\, .
\end{equation}

\medskip

We shall denote the projection $\pi_{(1,k)}(s_1,s_2)\in N$ as $[s_1,s_2]$, where $(s_1,s_2)\in S^3\times S^3$.  Note that the right $S^3$ action on the second factor of $S^3\times S^3$ commutes with the left $(1,k)$ action.  The quotient action on $N$:
\begin{equation}\label{eq:rightS^3}
	[s_1,s_2]\cdot s = [s_1, s_2 \cdot s]\, , \quad s\in S^3\, , \, \, [s_1, s_2]\in N\, ,
\end{equation}
is a free and isometric action, and thus $\pi: N\stackrel{S^3}{\longrightarrow}S^2$ admits a structure of principal $S^3$-bundle over $S^2$.  Note that
\begin{equation}
	\pi ([s_1,s_2]) = \pi_{\rm Hopf}(s_1)\, ,
\end{equation}
is the same as the Hopf projection on the first factor.

\begin{remark}
We can understand $N$ in the following manner.  Begin with $S^3\stackrel{S^1}{\longrightarrow} S^2$, viewed as an $S^1$-principal bundle over $S^2$ with respect to the {\it left} Hopf action.  Consider the homomorphism $\rho: S^1 \to S^3$, $\rho(z) = z^{-k}$ where $z\in S^1$ is identified with a complex number and $z^{-k}\in S^3$ with a unit quaternion.  Then we can identify $N$ as the associated $S^3$ bundle over $S^2$ under this representation.  This point of view is particularly convenient for writing coordinate expressions of the above. \\
\end{remark}

Our first claim is about the isomorphism class of the principal $S^3$-bundle $\pi:N\to S^2$.

\begin{lemma}\label{lemma:Xtrivialbundle}
	The principal $S^3$-bundle $\pi: N \to S^2$ is isomorphic to the trivial bundle $S^2\times S^3$.
\end{lemma}

\begin{proof}
It is a classical fact that the isomorphism classes of principal $G$-bundles over spheres $S^n$ are in bijection with $\pi_{n-1}(G)$, where $G$ is any connected Lie group. See for instance \cite[Chapter 18]{Steenrod}. As $n=2$, and $G=S^3$ is simply connected, we infer that any principal $S^3$-bundle over $S^2$ is isomorphic to the trivial one.
\end{proof}

\begin{remark}
It was already noted in \cite[Example 4.1]{PetruninTuschmann} that $N$ is diffeomorphic to $S^2\times S^3$, with an argument based on the classification of simply connected 5-manifolds, see also \cite{WangZiller}. The principal $S^3$-bundle structure of $N$ plays a key role for our purposes.\\
\end{remark}

\subsubsection{Geometry on $N$}
Our next goal is to understand the metric $h$ on $N$.\\ 

Begin by observing that if we endow the base space $S^2$ with $\frac{1}{4}g_{S^2}$, the round metric of radius $\frac{1}{2}$, then $\pi: (N, h) \to S^2_{1/2}$ is a Riemannian submersion.  To see this let us write $S^2 = N/S^3= S^1\backslash S^3\times S^3/S^3$.  If we quotient by the $S^3$ action first then we have $S^2=S^1\backslash S^3$, where the $(1,k)$ action descends to the $(1,0)$ action on $S^3$ as we have quotiented out the second factor.  Thus the quotient space is the usual Hopf quotient, which is the sphere of radius $\frac{1}{2}$ .\\

Notice that over $(S^3\times S^3, g_{0})$ we have the orthonormal basis
\begin{equation}\label{eq:frame 1k}
	\frac{1}{\sqrt{1 + k^2}}( U_1 + k V_1),\, U_2,\,  U_3,\, \frac{1}{\sqrt{1 + k^2}}(-k U_1 + V_1),\,  V_2,\,  V_3 \,  , 
\end{equation}
where as before $U_1, U_2, U_3$ are the right invariant vector fields in the first $S^3$ factor, and  $V_1, V_2, V_3$ are the right invariant vector fields in the second $S^3$ factor. The first vector field is vertical with respect to the $(1,k)$-projection map $\pi_{(1,k)}: (S^3 \times S^3,g_0) \to (N,h)$, while the last five vector fields are horizontal.\\

The following claims are almost immediate:
\begin{itemize}
	\item[(1)] $d\pi_{(1,k)}[V_2]$, $d\pi_{(1,k)}[V_3]$, $d\pi_{(1,k)}\Big[\frac{1}{\sqrt{1 + k^2}}(-k U_1 + V_1)\Big]$ span the vertical directions in $N$.
	
	\item[(2)]  $d\pi_{(1,k)}[U_2]$,  $d\pi_{(1,k)}[U_3]$ span the horizontal directions in $N$.\\
\end{itemize}
In order to check (1), it is enough to observe that the three vectors span a three-dimensional subspace of the tangent of $N$, and belong to the kernel of $d\pi$ since $\pi\circ \pi_{(1,k)}(g_1,g_2) = \pi_{\rm Hopf}(g_1)$, by definition.

Claim $(2)$ follows immediately since the span of the vectors in (2) is two-dimensional and orthogonal to the span of the vectors in (1). The orthogonality follows from the fact that \eqref{eq:frame 1k} is an orthonormal frame:
$h(d\pi_{(1,k)}[U_2], d\pi_{(1,k)}[V_2]) = g_0(U_2, V_2) = 0$, and similarly for the other vectors.

\vspace{.3cm}

\subsection{Step 1: Squishing the Fibers}\label{ss:equivariant_metric:step1}

By looking at the last three vector fields in \eqref{eq:frame 1k} we see that the metric on the $S^3$ fibers of $N$ is right invariant, but not left invariant.  We do see that it is invariant under the left $S^1$ action, the action of which just rotates the $V_2,V_3$ plane.  Our first step of the construction will be to round off the metric so that the fibers become bi-invariant round spheres.  We will work directly to build a family of metrics $(S^3\times S^3,g_t)$ for $t\in [0,1/2]$ which continue to be invariant by the left $S^1\times S^1$ Hopf action, and invariant by the right $S^3$ action.  In particular, for each such $t$ we can define $(N_t,h_t)=(N,h_t)\equiv S^1\backslash (S^3\times S^3,g_t)$ where we have quotiented out by the $(1,k)$-Hopf action.  As there are many properties of this family that we will need in the sequel, let us summarize the end results of our constructions for this step:\\

\begin{lemma}\label{l:equivariant_mapping:step1}
For $t\in[0,1/2]$ there exists smooth families $(S^3\times S^3,g_t)$, $(N,h_t)\equiv S^1\backslash (S^3\times S^3,g_t)$ which satisfy $\Ric_{h_t},\Ric_{g_t} >0$ and such that:
\begin{itemize}
	\item[(i)] $\pi: (N, h_t) \stackrel{S^3}{\longrightarrow} S^2_{1/2}=(S^2, \frac{1}{4}g_{\rm S^2})$ is a Riemannian $S^3$-principal bundle with totally geodesic fibers.  Further the $S^3$ fibers in $N_{1/2}$ are bi-invariant spheres;
	\item[(ii)]  $\pi_{(1,k)}: (S^3 \times S^3,g_t) \to (N_t, h_t)$ is a Riemannian $S^1$-principal bundle with total geodesic fibers and constant connection $\eta_t=\eta_0$.  Further, the connections $\eta_t$ are Yang-Mills and in Coulomb Gauge.
\end{itemize}
\end{lemma}
\begin{remark}
We will also take our construction so that $g_t,h_t$ are constant metrics near $t=0,\frac{1}{2}$.
\end{remark}
\begin{remark}
	Recall that we call a principal $G$-bundle $P\stackrel{G}{\longrightarrow} B$ a Riemannian $G$-principal bundle if $P$ is equipped with a metric $g_P$ which is invariant under the right $G$ action.  Recall such a metric defines a principal connection $\xi$ and a family over $B$ of right invariant metrics on $G$.
\end{remark}
\vspace{.1cm}

To define our family of metrics on $S^3\times S^3$ will have as a global orthonormal basis the vector fields
\begin{align}\label{eq:step1 ortho}
	&T^t\equiv \frac{1}{f_t}( U_1 + k V_1)\, ,\;\;X_2 \equiv U_2\, , \;\; X_3 \equiv U_3\, ,\notag\\
	& W_1^t\equiv \frac{1}{a_t}(-k U_1 + V_1)\, ,\;\;   W_2^t\equiv \frac{1}{a_tb_t}V_2\, ,\;\;  W_3^t\equiv \frac{1}{a_tb_t}V_3 \, .
\end{align}

Let us begin with some remarks on this basis.  Observe that $U_1+kV_1$ is the direction associated to the $(1,k)$-action, and so $T$ is the unit direction associated to the $(1,k)$ action.  In particular, we have the $S^1$-fibers have length $2\pi\,f_t$, and in order for the geometry at $t=0$ to be the standard geometry we will choose $f_t=\sqrt{1+k^2}$, $a_t=\sqrt{1+k^2}$ and $b_t=\frac{1}{\sqrt{1+k^2}}$ for $t$ near zero.  

The vector fields $X_2,X_3$ represent the horizontal directions associated to the base $S^2$.  As the geometry is symmetric with respect to the $X_2,X_3$ and $W_2^t,W_3^t$ indices we have that the left $S^1\times S^1$ actions are isometric actions.  As $X_2,X_3$ are time independent we have that $N_t/S^3 = S^2_{1/2}$ remains a sphere of radius $1/2$.

The directions $W^t_1,W^t_2,W^t_3$ represent the directions horizontal with respect to the $S^1$ action, but will induce vertical vector fields on $N$ with respect to the right $S^3$ action.  The orbit of the $(k,1)$ action in the torus generated by $\{T^t,W^t_1\}$ is $\frac{1}{\sqrt{1+k^2}}$ dense, and hence the geometry of the $S^3$ fibers of the Riemannian $S^3$-bundle $N_t\stackrel{S^3}{\longrightarrow} S^2_{1/2}$ can be seen to be determined by the right invariant orthonormal basis $\{\frac{1+k^2}{a_t} V_1, \frac{1}{a_tb_t}V_2,\frac{1}{a_tb_t} V_3\}$.\\

We will therefore choose our warping function $b_t=b(t)$ as any smooth function such that 
\begin{align}
	b(t)=\begin{cases}
		\frac{1}{\sqrt{1+k^2}}&\text{ if } t\text{ is near }0\, ,\notag\\
		\dot b_t\geq 0&\text{ if } t\in[0,1/2]\, ,\notag\\
		\frac{1}{1+k^2}&\text{ if } t\text{ is near }1/2\, .
	\end{cases}
\end{align}

We will choose $a_t$ to be nonincreasing shortly. This will play the role of squishing additional positive curvature into the system to compensate for the movement of $b_t$. \\ 

Let us now study some properties about our ansatz.  We begin by computing the nonzero brackets of our vector fields.  These computations all boil down to using that $U_j,V_j$ are a standard right invariant basis for $S^3$:
\begin{align}\label{e:step2:step1:1}
	&[T^t,X_2] = \frac{2}{f_t}X_3\, ,\;\; [X_3,T^t] = \frac{2}{f_t}X_2\, ,\;\;[X_2,X_3] = \frac{2f_t}{1+k^2}T^t-\frac{2ka_t}{1+k^2}W^t_1\, ,\notag\\
	&[T^t,W_2^t] =\frac{2k}{f_t}X_3\, ,\;\;[W_3^t,T^t] =\frac{2k}{f_t}W^t_2\, ,\notag\\
	&[W_1^t,W_2^t] = \frac{2}{a_t}W^t_3\, ,\;\; [W_3^t,W_1^t] = \frac{2}{a_t}W^t_2\, ,\;\;[W_2^t,W_3^t] = \frac{2kf_t}{(1+k^2)a_t^2b_t^2}T^t+\frac{2}{(1+k^2)a_tb^2_t}W^t_1\, ,\notag\\
	&[W_1^t,X_2] =-\frac{2k}{a_t}X_3\, ,\;\;[X_3,W_1^t] =-\frac{2k}{a_t}X_2\, .
\end{align}

There are several takeaways from the above computations.  In order to have an ease in the formulas, let us be abusive in notation and define $W^t_T = T^{t}$, $W^t_{X_2}=X_2$ and $W^t_{X_3}=X_3$.  Observe that the structural and Christofell coefficients 
\begin{align}
	&c_{ijk}\equiv g_t( [W^t_i,W^t_j],W^t_k)\, ,\notag\\
	&\Gamma_{ijk}\equiv g_t(\nabla_{W^t_i} W^t_j,W^t_k)\, ,
\end{align}
are constants and related by $\Gamma_{ijk} =\frac{1}{2}\big(c_{ijk}+c_{kij}-c_{jki}\big)$.  As a first observation note that to be nonvanishing all three of the indices must be distinct, and in particular we can conclude that
\begin{align}\label{e:step2:step1:2}
	&\nabla_{W^t_j}W^t_j = 0 \;\, ,\notag\\
	&g_t(\nabla_{W^t_k}W^t_j,W^t_k) = 0\, .
\end{align}
To put this into perspective, the first equation will tell us shortly that all the bundles of interest have totally geodesic fibers, and the second equation will tell us that our $S^1$-connections are Yang-Mills.\\

\subsubsection{The $S^1$-Principal Bundle $S^3\times S^3\stackrel{S^1}{\longrightarrow} N_t$}

Thus let us consider now the $S^1$-principal bundle $S^3\times S^3\stackrel{S^1}{\longrightarrow} N$.  It follows from \eqref{e:step2:step1:2} that $\nabla_{T^t}T ^t= 0$, and hence the $S^1$-fibers of this principal bundle are totally geodesic.  It follows from \eqref{eq:step1 ortho} that the connection $1$-form of this bundle is given by the metric dual
\begin{align}
	\eta_t = f_t^{-1}g_t(T^t,\cdot)\, .
\end{align}
Note that as $W^t_j\propto W^0_j$ we have that $\eta_t =\eta_0$ is independent of $t$ as a $1$-form.  We can compute the curvature form $\omega_t=d\eta_t$ of this bundle by
\begin{align}
	\omega_t(W_j^t,W_k^t) = \frac{1}{2f_t}\big(\langle \nabla_{W_j^t}T, W_k^t\rangle -\langle \nabla_{W_k^t}T, W_j^t\rangle\big) = \frac{c_{kjT}}{f_t}\, ,
\end{align}
so that we can use \eqref{e:step2:step1:1} to write
\begin{align}
\omega_t &= -\frac{2}{1+k^2}\big(X^{*}_2\otimes X^{*}_3-X^{*}_3\otimes X^{*}_2\big)-\frac{2k}{(1+k^2)a_t^2 b_t^2}	\big(W^{t,*}_2\otimes W^{t,*}_3-W^{t,*}_3\otimes W^{t,*}_2\big)\, .
\end{align}
Consequently, it follows from \eqref{e:step2:step1:2} that
\begin{align}
&\text{div}_{N_t} \eta_t=0\, ,\notag\\
	&\text{div}_{N_t} \omega_t=0\, ,
\end{align}
and hence $\eta_t$ is a Yang-Mills connection in Coulomb gauge.\\

Let us now compute the Ricci curvature $\Ric_{g_t}$ of $(S^3\times S^3,g_t)$ in terms of the Ricci curvature of $N_t$.  We will compute $\Ric_{N_t}$ shortly after.  Let us use $H\in \text{span}\{X_2,X_3,W^t_j\}$ to denote any unit horizontal vector with respect to the $S^1$ action. Using Proposition \ref{prop:S1bundlewarped} we have that
\begin{align}
	\Ric_{g_t}(T^t,T^t) &= \frac{1}{4}|\omega_t|^2_t\geq c(k)\frac{f_t^2}{a_t^2}\, ,\notag\\
	\Ric_{g_t}(T^t,H) &= -\text{div}_{N_t}\omega_t(H)=0\, ,\notag\\
	\Ric_{g_t}(H,H) &=\Ric_{h_t}(H,H)-\frac{f_t^2}{2}|\omega_t[H]|^2\notag\\
	&\geq \Ric_{h_t}(H,H)-c(k)\frac{f_t^2}{a_t^2}\, ,
\end{align}
where we have used that $b_t$ is bounded above and below by functions of $k$.  Our main takeaway is that if $\Ric_{h_t}>0$ then for $f_t\leq f_t(h_t,a_t)$ we have that the Ricci curvature of $(S^3\times S^3,g_t)$ is itself positive.

\vspace{.25cm}
\subsubsection{The Geometry of $N_t$}

We understand that the right $S^3$ action on $N_t$ is an isometric action which gives $N_t$ the structure of a Riemannian $S^3$-principal bundle
\begin{align}
	N_t\stackrel{S^3}{\longrightarrow} S^2_{1/2}\, .
\end{align}
Let us see that the $S^3$ fibers of this bundle are totally geodesic.  Indeed, on $S^3\times S^3$ we know that $W^t_1,W^t_2,W^t_3$ span the horizontal subspace induced the $S^3$ action downstairs.  On the other hand, we know by \eqref{e:step2:step1:2} that $\nabla_{W^t_j}W^t_j=0$.  One can compute directly from this that the $S^3$ fibers are totally geodesic, but let us also explain it geometrically. As $\nabla_{W^t_j}W^t_j=0$, it follows that the orbits of $W^t_j$ are horizontal geodesics in $S^3\times S^3$ and thus project to geodesics in $N_t$ which are tangent to the $S^3$ fibers.  As they span the tangent of the $S^3$ fibers at each point, we see that the fibers are totally geodesic.\\

To study the Ricci curvature of $N_t$ let us first study the geometry of the $S^3$ fibers.  We again observe from \eqref{eq:step1 ortho} that the geometry of the $S^3$ fibers are right invariant and invariant under the left $S^1$ action.  Thus we see that geometrically the fibers are a Hopf bundle 
\begin{align}
	S^3\stackrel{S^1}{\longrightarrow} S^2_{a_tb_t/2}\, ,
\end{align}
with fibers of length $\frac{a_t}{1+k^2}$ and the standard Hopf connection.  In particular, for $b_t\geq \frac{1}{1+k^2}$ the Ricci curvature of the fibers satisfy 
\begin{align}
	\Ric_F\geq \frac{c(k)}{a_t^2}\, .
\end{align}

Let us compute the Ricci curvature on $N_t$.  Recall that the vector fields $H\in \text{span}\{X_2,X_3,W^t_1,W^t_2, W^t_3\}$ in $S^3\times S^3$ horizontally span the tangent space of $N_t$. 
Though they do not well define vector fields on $N_t$ (one needs to lift a point $[s_1,s_2]\in N_t$ to $(s_1,s_2)\in S^3\times S^3$ for such an association), the value of $\Ric_{N_t}(H,H)$ is independent of the lift.  Now let $\xi_t\in\Omega^1(N_t;su(2))$ be the connection one form for the bundle $N_t\stackrel{S^3}{\longrightarrow} S^2_{1/2}$ with $\Omega_t = d\xi_t +\frac{1}{2}[\xi_t\wedge \xi_t]$ the curvature form.  Using Proposition \ref{prop:Rictotgeo} we have that the Ricci curvature of $N_t$ can be estimated
\begin{align}
	\Ric_{N_t}(W^t_i,W^t_i)\, &\geq c(k)\Big(\frac{1}{a_t^2}-a_t^2|\Omega_t|^2\Big)\, ,\notag\\
	|\Ric_{N_t}(W^t_i,X_j)|\, &\leq  c(k)a_t|\text{div}\, \Omega_t|^2\, ,\notag\\
	\Ric_{N_t}(X_j,X_j)\, &= 4-c(k)a_t^2|\Omega_t[X_j]|^2\, .\notag\\
\end{align}

It follows that if $a_t\leq a_t(k,\xi_t)$ is sufficiently small then $\Ric_{N_t}>2$.  This finishes Step 1 of the construction. $\qed$

\vspace{.3cm}

\subsection{Step 2: Trivializing the connection on $(N_t,h_t)$} \label{ss:equivariant_metric:step2}

In the second and third steps of our construction we focus on $(N_t,h_t)$.  In this second step we change smoothly the principal connection until we get to the flat one, which exists in view of Lemma \ref{lemma:Xtrivialbundle}. Along the process, we keep fixed the metric on the base space $(S^2, \frac{1}{4} g_{S^2})$, and we squish the metric on the fibers by a factor $\lambda(t)$, depending smoothly in time. The latter is needed to keep the Ricci curvature positive along the way.  In the third step we can increase the fiber size to arrive at $N_1=S^2_{1/2}\times S^3_1$ isometrically.  Precisely, the next two steps will accomplish the following:

\begin{lemma}\label{lemma:changeconnectionX}
There exists a smooth family $(N,h_t)$ for $t\in[1/2,1]$ of Riemannian metrics on $N_t$ with $Ric_{h_t}>0$, constant in a neighborhood of the end points, verifying the following properties:
\begin{itemize}
\item[i)] $N_t\stackrel{S^3}{\longrightarrow}S^2_{1/2}$ is a Riemannian $S^3$-principal bundle with totally geodesic fibers isometric to $(S^3,\lambda^2_t g_{S^3})$ 
\item[ii)] The $S^3$-connection $\eta_t$ on $N_t$ is flat for $t\in [3/4,1]$ ,
\item[iii)] $N_1$ is isometric to $(S^2\times S^3,\frac{1}{4}g_{S^2}+ g_{S^3})$ .
\end{itemize}
\end{lemma}

\medskip

Let $\lambda_{1/2}\equiv \frac{a_{1/2}}{1+k^2}$ be the size of the $S^3$ fibers of $N_{1/2}$, as in the previous section.  Recall that $N_{1/2}\stackrel{S^3_{\lambda_{1/2}}}{\longrightarrow} S^2_{1/2}$ has the structure of a Riemannian $S^3$-principal bundle whose metric is well defined by a principal connection $\xi_{1/2}\in \Omega^1(N;su(2))$, the base metric $S^2_{1/2}$, and the binvariant fiber metric $S^3_{\lambda_{1/2}}$.\\
  
For all $t\in [1/2,1]$ we will construct a (smoothly varying) family $\xi_{t}\in \Omega^1(N;su(2))$ and $\lambda_t$ such that $N_t$ is the induced Riemannian $S^3$-principal bundle
\begin{align}
	N_t\stackrel{S^3_{\lambda_t}}{\longrightarrow} S^2_{1/2}\, .
\end{align}

Let us first build the family of connections.
By Lemma \ref{lemma:Xtrivialbundle}, there exists a smooth $S^3$-equivariant map $\Phi:N\to S^2\times S^3$, where we view $S^2\times S^3$ as the trivial principal $S^3$-bundle over $S^2$. The flat principal connection on $S^2\times S^3$ can be identified with the Maurer-Cartan form $\xi^{\rm MC}$ of $S^3$.  Note that $\Phi^*\xi^{MC}$ is then a flat connection on $N\to S^2$, though clearly its coordinate expression does not mesh well with the earlier constructions.  We define $\xi_t$ as the family of connections which come from an affine combination:
\begin{align}
	\xi_t = \big(1-\alpha(t)\big)\,\Phi^*\xi^{\rm MC}+ \alpha(t)\,\xi_{1/2}\, ,
\end{align}
where $\alpha(t)\geq 0$ is a nonincreasing smooth function with $\alpha\equiv 1$ for $t$ near $1/2$ and $\alpha\equiv 0$ for $t$ near $\frac{3}{4}$.  We can then define the metric $h_t$ explicitly as by Vilms \cite{Vilms}: 
\begin{equation}\label{eq:Vilmsappl}
h_t(X,Y):=\frac{1}{4}g^{}_{S^2}(\pi_*[X],\pi_*[Y])+\lambda_t^2g_{S^3}(\xi_t[X],\xi_t[Y])\, ,
\end{equation} 
where we will specify the smooth function $\lambda(t):[\frac{1}{2},1]\to \dR^+$ momentarily.  Let us denote $\Omega_t = d\xi_i+\frac{1}{2}[\xi_t\wedge \xi_t]$ the curvature of our connection.  Let $X$ denote a unit horizontal direction with respect to the $S^2_{1/2}$ base, and let $W$ denote a unit vertical direction with respect to the $S^3$ action.  Then using Proposition \ref{prop:Rictotgeo} we can compute the Ricci curvature of this metric as
\begin{align}
	\Ric_{h_t}[W,W] &= \frac{2}{\lambda^2_t}+\frac{\lambda_t^2}{4}|\Omega_t|^2\, ,\notag\\
	\Ric_{h_t}[W,X] &=\lambda_t\text{div}_{S^2} \Omega_t[X]\, ,\notag\\
	\Ric_{h_t}[X,X] &=4-\lambda_t^2|\Omega_t[X]|^2\, .
\end{align}

It follows for $\lambda_t\leq\lambda_t(\Omega_t)$ that $\Ric_{h_t}>1$ is uniformly positive for $t\in [1/2,3/4]$.

\vspace{.3cm}

\subsection{Step 3:  Trivializing the Geometry}\label{ss:equivariant_metric:step3}

For $t=3/4$ we now have that $\Omega_{3/4}=0$ vanishes.  In particular we have isometrically that
\begin{align}
N_{3/4}\equiv S^2_{1/2}\times S^3_{\lambda_{3/4}}\, .	
\end{align}

It follows that over the range $t\in [3/4,1]$ we may increase $\lambda_t$ until $\lambda_t \equiv 1$ for $t$ in a neighborhood of $1$, at which point we have that $N_1\equiv S^2_{1/2}\times S^3_1$, finishing the proof of Lemma \ref{lemma:changeconnectionX}. $\qed$

\vspace{.3cm}

\subsection{Step 4: Constructing $(S^3\times S^3,g_t)$ with $\Ric_{g_t}\geq 0$}\label{ss:equivariant_metric:step4}

We have now built a family of geometries $(N,h_t)$ with $\Ric_{h_t}>0$ which begin at $N_0=S^1\backslash (S^3\times S^3,g_{0})$ and end at $N_1=S^2\times S^3$ isometrically.  Further, we have explicitly built for $t\in [0,1/2]$ a family of metrics $(S^3\times S^3,g_t)$ which are Riemannian $S^1$-principal bundles over $(N_t,h_t)$ with totally geodesic fibers of constant length $2\pi f_t$.  The $S^1$-connections $\eta_t$ of these bundles are both Yang-Mills and in Coulomb gauge.  That is, the curvature form $\omega_t=d\eta_t$ and $\eta_t$ both have vanishing horizontal divergence on $S^3\times S^3$.  For $\omega_t$ this is equivalent to asking that the curvature $2$-form on $N_t$ be divergence free, which is itself equivalent to asking that $\omega_t$ be Hodge harmonic.\\ 

In the next Step our goal is to extend this construction to a smoothly varying family of Riemannian metrics $g_t$ for $t\in [1/2,1]$ with nonnegative Ricci curvature.  The metrics $(S^3\times S^3,g_t)$ will also be Riemannian $S^1$-principal bundles over $N_t$ with respect to connections $\eta_t$ and with totally geodesic $S^1$ fibers of constant length $f_t$.  In this Step of the construction we will choose the connections $\eta_t$ uniquely so that they are Yang-Mills and in Coulomb Gauge.  We will see this is sufficient to force $\Ric_{g_t}\geq 0$.  In the next Step we will allow the warping function $f_t$ to not be constant and vary as a function of $N_t$ in order to push the Ricci curvature to become positive.\\

Now let $[\omega_t]=[\omega_0]\in H^2(N_t)$ be the cohomology class associated to the $(1,k)$-circle bundle $S^3\times S^3\to N_t$.  Note that $H^2(N_t)=\dZ$ as we already understand from Lemma \ref{lemma:Xtrivialbundle} that $N_t$ is diffeomorphic to $S^2\times S^3$, and that $[\omega_t]$ is the generating class of this cohomology.  We can view $[\omega_t]$ as the deRham cohomology class generated by the curvature of any connection of this bundle.\\

For each $t\in [0,1]$ let $\omega_t$ be the unique representative of $[\omega_0]$ which is divergence free with respect to the geometry of $N_t$.  That is, let $\omega_t$ be the unique Hodge-harmonic representative.  As this class is unique it follows that for $t\in [0,1/2]$ this choice agrees with our original construction of the curvature, and it smoothly extends this choice to all $t\in [0,1]$.  Thanks to the invariance of $h_t$ by the right action of $S^3$ on $N_t$, we automatically have that $\omega_t$ is also invariant by the right action of $S^3$ on $N_t$.\\

Now for each such $\omega_t$ there is a connection $[\eta_t]\in \Omega^1(S^3\times S^3)$ whose curvature $d[\eta_t] = \pi^*\omega_t$ is equal to our enforced curvature choice for our $(1,k)$-bundle.  We write $[\eta_t]$ to represent that the connection is not uniquely defined by this condition.  It is only well defined up to the addition of $d\rho_t$ where $\rho_t$ is a $S^1$-invariant function on $S^3\times S^3$.  That is, $\rho_t$ is a function on $N_t$.  In order to pick a unique $\eta_t\in [\eta_t]$ from this class we will ask that it minimizes
\begin{align}
	\eta_t = \arg\min_{\zeta_t\in [\eta_t]}\int_{S^3\times S^3} |\zeta_t|^2_t\, .
\end{align}
It is classical that this minimization exists, and indeed there is a unique solution.  The Euler-Lagrange equation for this minimization is given by the Coulomb gauge condition
\begin{align}
	\text{div}_{N_t}\eta_t =\text{div}_{g_t}\eta_t = 0\, .
\end{align}

A Riemannian metric $(S^3\times S^3,g_t)$ is now well defined by the metrics $(N,h_t)$, the family of connections $\eta_t$, and the fiber size $f_t\in \dR$ of the circle fibers.  If we let $T$ represent a unit vertical direction with respect to the $S^1$ action and $H$ a unit horizontal direction, then we can compute the Ricci curvature
\begin{align}\label{e:equivariant_mapping:step4:1}
	\Ric_{g_t}(T,T) &= \frac{f_t^2}{4}|\omega_t|^2\, ,\notag\\
	\Ric_{g_t}(T,H) &= \text{div}_{N_t}\omega_t[H] = 0\, ,\notag\\
	\Ric_{g_t}(H,H) &= \Ric_{h_t}[H,H]-\frac{f_t^2}{2}|\omega_t[H]|^2\, .
\end{align}

Let us now make some observations on the above computations.  Recall that for $t$ near $1$ we have that $N_t\equiv S^2\times S^3$.  For quantitative sake let us say this holds for $t\geq t_0$ with $t_0<1$.  Recall that for all $t$ we have $\Ric_{h_t}>0$, and so in particular for $t\leq t_0$ we see that if $f_t\leq f_t(h_t,\omega_t)$ then $\Ric_{g_t}\geq 0$ and $\Ric_{g_t}(H,H)>\tau$ for some $\tau>0$.  For $t\geq t_0$, where $N_t\equiv S^2_{1/2}\times S^3_1$, we see that $\omega_t$ is precisely the volume form on the $S^2$ factor, as this is the unique Hodge-harmonic representative of the cohomology class generated by $S^2$.  We further have that $\eta_t$ is the canonical Hopf connection on the first $S^3$ factor, as it is the unique Coulumb gauge connection representing this curvature.  In particular, in the range $t\in [t_0,1]$ we can increase $f_t$ until $f_t\equiv 1$ for $t$ near $1$.  We then get that $g_1\equiv S^3_1\times S^3_1$. \\

Observe however that while $g_1$ is isometrically equivalent to $S^3_1\times S^3_1$, the $S^1$ action coming from the $S^1$-principal bundle structure $S^3\times S^3\to N_1$ is now precisely the $(1,0)$ action, as our $S^1$ bundle was the one coming from the Hopf fiber of the first factor.  We have therefore nearly completed the proof of Theorem \ref{t:equivariant_mapping_class:2}.\\

\subsection{Step 5:  Constructing $(S^3\times S^3,g_t)$ with $\Ric_{g_t}> 0$}\label{ss:equivariant_metric:step5}

We have at this stage built a family of metrics $(S^3\times S^3,g_t)$ with isometric $S^1$ actions which begin and end isometrically at $S^3_1\times S^3_1$ with the $(1,k)$ and $(1,0)$ actions, respectively.  Further, these metrics all satisfy $\Ric_{g_t}\geq 0$. \\

 We have not yet completed the proof however.  It follows from \eqref{e:equivariant_mapping:step4:1} that we can choose $f_t$ sufficiently small that 
 \begin{align}
 	\Ric_{g_t}[H,H]>\tau>0\, ,
 \end{align}
is uniformly positive for all $t\in [1/2,1]$.  However, if we have $|\omega_t|^2=0$ at some point, then $\Ric_{g_t}(T,T) = \frac{f_t^2}{4}|\omega_t|^2=0$ at this point.  Were this to occur, then we would only have $\Ric_{g_t}\geq 0$, and it will be important in the applications that we get strict positivity.\\

In order to handle this, we will perturb our $S^1$ warping functions $f_t$. Currently $f_t$ is spatially constant, and our perturbation will be by a small function of $N_t$, cf. with \cite{GilkeyParkTuschmann}.  Let us write our new warping function as
\begin{align}
	\tilde f_t = f_t+\epsilon_t h_t\, ,
\end{align}
where $h_t:N\to \dR$ will be a smoothly varying collection of smooth functions which vanish for $t$ near $\frac{1}{2}$ and $1$, and $\epsilon_t$ will be sufficiently small constants depending smoothly on time.\\

In order to pick $h_t:N_t\to \dR$ let us begin with several observations.  First as $(N,h_t)$ is invariant under the right $S^3$ action, we have that $|\omega_t|^2$ is invariant under this right action as well.  In particular, we can view $|\omega_t|^2$ as a function of $N/S^3=S^2_{1/2}$.  We will similarly choose $h_t$ to be an $S^3$ invariant function, which is to say a function of $S^2$.\\

As a second observation, let us point out that $\omega_t$ is non-trivial in cohomology, hence cannot be flat and so we have
\begin{align}\label{e:equivariant_mapping:step5:1}
	\fint_{N}|\omega_t|^2
	\ge 
	c_0 \, ,
	\quad \text{for every $t\in [0,1]$}\, ,
\end{align}
for some $c_0>0$.\\

Let us now consider a smooth cutoff function $\phi:\dR\to [-1,0]$ with $\phi(s)=-1$ if $s\leq 10^{-2}c_0$ and $\phi(s)=0$ if $s\geq 10^{-1}c_0$.  Let us define $h_t:S^2\to \dR$ as the solution of 
\begin{align}
	\Delta_{S^2} h_t \equiv \phi(|\omega_t|^2)-\fint_{S^2}\phi(|\omega_t|^2)\, .
\end{align}  
If we take $\fint h_t = 0$ then $h_t$ is uniquely defined, smooth, and smoothly varying in $t$.  As a first observation note that if $|\omega_t|^2>10^{-1}c_0$ on $N$, then we have that $\Delta h_t\equiv 0$ identically vanishes.  Let us also see that $\Delta h_t$ is uniformly negative when $|\omega_t|$ is small.  To begin, as everything is smooth let us define $M$ so that 
\begin{align}\label{e:equivariant_mapping:step5:2}
	|\nabla h_t|+|\nabla^2 h_t|+|\omega_t|+|\nabla\omega_t|\leq M\, ,
\end{align}
uniformly for all $t\in [1/2,1]$ .  Using \eqref{e:equivariant_mapping:step5:1} we have that there exists at least one point with $|\omega_t|^2(p) \geq c_0$, and so by \eqref{e:equivariant_mapping:step5:2} we have that $|\omega_t|^2 > 10^{-1}c_0$ and hence $\phi(|\omega_t|^2)=0$ on $B_{c_0(20M^2)^{-1}}(p)$.  Consequently, we have that
\begin{align}
	\fint \phi(|\omega_t|^2)>-1 + c_0^2(20 M^2)^{-2}\, .
\end{align}
It follows that if $x\in \{|\omega_t|^2<10^{-2}c_0\}$ then
\begin{align}\label{e:equivariant_mapping:step5:3}
	\Delta h_t(x)<-c_0^2(20 M^2)^{-2}\, .
\end{align}

Now let $\tilde f_t = f_t+\epsilon_t h_t$ be our warping function and to begin let $\epsilon_t<(2M)^{-1}f_t$, so that $\frac{1}{2}f_t<|\tilde f_t|<2f_t$.  We will further decrease $\epsilon_t$ later. Recall that we previously chose $f_t$ small enough in Step 4, in order to guarantee $\Ric_{g_t}\ge 0$ and $\Ric_{g_t}>\tau>0$ in the horizontal directions when we set $\epsilon_t=0$.
 
We will use Proposition \ref{prop:S1bundlewarped} to compute the Ricci curvature on $(S^3\times S^3,g_t)$.  Let $T$ be a unit direction in the $S^1$ fiber direction and let $H$ represent any unit direction perpendicular to $T$.  Let us split our computation into two regions.  If we are on the region $\{|\omega_t|^2>10^{-2}c_0\}$ then we can estimate 
\begin{align}
	\Ric(T,T)&\geq \frac{f_t^2c_0}{400} - \frac{\epsilon_t}{f_t}\, ,\notag\\
	|\Ric(T,H)|&\leq \frac{3}{2} M^2 \epsilon_t\, ,\notag\\
	\Ric(H,H)&\geq \Ric_{h_t}(H,H)-\frac{f_t^2}{2} M^2-\frac{4M\epsilon_t}{f_t}\, .
\end{align}

Here we used that Laplacians and Hessians of $S^3$-invariant functions on $N$ can be identified with the corresponding objects in the base space $S^2$, since the $S^3$ fibers are totally geodesic.  Then we can choose $f_t<f_t(h_t)$ and $\epsilon_t<\epsilon_t(f_t,M,c_0,|\Ric|_{h_t})$ and obtain $\Ric_{g_t}>0$ in the region $\{|\omega_t|^2>10^{-2}c_0\}$.\\

On the other hand, let us consider the region $\{|\omega_t|^2<10^{-2}c_0\}$, then we can use \eqref{e:equivariant_mapping:step5:3} to estimate
\begin{align}
	\Ric(T,T)&\geq \frac{\epsilon_tc_0^2(20 M^2)^{-2}}{2f_t}\, ,\notag\\
	|\Ric(T,H)|&\leq \frac{3}{2} M^2 \epsilon_t\, ,\notag\\
	\Ric(H,H)&\geq \Ric_{h_t}(H,H)-\frac{f_t^2}{2} M^2-\frac{4M\epsilon_t}{f_t} \, .
\end{align}
It is the estimate of the first term which has changed.
If we now further assume $f_t<f_t(h_t)$ and $\epsilon_t<\epsilon_t(f_t,M,c_0,|\Ric|_{h_t})$, then we can again conclude that $\Ric_{g_t}>0$.  

\medskip

\medskip

\subsection{Finishing the Proof of Theorem \ref{t:equivariant_mapping_class:2}}\label{ss:equivariant_metric:finish_proof}

We have essentially finished the proof of Theorem \ref{t:equivariant_mapping_class:2} at this stage, however let us put all the ingredients together to see that this is the case.\\

We already noticed that there is an $S^3$-equivariant isometry $\Phi:(N,h_1)\to (S^2\times S^3,\frac{1}{4}g_{S^2}+g_{S^3})$ with respect to the respective right $S^3$ actions.  Moreover, consider the $S^1$ bundles $\pi_{(1,k)}:S^3\times S^3\to N$ and $\pi_{(1,0)}:S^3\times S^3\to S^2\times S^3$.  Then the $S^1$ bundles $\pi_{(1,k)}:S^3\times S^3\to N$ and $\Phi^*\pi_{(1,0)}:S^3\times S^3\to N$ are necessarily isomorphic as $S^1$ bundles over $N$ as they arise from the same cohomology class.  That is, we can find an $S^1$-equivariant diffeomorphism 
\begin{align}
	\hat{\Phi}:S^3\times S^3\to S^3\times S^3\, ,
\end{align}
with $\hat\Phi(\theta\cdot_{(1,k)}(s_1,s_2)) = \theta\cdot_{(1,0)}\hat\Phi(s_1,s_2)$ whose induced mapping on the quotients is given by the isometry $\Phi:(N,h_1)\to S^2_{1/2}\times S^3_1$.\\

We claim that, up to composition with a gauge transformation, $\hat{\Phi}$ is an $S^1$-equivariant isometry between $(S^3\times S^3,g_1)$ with the $(1,k)$-Hopf action, and $(S^3\times S^3,g_{S^3}+g_{S^3})$ with the $(1,0)$-Hopf action.\\

Begin by observing that as the induced quotient map $\Phi$ is an isometry, if we denote by $\eta_1$ and $\eta_c$ the $S^1$-principal connections of $\pi_{(1,k)}:S^3\times S^3\to N_1$ and $\pi_{(1,0)}:S^3\times S^3\to S^2\times S^3$ respectively, it holds
\begin{equation}
	d\left(\hat{\Phi}^*\eta_c\right)=d \eta_1\, ,
\end{equation}
since both connections are Yang-Mills, and the curvature form being Hodge-harmonic is a uniquely defined condition in its cohomology class.  As $S^3\times S^3$ has trivial first cohomology, there exists a smooth function $\phi:S^3\times S^3\to \mathbb{R}$ such that 
\begin{equation}
	\hat{\Phi}^*\eta_c=\eta_1 + d\phi\, ,
\end{equation}
and by right invariance we can assume that $\phi=\phi'\circ \pi_{(1,k)}$ for some smooth function $\phi':N\to \mathbb{R}$.  In particular, and with a slight abuse of notation, after composing with the gauge transformation induced by $\phi'$ we can assume that 
\begin{equation}
	\hat{\Phi}^*\eta_c=\eta_1\, .
\end{equation}
Therefore, $\hat{\Phi}$ is an $S^1$-equivariant diffeomorphism between principal $S^1$-bundles with totally geodesic and isometric $S^1$-fibers, it induces and isometry between the base spaces and maps one connection form to the other. Hence it is an $S^1$-equivariant isometry, and this finishes the proof of Theorem \ref{t:equivariant_mapping_class:2}. $\qed$\\

\subsection{Explicit Diffeomorphism for $k=1$}\label{ss:equivariant_metric:diffeo}

Let us end by addressing Remark \ref{r:equivariant_mapping_diffeo} and show that for $k=1$ we can explicitly choose the diffeomorphism $\phi$ in Theorem \ref{t:equivariant_mapping_class:2} by $\phi=\phi_1$ from Example \ref{ss:equivariant_metric:examples}.  \\

Our first observation is that $\phi_1$, besides pushing forward the $(1,1)$-action to the $(1,0)$-action is also equivariant with respect to the right $S^3$-action. Hence it induces an $S^3$-principal bundles isomorphism $\Phi_1:N\to S^2\times S^3$. It turns out that there is an explicit expression for $\Phi_1$, namely
\begin{equation}
\Phi_1([s_1,s_2])=(\pi_{\rm Hopf}(s_1),s_1^{-1} \cdot s_2)\, .
\end{equation}
By using $\Phi_1$ in place of $\Phi$ in Step 2, and arguing as Section \ref{ss:equivariant_metric:finish_proof}, we get an isometry $\hat \Phi_1 : (S^3\times S^3, g_0) \to (S^3 \times S^3, g_1)$ which coincides with $\phi_1$ up to gauge transformation. In particular, $\hat \Phi_1$ and $\phi_1$ are isotopic.

\vspace{.5cm}

\section{Step 2: Equivariant Twisting}\label{s:step2_proof}

Our main goal in this Section is to prove Proposition \ref{p:step2}, which is Step 2 of the construction, which we restate for the convenience of the reader:

\begin{proposition}[Step 2: Twisting the Action]\label{p:step2:2}
	Let $1>\epsilon,\hat \epsilon,\delta>0$ with $k\in \dZ$.  Then there exist $\hat \delta(\epsilon,\hat \epsilon,\delta,k)>0$ with $R(\epsilon,\hat \epsilon,\delta,k)>1$ and a metric space $X$ with an isometric and free $S^1$ action such that
	\begin{enumerate}
		\item $X$ is smooth away from a single three sphere $S^3_\delta\times \{p\}\in X$ with $\Ric _X\geq 0$.
		\item There exists $B_{10^{-3}}(p)\subseteq U\subseteq B_{10^{-1}}(p)\subseteq X$ which is isometric to $S^3_\delta\times B_{10^{-2}}(0)\subseteq S^3_\delta\times C(S^3_{1-\epsilon})$ , and under this isometry the $S^1$ action on $U$ identifies with the $(1,k)$-Hopf action. 
		\item There exists $B_{10^{-1}R}(p)\subseteq \hat U\subseteq B_{10R}(p)\subseteq X$ s.t. $X\setminus \hat U$ is isometric to $S^3_{\hat \delta R}\times A_{R,\infty}(0)\subseteq S^3_{\hat \delta R}\times C(S^3_{1-\hat \epsilon})$, and under this isometry the $S^1$ action on $X\setminus \hat U$ identifies with the $(1,0)$-Hopf action. 
	\end{enumerate}
\end{proposition}

We will build the space $X$ in pieces throughout this Section.  Let us give a rough description of the steps involved and break down the role of each subsection.  The starting point for our construction is to take $X_0 \equiv S^3_\delta\times C(S^3_{1-\epsilon})$, which by condition (2) in Proposition \ref{p:step2:2} is the beginning of our $X$.  In Section \ref{ss:step2_proof:X1} we will construct $X_1$ by first bending the $\dR^4=C(S^3_{1-\epsilon})$ factor down to a sharper cone $C(S^3_{1/4})$ .  This will give us the extra curvature we need in Sections \ref{ss:step2_proof:X2} and \ref{ss:step2_proof:X3} in order to construct $X_3$.  The goal of $X_2$ and $X_3$ will be to take the constant $S^3_\delta$ factor and lift it to a linearly growing factor.  That is, outside a compact subset $X_3$ will be isometric to $C(S^3_{\delta_3}\times S^3_{1/8})$.  \\

In Section \ref{ss:step2_proof:refined_twisting} we will prove a refinement of Theorem \ref{t:equivariant_mapping_class:2}.  Namely we will use Theorem \ref{t:equivariant_mapping_class:2} in order to construct a family of $S^1$ invariant metrics $(S^3\times S^3,g_t)$ which begin at $S^3_{\delta_3}\times S^3_{1/8}$ with the $(1,k)$-action, end at $S^3_{\delta_3}\times S^3_{1/8}$ with the $(1,0)$ action, satisfy $\Ric_t>6$ and which have constant volume form.  We will apply this in Section \ref{ss:step2_proof:X4} to build $X_4$, which will perform our twisting and end so that outside a compact subset $X_4$ is isometric to an annulus in $C(S^3_{\delta_4}\times S^3_{1/16})$ with the new action. \\

We still need to take $X_4$ back to the $S^3_{}\times \dR^4$ geometry, which will require several steps.  In Section \ref{ss:step2_proof:X5} we will construct $X_5$, which outside of a compact subset will turn the linear growth of the second $S^3_{\delta_4}$ factor into a slow polynomial growth.  That slow growth will be useful in  Section \ref{ss:step2_proof:X6} to construct $X_6$, which will take the $C(S^3_{1/16})$ factor and increase the size of the cross section until we arrive at $C(S^3_{1-\hat\epsilon})$, which is isometrically very close to $\dR^4$.  Finally in Section \ref{ss:step2_proof:X7} we will construct $X=X_7$, which will take the first sphere, which is still growing at a slow polynomial rate, and level it out to constant size $S^3_{\hat\delta}$, which will finish the proof of Proposition \ref{p:step2:2}.

\vspace{.3cm}

\subsection{Constructing $X_1$}\label{ss:step2_proof:X1}

Let us begin with $X_0 \equiv S^3_{\delta}\times C(S^3_{1-\epsilon})$, which geometrically has a metric which may be written 
\begin{align}
	g_0 \equiv dr^2 +\delta^2 g_{S^3}+ (1-\epsilon)^2 r^2 g_{S^3} \, .
\end{align}

Our first step of the construction is to shrink down the size of the cone cross section.  The purpose of this is to increase the curvature sufficiently so that in the second and third steps we can turn our constant $S^3$ factor into a cone factor as well.\\

Let $U_0\equiv \{r\leq 1\}$, then we will define the metric $g_1$ on $X_1\setminus U_0$ through the ansatz

\begin{align}
	g_1 \equiv dr^2 + \delta^2 g_{S^3}+ h(r)^2 g_{S^3}\, .
\end{align}

If we let $a,b,c$ denote the directions on the first $S^3$ factor and $i,j,k$ on the second $S^3$ factor (the one that we are currently viewing as a cross section), then we can compute the nonzero terms of the Ricci curvature of this ansatz as

\begin{align}\label{e:step2:X1:Ricci}
	&\Ric_{rr} = -3\frac{h''}{h}\, ,\notag\\
	&\Ric_{aa} = \frac{2}{\delta^2}\, ,\notag\\
	&\Ric_{ii} = 2 \frac{1-(h')^2}{h^2}-\frac{h''}{h}\, .
\end{align}

Now let us define $h(r)$ so that

\begin{align}
	h(r) \equiv \begin{cases}
		(1-\epsilon)r & \text{ if } r\leq 10^{-1}\, ,\notag\\
		h''<0 & \text{ if } 10^{-1}\leq r\leq 10 \, ,\notag\\
		(1-\epsilon)+(r-1)/4 & \text{ if } r\geq 10^{}\, .
	\end{cases}
\end{align}

We can build a function as above by smoothing out $h(r)=\min\{(1-\epsilon)r,(1-\epsilon)+(r-1)/4 \}$ near $r=1$, the intersection point of the two lines.  Observe that we always have $|h'|\leq 1-\epsilon$ as $h''<0$, and so we can estimate the Ricci curvature 
\begin{align}\label{e:step2:X1:Ricci_estimate}
	&\Ric_{rr} \geq 0\, ,\notag\\
	&\Ric_{aa} = \frac{2}{\delta^2}\geq 0,\notag\\
	&\Ric_{ii} \, \geq \, \frac{2\epsilon-\epsilon^2}{h^2}\geq 0\, \, .
\end{align}

Note that outside the region $\{r\leq 10\}$ we have that $X_1$ is isometric to $S^3_\delta\times C(S^3_{1/4})$.  Let us change coordinates $r\to r-1+4(1-\epsilon)$.  Set $R_1\equiv 20$, and $U_1\equiv \{r\leq R_1\}$ in the new coordinates.
Then $X_1\setminus U_1$ is isometric to
\begin{align}
	g_1 \equiv dr^2 + \delta^2 g_{S^3}+ (r/4)^2 g_{S^3} \, .
\end{align}

\vspace{.3cm}

\subsection{Constructing $X_2$}\label{ss:step2_proof:X2}

Let us now define $X_2$ by changing the metric on $X_1$ in the region $X_1\setminus U_1$.  Our goal over the next two steps will be to turn the constant $S^3_\delta$ factor into a cone factor.  Let us begin by defining the metric $g_2$ on $X_2\setminus U_1$ through the ansatz

\begin{align}
	g_2 \equiv dr^2 + f(r)^2 g_{S^3}+ h(r)^2 g_{S^3} \, .
\end{align}

We can compute the nonzero terms of the Ricci curvature of this ansatz as

\begin{align}\label{e:step2:X2:Ricci}
	&\Ric_{rr} = -3\frac{h''}{h}-3\frac{f''}{f}\, ,\notag\\
		&\Ric_{aa} = \frac{2}{f^2}-\frac{f''}{f}-\frac{f'}{f}\Big(\,2\frac{f'}{f}+3\frac{h'}{h}\Big)\, ,\notag \\
	&\Ric_{ii} = 2\frac{1-(h')^2}{h^2}-\frac{h''}{h}-3\frac{h'}{h}\frac{f'}{f}\, .
\end{align}

Let us denote $R_2\equiv 10^4$, and let $\delta_2\leq \delta_2(\delta)$ be a constant we will choose in the next subsections.  Let us  choose a smooth $f(r)$ so that
\begin{align}
	f(r) \equiv \begin{cases}
		\delta & \text{ if } r\leq 10^{-1}R_2\, ,\notag\\
		0<R_2\,f''< \delta_2 & \text{ if } 10^{-1}R_2\leq r\leq 10^{}R_2\, ,\notag\\
		\delta+\delta_2 (r-R_2) & \text{ if } r\geq 10^{}R_2\, .\notag
	\end{cases}
\end{align}
We can build $f(r)$ by smoothing out $f(r)=\min\{\delta, \delta+\delta_2 (r-R_2)\}$ near $r=R_2$, the intersection point.\\

Let $h(r)$ be a smooth function such that
\begin{align}
	h(r) \equiv \begin{cases}
		r/4 & \text{ if } r\leq 10^{-2}R_2\, ,\notag\\
		h''<0 & \text{ if } 10^{-2}R_2\leq r\leq 10^4 R_2\, ,\notag\\
		R_2\, h''<-10^{-4} & \text{ if } 10^{-1}R_2\leq r\leq 10^{}R_2\, ,\notag\\
		R_2/4+(r-R_2)/6 & \text{ if } r\geq 10^2R_2\, .\notag\\
	\end{cases}
\end{align}
We can build such a function by similarly smoothing out $h(r)=\min\{r/4, R_2/4+(r-R_2)/6\}$ near $r=R_2$.\\

  Observe that the ending metric will have both sphere factors end with a linear growth, but the center of cones will be different for the first and second spheres.  Our construction of $X_3$ will fix this, for now let us compute the Ricci tensor of the above.\\  

We will choose $\delta_2$ in a future subsection, however we will impose the restriction now that
\begin{equation}
	\delta_2\leq 10^{-9}\delta = 10^{-6}\frac{\delta}{R_2}\, .
\end{equation}
Note that away from the interval $r\in [10^{-1}R_2,10^{}R_2]$ the Ricci curvature is nonnegative by similar computations as the last subsection.  The interval $r\in [10^{-1}R_2,10^{}R_2]$ is a little more complicated. We use that for every $r\in [10^{-2}R_2, 10R_2]$ it holds
	\begin{equation}
		\begin{split}
			&f'(r) \le 10 \delta_2
			\\
			& \delta \le f(r) \le 2\delta
			\\
			&\frac{R_2}{4} \le h(r) \le \frac{7}{4} R_2\, ,
		\end{split}
	\end{equation}
	to conclude that

\begin{align}
\Ric_{rr}&\geq \frac{3}{R_2}\Big(\frac{4\cdot 10^{-4}}{7\cdot  R_2}-\frac{\delta_2}{\delta}\Big)> 0\, ,\notag\\
\Ric_{aa}&\geq 
\frac{1}{f}\Big(\frac{1}{\delta} - \frac{\delta_2}{R_2} - 200 \frac{\delta_2^2}{\delta} - 60 \frac{\delta_2}{R_2} \Big)
>0\, ,\notag\\
\Ric_{ii}&\geq \frac{1}{h}\Big( \frac{15}{14 \cdot R_2} + \frac{10^{-4}}{R_2} - \frac{15}{2}\cdot \frac{\delta_2}{\delta} \Big)>0\, .
\end{align}

Note that if we define $U_2\equiv \{r\leq 10R_2\}$ then $X_2\setminus U_2$ is {\it almost isometric} to an annulus in $C(S^3_{1/6}\times S^3_{\delta_2})$.  Precisely, we have:

\begin{align}
	g_2 \equiv dr^2 + (\delta+\delta_2 (r-R_2))^2 g_{S^3}+ (R_2/4+(r-R_2)/6)^2 g_{S^3} \, ,
\end{align}

\vspace{.3cm}

\subsection{Constructing $X_3$}\label{ss:step2_proof:X3}

Our space $X_2$ has ended so that it looks like a cone over each of the sphere factors, however the centers of those cone points are not the same.  Specifically let us write the metric in the form
\begin{align}
	g_2 = dr^2 + \delta_2^2(r+c_2)^2 g_{S^3}+ (R_2/4+(r-R_2)/6)^2 g_{S^3} \, ,
\end{align}
where $c_2\equiv \delta_2^{-1}\delta-R_2>>R_2$ .  We will want to change the metric in this subsection so that they are both metric cones with respect to the center $c_2$.  We will change the geometry on $X_2\setminus U_2$ through the ansatz 

\begin{align}
	g_3 \equiv dr^2 + \delta_2^2(r+c_2)^2 g_{S^3}+ h(r)^2 g_{S^3} \, .
\end{align}

We can compute the nonzero terms of the Ricci curvature of this ansatz as
\begin{align}\label{e:step2:X3:Ricci}
	&\Ric_{rr} = -3\frac{h''}{h}\, ,\notag\\
		&\Ric_{aa} = \frac{2}{\delta_2^2(r+c_2)^2}-\frac{1}{r+c_2}\Big(\,\frac{2}{r+c_2}+3\frac{h'}{h}\Big)\, , \notag\\
	&\Ric_{ii} = 2\frac{1-(h')^2}{h^2}-\frac{h''}{h}-\frac{3}{r+c_2}\frac{h'}{h}\, .	
\end{align}

Let $R_3\equiv 3c_2-2R_2=R_3(\delta,\delta_2)$ be the point where the lines $R_2/4+(r-R_2)/6$ and $(r+c_2)/8$ intersect.  Then we will choose $h(r)$ as a smooth function which satisfies

\begin{align}
	h(r) \equiv \begin{cases}
		R_2/4+(r-R_2)/6 & \text{ if } r\leq 10^{-1}R_3\, ,\notag\\
		h''<0 & \text{ if } 10^{-1}R_3\leq r\leq 10^{}R_3\, ,\notag\\
		(r+c_2)/8 & \text{ if } r\geq 10 R_3\, .\notag\\
	\end{cases}
\end{align}

Note that we may pick such an $h$ by smoothing out $h(r) = \min\{R_2/4+(r-R_2)/6, (r+c_2)/8\}$, so that we may also insist that $\frac{1}{2}\min\{r/6, (r+c_2)/8\}\leq h(r)\leq \min\{r/6, (r+c_2)/8\}$.  As $h''<0$ we also have that $|h'|\leq 1/6$, and so we can compute  
\begin{align}\label{e:step2:X3:Ricci_compute}
	&\Ric_{rr} \geq 0\, ,\notag\\
	&\Ric_{aa} =
	\frac{1}{r + c_2}\Big( \frac{2}{\delta_2^2(r+c_2)}-
	\frac{2}{r+c_2}  - \frac{6}{r}   \Big)
	\geq 0\, ,\notag \\
	&\Ric_{ii} \geq 2\cdot\frac{35}{36}\frac{1}{h^2}-\frac{1}{16}\frac{1}{h^2}\geq 0\, .
\end{align}

If we define $\delta_3\equiv \delta_2$ and consider the domain $U_3\equiv \{r\leq 10R_3\}$, then we have that $X_3\setminus U_3$ is isometric to $C(S^3_{\delta_3}\times S^3_{1/8})$.  After shifting $r\to r+c_2$ this metric is written in coordinates as
\begin{align}
	g_3 \equiv dr^2 + (\delta_3 r)^2 g_{S^3}+ (r/8)^2 g_{S^3}\, .
\end{align}
Note that in these coordinates we can identify $U_3 = \{r\leq 10R_3+c_2\}\subseteq \{r\leq 11 R_3\}$ .

\vspace{.3cm}

\subsection{Refinement of Theorem \ref{t:equivariant_mapping_class:2}}\label{ss:step2_proof:refined_twisting}

Theorem \ref{t:equivariant_mapping_class:2} from Section \ref{s:equivariant_Ricci} proved the existence of a family of $S^1$-invariant metrics $(S^3\times S^3,g_t)$ which all have positive Ricci, begin and end at the standard metric, but for which the beginning and ending isometric $S^1$ actions are homotopically inequivalent.  In this subsection we would like to construct from this a refinement which will keep control for us various geometric quantities.  This refinement will be directly used in the next subsection to build $X_2$.

\begin{lemma}[Refined Equivariant Twisting]\label{l:step2:refined_twisting}
Let $(S^3\times S^3,g_0)$ be a product of two spheres with $g_0=g_{S^3_{\delta}\times S^3_{1/8}}$, and take this space to be equipped with the $(1,k)$-$S^1$ isometric Hopf action.  If $\delta\leq \delta(k)$ then there exist a family of metrics $g_t$, and a diffeomorphism $\phi:S^3\times S^3\to S^3\times S^3$ such that
\begin{enumerate}
	\item The $(1,k)$-$S^1$ action on $S^3\times S^3$ is an isometric action for all $t$ ,
	\item The Ricci curvature $\Ric_t > 6$ is uniformly positive,
	\item The volume form $dv_{g_t}=dv_{g_0}$ is a constant.
	\item $g_1 = \phi^*g_{S^3_{\delta}\times S^3_{1/8}}$ with $\phi\big(\theta \cdot_{(k,{1})}(s_1,s_2)\big)= \theta\cdot_{(1,0)}\phi(s_1,s_2)$ .
\end{enumerate}	
\end{lemma}
The construction of the above is in several steps.  To begin, let $\hat g_r$ be the metric from Theorem \ref{t:equivariant_mapping_class:2}, reparametrized so that $\hat g_{1/3} = g_{S^3\times S^3}$ is the standard metric on $S^3\times S^3$ and $\hat g_{2/3} = \varphi^*g_{S^3\times S^3}$ is the pullback of the standard metric by our required nontrivial diffeomorphism.  Let us first try and normalize this collection.  Namely, let us first consider the family of metrics
\begin{align}
	\tilde g_r = a_r \varphi_r^* \hat g_r\, ,
\end{align}
where 

\begin{equation}
	a_r := \frac{\tilde a}{\Vol(\hat g_r)^{1/6}}\, ,
\end{equation}
$\tilde a>0$ is small enough so that $\Ric_{a_r \hat g_r}>6$ uniformly in $r\in [1/3,2/3]$, and $\varphi_r:S^3\times S^3\to S^3\times S^3$ is a family of diffeomorphisms with $\varphi_{1/3}=Id$.
Observe that, by construction, $\Vol(a_r \hat g_r)$ is constant with respect to $r$.

\medskip

Let us now build the diffeomorphisms $\varphi_r:S^3\times S^3\to S^3\times S^3$.  Consider first the volume forms $\nu_r\equiv dv_{\hat g_r}$.  Recall that our remaining challenge is to force this family to be a constant.  In order to do so we will follow \cite{Moser} in a spirit related to \cite{ColdingNabercones}, with the additional subtlety that we need to work equivariantly.\\  

Even though the family $\nu_r$ is not constant, as the action is isometric we do have that $\nu_r$ is invariant under the $(1,k)$-Hopf action.  Let $\nu \equiv dv_{a_0 \hat g_0}$, which (up to multiple) is simply the standard volume form on $S^3\times S^3$.  Let us write $\nu_r \equiv \rho_r \nu_0$, so note that the function $\rho_r$ is invariant under the $(1,k)$-hopf action.  For each $r$ let us solve on $S^3\times S^3$
\begin{align}
	\Delta f_r+\langle \nabla\ln\rho_r,\nabla f_r\rangle = -\frac{\partial}{\partial r}\ln\rho_r\, .
\end{align}
Observe first that the above is smoothly solvable because $\int \frac{\partial}{\partial r}\ln\rho_r\cdot \rho_r\nu_0 = \int \frac{\partial}{\partial r}\rho_r\nu_0 = \frac{d}{dr}\int \nu_r = 0$.  The solution is a unique up to a constant, so if we assume $\int f_r \rho_r\nu_0=0$ then the solution is uniquely defined.  As the equation and right hand side commute with this action, as the solution is unique we must have that is also invariant by the $(1,k)$-Hopf action.\\

Let us now define the family of diffeomorphisms $\varphi_r:S^3\times S^3\to S^3\times S^3$ by $\varphi_{1/3}=Id$ and
\begin{align}
	\frac{d}{dr}\varphi_r(x) = \nabla f_r(\varphi_r(x))\, \quad \text{for $r\ge 1/3$} \, .
\end{align}
Note that as $\nabla f_r$ is invariant by the $(1,k)$-Hopf action we get that $\varphi_r$ commutes with the $(1,k)$-Hopf action.  We also get that
\begin{align}
	\frac{d}{dr}\varphi^*_r\nu_r = \big(\frac{\partial}{\partial r}\ln\rho_r+\Delta f_r+\langle \nabla\ln\rho_r,\nabla f_r\rangle \big)\nu_r = 0\, .
\end{align}

From this we have built a family of metrics $\tilde g_r$ such that
\begin{align}
&\tilde g_{1/3} = a_{1/3}\,g_{S^3\times S^3}\, ,\notag\\
&\tilde g_{2/3} = a_{1/3}\,\varphi_{2/3}^*\phi^*g_{S^3\times S^3}\equiv a_{1/3}\varphi^*g_{S^3\times S^3}\, ,\notag\\
	&dv_{\tilde g_r} = a_{1/3}^6 dv_{S^3\times S^3}\, ,\notag\\
	&\Ric_{\tilde g_r}>6\, .
\end{align}

We can now finish the construction with one more interpolation.  Let us define the family $g_r$ by
\begin{align}
	g_r=\begin{cases}
	b_r\, g_{S^3}+c_r \,g_{S^3}& \text{ if }0\leq r \leq 1/3\notag\\
	\frac{b_{1/3}}{a_{1/3}}\,\tilde g_r& \text{ if }1/3\leq r \leq 2/3\notag\\
	\varphi^*\Big(d_{r}\, g_{S^3}+e_{r}\, g_{S^3}\Big)& \text{ if }2/3\leq r \leq 1\, .
\end{cases}
\end{align}

Here we have that $b_r,c_r$ are smooth functions such that
\begin{align}
	b_0=1/8\, ,\;\;c_0=\delta\, ,\;\;b_{1/3}=c_{1/3}\, ,\;\;b_r\cdot c_r = const=\delta/8\, ,\notag\\
	d_{2/3}= e_{2/3} = b_{1/3}\, ,\;\;d_{1}=1/8\, ,\;\;e_{1}=\delta\, ,\;\;d_r\cdot e_r = const=\delta/8\, .
\end{align}
Note we need $\frac{b_{1/3}}{a_{1/3}}\leq 1$, so that the Ricci curvature of the above satisfies $\Ric>6$.  This becomes a restriction on $\delta$ given by
\begin{align}
	\delta\leq 8 a_{1/3}^2 \leq \delta(k)\, .
\end{align}
This finishes the construction of Lemma \ref{l:step2:refined_twisting} $\qed$\\
\vspace{.3cm}

\subsection{Constructing $X_4$}\label{ss:step2_proof:X4}

Recall that in the construction of $X_3$ we had a variable $\delta_3=\delta_2$ which had not yet been fixed.  Let us now use Lemma \ref{l:step2:refined_twisting} and Section \ref{ss:step2_proof:X2} to fix $\delta_2=\delta_2(k)$.  Recall there is a compact subset $U_3\subseteq X_3$ such that $X_3\setminus U_3$ is isometric 
\begin{align}
	g_2 = dr^2  + (\delta_3 r)^2 g_{S^3}+ (r/8)^2 g_{S^3}\, .
\end{align}

In order to construct $X_4$ let us modify the metric on $X_3\setminus U_3$ by looking at an ansatz of the form
\begin{align}
		g_4 \equiv dr^2 + h(r)^2g_r\, ,
\end{align}
where $g_r$ will be some family of metrics on $S^3\times S^3$.  Following a line of construction similar to \cite{ColdingNabercones}, if we assume that the volume forms on $g_r$ are constant, then we can compute the Ricci curvature of the above ansatz as

\begin{align}\label{e:step2:X4:Ricci}
	&\Ric_{rr} = -6\frac{h''}{h}-\frac{1}{4}g^{ab}g^{cd}g'_{ac}g'_{bd}\, ,\notag\\
	&\Ric_{ir} = \frac{1}{2}\Big(\partial_a\big(g^{ab}g'_{bi}\big)+\frac{1}{2}(g^{ab})'(\partial_i g_{ab} - g_{ib}g^{pq}\partial_a g_{pq}  )\Big)\, ,\notag\\
	&\Ric_{ij} = \Ric^g_{ij}+h^2\Big(-\frac{h''}{h}-5\big(\frac{h'}{h}\big)^2\Big)g_{ij}+\Big(-\frac{7}{2}\frac{h'}{h}g'_{ij}+\frac{1}{2}g^{ab}g'_{ai}g'_{bj}\Big)\, ,
\end{align}
where $g'_{ab} = \frac{\partial}{\partial r}g_{ab}$ and analogously for $(g^{ab})'$.

Let $10^4R_3<R_4<\infty$ be chosen momentarily.  To define $h(r)$ let us consider the three functions
\begin{align}
	&h_1(r) = r\, ,\notag\\
	&h_2(r) = 10^3 R_3+\Big(1-\frac{1}{4}+\frac{1}{16}\frac{\ln(\ln(15 R_3))}{\ln(\ln(r))} \Big)(r-10^3 R_3)\, ,\notag\\
	&h_3(r) = 10^3 R_3+\Big(1-\frac{1}{4}+\frac{1}{16}\frac{\ln(\ln(15 R_3))}{\ln(\ln(10 R_4))} \Big)(10R_4-10^3 R_3)+(r-10R_4)/2 \, .
\end{align}
Note that $h_1$ and $h_3$ are linear functions, while $h_2$ is almost linear but has a slight amount of convexity added.  Observe that all of the normalizing constants are built so that $h_1$ and $h_2$ intersect at $10^3 R_3$, while $h_2$ and $h_3$ intersect at $10 R_4$.  We will want to build a smooth $h(r)$ of the form
\begin{align}
	h(r) \equiv \begin{cases}
		h_1(r) & \text{ if } r\leq 10^2 R_3\, ,\notag\\
		h''<0 & \text{ if } 10^2 R_3\leq r\leq 10^{2}R_4\, ,\notag\\
		h_2(r)& \text{ if } 10^{4}R_3\leq r\leq R_4\, ,\notag\\
		h_3(r) & \text{ if } r\geq 10^{2}R_4\, .\notag\\
	\end{cases}
\end{align}
We can build such a function $h$ as above by smoothing out the function $h(r) = \min\{h_1(r),h_2(r),h_3(r)\}$ at the relevant intersection points. \\

We define the family of metrics $g_r$ as follows.  Consider first the metrics $\tilde g_t$ defined in Lemma \ref{l:step2:refined_twisting}. In the range $r\leq 10^4 R_3$ let $g_r = \tilde g_0$, and in the range $r\geq R_4$ let $g_r = \tilde g_1$.  In these two ranges the nonnegativity of the Ricci curvature follows by computations analogous to the previous subsections.\\

In the range $r\in [10^4 R_3,R_4]$ we proceed as follows.  Let $t:[10^4R_3, R_4]\to [0,1]$ be a smooth function such that
\begin{align}
	t(r) = \frac{\ln\ln\ln(r) - \ln\ln\ln(10^4 R_3)}{\ln\ln\ln(R_4) - \ln\ln\ln(10^4 R_3)}\, .
\end{align}
Note that $t(r)$ has the property that $t(10^4R_3) = 0$ and $t(R_4) = 1$.  We define $g_r = \tilde g_{t(r)}$.  Note that we have not yet defined $R_4$.  \\

Let $M>0$ be such that
\begin{equation}
	 | \nabla\tilde g|_{\tilde g} + |\tilde g'|_{\tilde g} + | \nabla \tilde g'|_{\tilde g}
	\le M \, , \quad \text{uniformly in $t\in [0,1]$}\, .
\end{equation}
Thus we can estimate
\begin{align}
	|g'|_g + |\nabla g'|_g & \le \frac{1}{r\cdot \ln(r)\cdot \ln\ln(r)} \cdot  \frac{1}{\ln \ln \ln (R_4) - \ln \ln \ln (10^4 R_3)}( |\tilde g'|_{\tilde g} + |\tilde g'|_{\tilde g})\notag
	\\& \le \frac{M}{r\cdot \ln(r)\cdot \ln\ln(r)} \cdot  \frac{1}{\ln \ln \ln (R_4) - \ln \ln \ln (10^4 R_3)} \, .
\end{align}
Notice that $\frac{r}{2} \le h(r) \le r$, and $0\le h'(r) \le 1$ for every $r\in [10^4 R_3, R_4]$, hence
\begin{equation}
	\begin{split}
		&\frac{h'(r)}{h(r)} \le \frac{1}{r}
		\\
		&\frac{h''(r)}{h(r)} \le - \frac{1}{100} \frac{\ln \ln (15R_3)}{r^2 \cdot \ln(r)\cdot (\ln \ln(r))^2} \, ,
	\end{split}
\end{equation}
for every $r\in [10^4 R_3, R_4]$.

From the expression \eqref{e:step2:X4:Ricci}, we can estimate the diagonal Ricci terms
\begin{equation}
	\begin{split}
	\Ric_{rr} \ge  -6\frac{h''}{h} - |g'|_g^2
	 \ge \frac{1}{100} \frac{\ln \ln (15R_3)}{r^2 \cdot \ln(r)\cdot (\ln \ln(r))^2}\, ,
	\end{split}
\end{equation}
\begin{equation}
	\begin{split}
		\Ric_{ii} &\ge 6  - h h'' - 5 (h')^2 - \frac{7}{2r} |g'|_{g} - \frac{1}{2} |g'|_{g}^2
		\ge	\frac{1}{2}\, ,
	\end{split}
\end{equation}
provided $R_4= R_4(M)$ is chosen big enough.

The cross term has the estimate,
\begin{equation}
	\begin{split}
		|\Ric_{ir}| \le 2 (|g'|_{g} + |\nabla g'|_{g})(1 + |\nabla g|_{g})
		&\le 
		\frac{2(M+1)^2}{r\cdot \ln(r)\cdot \ln\ln(r)} \cdot  \frac{1}{\ln \ln \ln (R_4) - \ln \ln \ln (10^4 R_3)}
		\\& 
		\le \frac{1}{r\cdot \ln(r)\cdot \ln\ln(r)}\, ,
	\end{split}
\end{equation}
again, assuming $R_4= R_4(M)$ big enough.\\

Note that the negativity of the cross term $\Ric_{ir}$ dominates the positivity of the radial $\Ric_{rr}$.  However, it is itself dominated by the positivity of the cross section $\Ric_{ii}$.  Consider now a direction $v=a\partial_r +\frac{b}{r}\partial_i$, then we can estimate
\begin{align}
	\Ric_{vv}\geq \frac{a^2}{100} \frac{\ln \ln (15R_3)}{r^2 \cdot \ln(r)\cdot (\ln \ln(r))^2}-2\frac{ab}{r^2\cdot \ln(r)\cdot \ln\ln(r)}+\frac{b^2}{2r^2}\, .
\end{align}
Let us split the above into two cases.  If $b\leq \frac{a}{10^3\ln\ln r}$ then the negative term is dominated by the first term when $r\geq R_4(R_3,M)$.  On the other hand, if $b\geq \frac{a}{10^3\ln\ln r}$ then the negative middle term is dominated by the last term when $r\geq R_4(R_3,M)$.  In any situation we then see for $r\geq R_4(R_3,M)$ that $\Ric>0$ is positive.\\

Note that for $U_4\equiv \{r\leq 10^2 R_4\}$ we have that $X_4\setminus U_4$ is isometric to $C(S^3_{\delta_3/2}\times S^3_{1/16})\equiv C(S^3_{\delta_4}\times S^3_{1/16})$:
\begin{align}
	g_3\equiv dr^2+(\delta_4 r)^2 g_{S^3} +(r/16)^2 g_{S^3}\, .
\end{align}

\vspace{.3cm}

\subsection{Constructing $X_5$}\label{ss:step2_proof:X5}

To construct $X_5$ we want to modify $X_4$ on the neighborhood $X_4\setminus U_4$.  The goal will be to end so that the second $S^3$ factor is growing at a slow polynomial rate.  Our ansatz will be of the form
\begin{align}
	dr^2+f(r)^2 g_{S^3}+(r/16)^2g_{S^3}\, .
\end{align}

The Ricci curvature of this ansatz may be computed 

\begin{align}\label{e:step2:X5:Ricci}
	&\Ric_{rr} = -3\frac{f''}{f}\, ,\notag\\
	&\Ric_{aa} = \frac{2}{f^2}-\frac{f''}{f}-\frac{f'}{f}\Big(\,2\frac{f'}{f}+\frac{3}{r}\Big)\, ,\notag\\
	&\Ric_{ii} = 2\frac{16^2-1}{r^2}-\frac{3}{r}\frac{f'}{f}\, .
\end{align}

For $0<\alpha<<1$, which will be chosen in the next construction, let us consider a function $f$ of the form
\begin{align}
	f(r) \equiv \begin{cases}
		\delta_4 r & \text{ if } r\leq 10R_4\, ,\notag\\
		f''<0 & \text{ if } 10R_4\leq r\leq 10^{3}R_4\, ,\notag\\
		10^2\delta_4R_4\,\Big(\frac{r}{10^2 R_4}\Big)^\alpha & \text{ if } r\geq 10^{3}R_4\, .\notag\\
	\end{cases}
\end{align}
To build such an $f$ one can smooth the function $f\equiv \min\Big\{\delta_4 r,10^2\delta_4R_4\,\Big(\frac{r}{10^2 R_4}\Big)^\alpha \Big\}$.  Note that these two functions agree at $10^2 R_4$ by construction.  If we plug this into \eqref{e:step2:X5:Ricci} we see that the resulting space has $\Ric\geq 0$.\\

If we define $R_5\equiv 10^4R_5$, $c_5\equiv 10^2\delta_4R_4\,\Big(\frac{1}{10^2 R_4}\Big)^\alpha$ and $U_5\equiv \{r\leq R_5\}$, then $X_5\setminus U_5$ is isometric to the warped product
\begin{align}
	g_5 = dr^2 +\big( c_5 r^\alpha \big)^2 g_{S^3}+ (r/16)^2 g_{S^3}\, .
\end{align}

\vspace{.3cm}

\subsection{Constructing $X_6$}\label{ss:step2_proof:X6}

The next step of the construction is dedicated to increasing the size of the cone $S^3$ factor until we are again geometrically close to $\dR^4$.  We will construct $X_6$ by modifying $X_5$ on the neighborhood $X_5\setminus U_5$. The ansatz of our new metric will take the form 
\begin{align}
	g_6 = dr^2 +\big( c_5 r^\alpha \big)^2  g_{S^3}+ h(r)^2 g_{S^3}\, .
\end{align}

The nonzero terms of the Ricci tensor may be computed as
\begin{align}\label{e:step2:X6:Ricci}
	&\Ric_{rr} = -3\frac{h''}{h}+\frac{3\alpha(1-\alpha)}{r^2}\, ,\notag\\
	&\Ric_{aa} = \frac{2}{\big( {c_5} r^\alpha \big)^2 }+\frac{\alpha(1-\alpha)}{r^2}-\frac{\alpha}{r}\Big(\,\frac{2\alpha}{r}+{3}\frac{h'}{h}\Big)\, ,\notag \\
	&\Ric_{ii} = {2}\frac{1-(h')^2}{h^2}-\frac{h''}{h}-\frac{{3}\alpha}{r}\frac{h'}{h}\, .
\end{align}

Recall the construction of $X_5$ depended on the parameter $\alpha>0$, let us now choose $\alpha = 10^{-3}\hat\epsilon$.  Then for $R_6=R_6(\hat\epsilon)$ we can choose a smooth function $h(r)$ so that it satisfies
\begin{align}
	h(r) \equiv \begin{cases}
		r/16 & \text{ if } r\leq 10R_5\, ,\notag\\
		{|h'|<(1-10^{-1}\hat\epsilon),\; |r \, h''|< 10^{-10}\hat\epsilon}
			 & \text{ if } 10R_5\leq r\leq 10^{-1}R_6\, ,\notag\\
		(1-\hat\epsilon)r& \text{ if } r\geq 10^{-1}R_6\, .\notag\\
	\end{cases}
\end{align}

If we plug this into \eqref{e:step2:X5:Ricci} then we see that $\Ric\geq 0$.  If we let $U_6\equiv \{r\leq R_6\}$ then we see that $X_6\setminus U_6$ is isometric to the warped product
\begin{align}
	g_6 = dr^2 +\big( c_5 r^\alpha \big)^2  g_{S^3}+ (1-\hat\epsilon)^2r^2 g_{S^3}\, .
\end{align}

\vspace{.3cm}

\subsection{Constructing $X=X_7$}\label{ss:step2_proof:X7}

We are now in a position to finish the construction of $X$ and prove Proposition \ref{p:step2:2}.  The last step of the construction just needs to flatten out the first $S^3$ factor back into a cross product.  Recall that we have built $X_6$ and that outside of $U_6$ we have that it is isometric to 
\begin{align}
	g_6 = dr^2 +\big( c_5 r^\alpha \big)^2  g_{S^3}+ (1-\hat\epsilon)^2r^2 g_{S^3}\, .
\end{align}

We will look to alter this metric by looking for an ansatz of the form

\begin{align}
	g_7 = dr^2 +f(r)^2  g_{S^3}+ (1-\hat\epsilon)^2r^2 g_{S^3}\, .
\end{align}

The nonzero Ricci curvatures of this ansatz can be computed
\begin{align}\label{e:step2:X7:Ricci}
	&\Ric_{rr} = -3\frac{f''}{f}\, ,\notag\\
		&\Ric_{aa} = \frac{2}{f^2}-\frac{f''}{f}-\frac{f'}{f}\Big(\,2\frac{f'}{f}+\frac{3}{r}\Big)\, ,\notag\\
	&\Ric_{ii} = \Big(\frac{1}{(1-\hat \epsilon)^2}-1\Big)\frac{2}{r^2}-\frac{3}{r}\frac{f'}{f}\, .
\end{align}

We will choose a smooth function $f(r)$ of the form
\begin{align}
	f(r) \equiv \begin{cases}
		c_5 r^\alpha  & \text{ if } r\leq 10R_6\, ,\notag\\
		f''<0 & \text{ if } 10R_5\leq r\leq 10^{4}R_6\, ,\notag\\
		\hat\delta & \text{ if } r\geq 10^{4}R_6\, .
	\end{cases}
\end{align}

If $\hat\delta = \hat\delta(c_5,R_6)$ then we can build such a function $f$ by smoothing the function  $f(r) = \min\{c_5 r^\alpha , c_5 (10^3 R_6)^\alpha\}\equiv \min\{c_5 r^\alpha , \hat\delta\}$ .  By plugging this into \eqref{e:step2:X7:Ricci} we see that $\Ric\geq 0$.\\

This completes the construction of $X$, let us finally check that all the requirements in Proposition \ref{p:step2:2} are satisfied.

Globally on $X$, we have a doubly warped product metric
\begin{equation}
g = 	dr^2 + f(r)^2g_{S^3} + h(r)^2 g_{S^3}
\end{equation}
hence, the $(1,k)$-Hopf action on $S^3\times S^3$ induces an isometric $S^1$ action on $X$. By construction, in $U=\{r \le 1\}$ the metric is $g = dr^2 + \delta^2g_{S^3}	+ (1-\epsilon)^2 r^2 g_{S^3}$. 

If we let $R_7\equiv 10^5R_6$ and $U_7\equiv \{r\leq R_7\}$ then we see that $X_7\setminus U_7$ is isometric to
\begin{align}
	g_7 = dr^2 +\hat\delta^2 g_{S^3}+ (1-\hat\epsilon)^2r^2 g_{S^3}\, ,
\end{align}
through a map induced by $\phi$, the diffeomorphism built in Lemma \ref{l:step2:refined_twisting}. Hence, it is immediate to see that in these coordinates the $S^1$ action on $X$ identifies with the $(1,0)$-Hopf action.
$\qed$

\vspace{.5cm}

\section{Step 3: Extending the Action}\label{s:Step3_proof}

In this Section we focus on the third step of the construction, which was outlined in Section \ref{ss:step 3 outline}.  The primary goal of this step of the construction is to extend the $\Gamma_j$ action on $M_j$ to a $\Gamma_{j+1}$ action on $M_{j+1}$.  Recall that $\Gamma_{j+1}$ is generated by $\gamma_{j+1}$ such that $\gamma_{j+1}^{k_{j+1}}=\gamma_j \in \Gamma_j$.  

In order to accomplish this we will begin with a model space $\cB_{j+1}=\cB(\epsilon_{j+1},\delta_{j+1})\approx S^3\times \dR^4$, pluck out $k_{j+1}$ balls $S^3\times D^4$ and glue in copies of $M_j$.  To do this precisely, and in order to preserve the geometry in the process, the main technical Proposition we need to prove in this Section is the following:\\

\begin{proposition}[Step 3:  Action Extension]\label{p:step3:2}
	Let $\epsilon,\epsilon',\delta>0$ with $0<\epsilon-\epsilon'\le \frac{\epsilon}{10^2 }$, and let $\hat\Gamma\leq \dQ/\dZ\subseteq S^1$ be a finite subgroup with $\Gamma = \langle \gamma, \hat\Gamma\rangle$ such that $\hat\gamma\equiv \gamma^k$ is the generator of $\hat\Gamma$.  Then for $\hat\epsilon\leq \hat\epsilon(\epsilon,\epsilon')$ there exists a pointed space $(\tilde \cA,p)$, isometric to a smoooth Riemannian manifold with $\Ric\ge 0$ away from $k+1$ singular three spheres, with an isometric and free action by $\Gamma$ such that
\begin{enumerate}
	\item There exists $\Gamma$-invariant set $B_{10^{-1}}(p)\subseteq U'\subseteq B_{10}(p)$ which is isometric to $S^3_\delta\times B_{1}(0)\subseteq S^3_\delta\times C(S^3_{1-\epsilon'})$ and such that $\Gamma$ is induced by the $(1,k)$-Hopf action on $S^3_\delta\times S^3_{1-\epsilon'}$ ,
	\item There exists $\Gamma$-invariant set $B_{10^3 k}(p)\subseteq U\subseteq B_{10^5 k}(p)$ such that $\tilde \cA\setminus U$ is isometric to $ S^3_\delta \times A_{10^4k,\infty}(0)     \subseteq S^3_\delta\times C(S^3_{1-\epsilon})$ and such that $\Gamma$ is induced by the $(1,k)$-Hopf action on $S^3_\delta\times S^3_{1-\epsilon}$ 
	\item There exists $\hat\Gamma$-invariant sets $S^3_\delta\times B_{2^{-1}}(x^a)\subseteq V^a\subseteq S^3_\delta\times B_{2}(x^a)$ with $d(S^3_\delta\times \{x^a\},S^3_\delta\times \{p\})=10^2 k$ which are isometric to $S^3_\delta\times B_{1}(0)\subseteq S^3_\delta\times C(S^3_{1-\hat\epsilon})$ and such that $\hat\Gamma$ is induced by the $(1,0)$-Hopf action on $S^3_\delta\times S^3_{1-\hat\epsilon}$ .
\end{enumerate}
\end{proposition}

The construction will come in three steps.  We will begin with $\tilde \cA_0 =S^3_\delta\times C(S^3_{1-\epsilon'})$, which we see by $(1)$ is what our space should look like on small scales.  In Section \ref{ss:constructingA1} we will construct $\tilde \cA_1$ by adding a bend to $\tilde \cA_0$.  The effect of this will be that on large scales the space looks like $S^3_\delta\times C(S^3_{1-\epsilon})$, however on some middle scale $\tilde \cA_1$ will be isometric to an annulus in a $4$-sphere $S^4_R$, where $R=R(\epsilon',\epsilon,k)$ is potentially very large.\\

In Section \ref{ss:constructionhatA} we will construct the gluing pieces $\hat \cA$ by beginning with $\hat \cA_0 =S^3_\delta\times C(S^3_{1-\hat\epsilon})$ and bending it in an analogous manner to which we built $\tilde \cA_1$.  However, $\hat \cA$ will have boundary and near the boundary will be isometric to a small annulus in the $4$-sphere $S^4_R$.\\

We see we are now able to glue copies of $\hat \cA$ into $\tilde \cA_1$ as an open set near boundary of $\hat \cA$ is isometric to a region in $\tilde \cA_1$.  We will want to glue $k$ copies of $\hat \cA^a=\hat \cA$ into $\tilde \cA_1$ in order to complete our construction, however as in the discussion in Section \ref{ss:step 3 outline} we need to be careful about the choice of gluing maps.  This will be done in Section \ref{ss:constrtildecA}.

\vspace{.3cm}

\subsection{Constructing $\cA_1$}\label{ss:constructingA1}

Let us begin with $\tilde\cA_0 \equiv S^3_\delta \times C(S^3_{1-\epsilon'})$, which geometrically has the metric 
\begin{align}
	g_0 \equiv dr^2 + \delta^2 g_{S^3}+ (1-\epsilon')^2 r^2 g_{S^3} \, .
\end{align}

Let $U_0 = \{r\leq 1\}$, then we will build $\tilde\cA_1$ by modifying the above metric on the region $\tilde\cA_0\setminus U_0$.  We will look for a metric which is of the form

\begin{align}
	g_1 \equiv dr^2 + \delta^2 g_{S^3}+ h(r)^2 g_{S^3} \, .
\end{align}

The nonzero Ricci curvatures of a warped metric as above are

\begin{align}\label{e:step3:A1:Ricci}
	&\Ric_{rr} = -3\frac{h''}{h}\, ,\notag\\
	&\Ric_{aa} = \frac{2}{\delta^2}\, , \notag\\
	&\Ric_{ii} = 2 \frac{1-(h')^2}{h^2}-\frac{h''}{h}\, .
\end{align}

In order to choose our warping function $h(r)$ let us begin by defining the following three functions:
\begin{align}
	&h_1(r)\equiv (1-\epsilon')r\, ,\notag\\
	&h_2(r)\equiv R \sin\big(R^{-1}(r-r_R)\big)\, ,\notag\\
	&h_3(r)\equiv (1-\epsilon)(r+r_\epsilon)\, .
\end{align}

Our goal will be to show for appropriate constants $R, r_R, r_\epsilon$, $\hat r$ and $r_1\in [10k, 10^3k]$ that we can choose $h(r)$ in the form
\begin{align}
	h(r) \equiv \begin{cases}
		h_1(r) & \text{ if } r\leq 10\, ,\notag\\
		h''<0 & \text{ if } 10\leq r\leq 10^{4}k\, ,\notag\\
		h_2(r)& \text{ if } r_1 \leq r\leq r_1 +  \hat r\, ,\notag\\
		h_3(r) & \text{ if } r\geq 10^{4}k\, .\notag\\
	\end{cases}
\end{align}

If we have an $h(r)$ then by \eqref{e:step3:A1:Ricci} we have that
\begin{align}\label{e:step3:A1:Ricci_compute}
	&\Ric_{rr} \geq 0\, ,\;\;
	\Ric_{aa} = \frac{2}{\delta^2}>0 \, ,\;\;
	\Ric_{ii} \geq \frac{1-(1 - \epsilon')^2}{h^2}>0
	\, .
\end{align}

Recall now that $\epsilon'<\epsilon$ have already been fixed, and the lines $h_1(r)$ and $h_3(r)$ must intersect at a unique point.  Let us choose $r_\epsilon$ uniquely so that the point of intersection is at $10^2 k$.  In particular, we can solve for $r_\epsilon$ as
\begin{align}
	r_\epsilon \equiv \frac{\epsilon-\epsilon'}{1-\epsilon}10^2 k\, .
\end{align}
With $r_\epsilon$ fixed, let us observe that for any $R\geq 0$ there is a unique smallest $r_R\in (0, 2\pi R]$ such that $h_2(r)\leq \min\{h_1(r),h_3(r)\}$ for every $r>0$.  Note for this $r_R$ that $h_2(r)$ intersects $h_1(r)$ and $h_3(r)$ at most once, but must intersect one of them (otherwise $h_2<\min\{h_1,h_3\}$ and we could have decreased $r_R$).  On the other hand, note that for $R$ small we must have that $h_2(r)$ intersects $h_1(r)$, while for $R$ large we must have that $h_2(r)$ intersects $h_3(r)$.  We can then also find a unique value of $R\, =R(\epsilon,\epsilon',k)$ for which $h_2$ intersects both $h_1$ and $h_3$.  Let us fix this as our value of $R$ and hence $r_R$, and let us call these intersection points $s_1 < 10^2 k< s_2$ respectively.\\

In order to estimate the value of $r_R$ let us observe that $h_2(r)\leq r-r_R$, and as such we get that
\begin{align}
	&(1-\epsilon')s_1 = h_1(s_1)=h_2(s_1)\leq s_1-r_R\notag\\
	&\implies r_R\leq \epsilon' s_1\leq 10^2 k\,\epsilon' \, .
\end{align}

Let us observe that $\dot h_2(s_1)=\dot h_1(s_1)$ and $\dot h_2(s_3)=\dot h_3(s_3)$ to get the relations
\begin{align}
	&\cos(R^{-1}(s_1-r_R)) = 1-\epsilon'\, ,\notag\\
	&\cos(R^{-1}(s_2-r_R)) = 1-\epsilon\, ,\notag\\
	\implies& \big|s_1-r_R-\sqrt{2\epsilon'}R\big| \leq 10\epsilon' R\notag\\
	& \big|s_2-r_R-\sqrt{2\epsilon}R\big| \leq 10\epsilon R\, ,\notag\\
	\implies&|s_1-\sqrt{2\epsilon'}R|\leq 10(10 k +R)\epsilon'\, ,\notag\\
	&|s_2-\sqrt{2\epsilon}R|\leq 10(10 k +R)\epsilon\, ,
\end{align}
where we used the Taylor expansion of $\cos(x)$ and the fact that $s_1-r_R, s_2 - r_R \in (0, 2\pi R]$.\\

Using that $s_1\leq 10^2 k\leq s_2$, this gives the estimate on $R$:

\begin{align}
	|10^2 k - \sqrt{2\epsilon}R | 
	&\leq 
	 \max\{ |s_1 - \sqrt{2\epsilon}R |, |s_2 - \sqrt{2\epsilon}R|  \}
	\\& \le \sqrt{2\epsilon}R
	\Big(\frac{\sqrt{\epsilon} - \sqrt{\epsilon'}}{\sqrt{\epsilon}}\Big) + 20(10k + R)\epsilon \, .
\end{align}
Hence, if $\sqrt{\epsilon} - \sqrt{\epsilon'} \le 10^{-1}\sqrt{\epsilon}$, we can deduce
\begin{equation}
	\big|R - \frac{10^2 k}{\sqrt{2\epsilon}}\,\big| \leq 10^3 k\Big( \sqrt{\epsilon} + \frac{\sqrt{\epsilon} - \sqrt{\epsilon'}}{\sqrt{\epsilon}} \Big)\, .
\end{equation}

From this we get the estimate
	\begin{align}
		\big|(s_2 - s_1) - \frac{\sqrt{\epsilon}-\sqrt{\epsilon'}}{\sqrt{\epsilon}}10^2 k\big|
		\leq 10^4k \sqrt{\epsilon} \Big( 1  + \Big(\frac{ \sqrt{\epsilon} - \sqrt{\epsilon'}}{\sqrt{\epsilon}}\Big)^2 \Big)\, .
	\end{align}
Hence, we can define $\hat r := \frac{\sqrt{\epsilon}-\sqrt{\epsilon'}}{100\sqrt{\epsilon}}10^2 k \le \frac{1}{10}\cdot 10^2 k$.\\

Then we may build $h(r)$ by smoothing out 
\begin{align}
	h(r) \equiv \begin{cases}
		h_1(r) & \text{ if } r\leq s_1\, ,\notag\\
		h_2(r)& \text{ if } s_1\leq r\leq s_2\, ,\notag\\
		h_3(r) & \text{ if } r\geq s_2\, .
	\end{cases}
\end{align}

\vspace{.3cm}
\subsection{Construction of $\hat A$}\label{ss:constructionhatA}

Let us begin with $\hat \cA_0\equiv S^3_\delta\times C(S^3_{1-\hat\epsilon})$ on the domain $\hat U_0\equiv \{r\leq \hat r\}$, where  $\hat r\equiv  \frac{\sqrt{\epsilon}-\sqrt{\epsilon'}}{100\sqrt{\epsilon}}10^2 k$ was defined in the previous section.  The metric on $\hat\cA_0$ may be written
\begin{align}
	\hat g_0 \equiv dr^2 +\delta^2 g_{S^3}+ (1-\hat\epsilon)^2 r^2 g_{S^3}\, .
\end{align}

In order to construct $\hat\cA$ we will need to alter the geometry by looking for a metric of the form
\begin{align}
	\hat g \equiv dr^2 +\delta^2 g_{S^3}+ h(r)^2 g_{S^3}\, .
\end{align}

We will build $h(r)$ in a manner analogous to the previous subsection.  Let us start by looking at the functions
\begin{align}
	&h_1(r)\equiv (1-\hat\epsilon)r\, ,\notag\\
	&h_2(r)\equiv R \sin\big(R^{-1}(r-\hat r_R)\big)\, ,
\end{align}
where $R\, {=R(\epsilon,\epsilon',k)}>0$ has been fixed in the previous subsection.  Observe that $h_1$ and $h_2$ will intersect twice for $\hat r_R$ small, and for $\hat r_R$ large they will not intersect.  Let us choose $\hat r_R$ uniquely so that they intersect at $\hat s$ precisely once.  Note that at this point of intersection $\hat s$ we will then have $\dot h_1(\hat s) = \dot h_2(\hat s)$, which is the equation
\begin{align}
	&\cos(R^{-1}(\hat s-\hat r_R))= 1-\hat\epsilon\, ,\notag\\
	&\implies \big|\hat s-\hat r_R - \sqrt{2\hat\epsilon} R\big|\leq \hat\epsilon R\, .
\end{align}
Additionally, using that $h_2(r)\leq r-\hat r_R$ we have the inequality
\begin{align}
	&(1-\hat\epsilon)\hat s = h_1(\hat s)=h_2(\hat s)\leq \hat s-\hat r_R\, ,\notag\\
	&\implies \hat r_R\leq \hat \epsilon \hat s\, .
\end{align}
Combining the last two estimates we conclude
\begin{align}
	&\big|\hat s- \sqrt{2\hat\epsilon} R\big|\leq (\hat\epsilon+\sqrt{2}\hat\epsilon^{3/2})R\, ,\notag\\
	&\hat r_R\leq 2\hat\epsilon R\, .
\end{align}

Let us now choose $\hat\epsilon\leq \hat\epsilon(R)\leq \hat\epsilon(\epsilon,\epsilon')$ so that $\hat s\leq \frac{1}{2}\hat r$.  Then we can define $h(r)$ for $r\leq 2\hat s$ by smoothing 

\begin{align}
	h(r) \equiv \begin{cases}
		h_1(r) & \text{ if } r\leq \hat s\, ,\notag\\
		h_2(r)& \text{ if } \hat s\leq r\leq 2\hat s\, .
	\end{cases}
\end{align}

In particular, for $\hat \cA = \{r\leq 2\hat s\}$ we see that a neighborhood of the boundary is isometric to the product of $S^3_{\delta}$ with an annulus in $S^4_R$.  The verification that $\Ric\ge 0$ with this choice of $h$ is completely analogous to the one discussed in the previous subsection, using \eqref{e:step3:A1:Ricci}.

\vspace{.3cm}
\subsection{Constructing $\tilde A$}\label{ss:constrtildecA}

As our final step let us now denote $\hat\cA^a$ for $a=0,\ldots,k-1$ as $k$ copies of our constructed neck from the last subsection.  If we denote $\tilde r \equiv 2\hat s-\hat r_R < r_R$ then we have
\begin{align}
	\partial \hat\cA^a 
	= 
	\{r=2\hat s\}\equiv S^3_\delta\times \partial B_{\tilde r}
	\subseteq 
	S^3_\delta\times S^4_R\, ,
\end{align}
be the boundary of our neck region.  The boundary, and indeed all of $\hat \cA^a$ near the boundary, is isometric to a neighborhood in $S^3_\delta\times S^4_R$.  \\

Recall that $\tilde\cA_1$ has a metric of the form
\begin{align}
	\tilde g_1 \equiv dr^2+\delta^2 g_{S^3}+h(r)^2g_{S^3}\, ,
\end{align}
such that the region $\{10^2k-\hat r\leq r\leq 10^2k+\hat r\}$ is isometric to an annulus in $S^3_\delta\times S^4_{R}$.  In the coordinates from the above description let us choose the point $x^0=(10^2 k,e,e)\in \tilde\cA_1$ so that $r(x^0) = 10^2 k$, {where $e\in S^3$ is the identity}.  Let $x^a=(10^2 k,(2\pi a/k)\cdot e,e)$ be the Hopf rotation of $x^0$ by angle $2\pi a/k$.  Consider the domains $S^3_\delta\times B_{\tilde r}(x^a)$, and let 
\begin{align}
	\varphi^0:\partial \hat\cA^0=S^3_\delta\times \partial B_{\tilde r}\to S^3_\delta\times \partial B_{\tilde r}(x^0)\, ,
\end{align}
be the canonical isometry which fixes the $S^3_\delta$ factor.  Let $\gamma\in \Gamma$ be the action which Hopf rotates the first $S^3$ factor by $2\pi/|\gamma|$ and Hopf rotates the second $S^3$ factor by $2\pi/k$.  Then we define the mappings
\begin{align}
	\varphi^a:\partial \hat\cA^a\to S^3_\delta\times \partial B_{\tilde r}(x^a)\, , \text{ by }\, \varphi^a \equiv \gamma^a \cdot\varphi^0\, .
\end{align}
This allows us to define our space
\begin{align}
	\tilde\cA \equiv \Big(\tilde\cA_1\setminus \bigcup_a S^3_\delta\times B_{\tilde r}(x^a)\Big)\bigcup_{\varphi^a} \hat\cA^a\, .
\end{align}
Note that this gluing extends isometrically to a small neighborhood, and $\tilde\cA$ is a smooth manifold (away from $(k+1)$ singular three spheres) with $\Ric\geq 0$.  Additionally, we have the required property that $\Gamma=\langle\gamma\rangle$ acts isometrically on $\tilde\cA$ such that in both the $\tilde\cA_1$ domain and the glued domains $\hat\cA^a$ the action of $\gamma^k=\hat\gamma$ is purely by Hopf rotation of the $S^3_\delta$ factor. $\qed$\\

\vspace{.5cm}


\vspace{.5cm}
\section{Geometry of the Mapping Class Group of $S^3\times S^3$}\label{s:mappingS3S3}

The primary goal of this Section is to prove Lemma \ref{l:main_mapping_class}, which we restate for the readers convenience below:

\begin{lemma}[Mapping Class Group and Ricci Curvature on $S^3\times S^3$]\label{l:main_mapping_class:2}
	Let $g_0=g_{S^3\times S^3}$ be the standard metric on $S^3\times S^3$.  Then given $\phi\in {\rm Diff}(S^3\times S^3)$\, there exists a smooth family $g_t$ of metrics  with $\Ric_{g_t}>0$ such that $g_0$ is the standard metric and $g_1=\phi^*g_0$.  That is, the orbit $\pi_0{\rm Diff}(S^3\times S^3)\cdot [g_{S^3\times S^3}]$ of the mapping class group lives in a connected component of $\cM^+_0(S^3\times S^3)$, the space of metrics with strictly positive Ricci curvature.
\end{lemma}

Let us begin by recalling some basic structure of the mapping class group $\pi_0\text{Diff}(S^3\times S^3)$.  So consider $\text{Diff}(S^3\times S^3)$, the diffeomorphism group of $S^3\times S^3$, and let $\pi_0\text{Diff}(S^3\times S^3)$ denote the connected components of it.  This set inherits a group structure, and the mapping class group of $S^3\times S^3$ is a discrete group.  There is a natural surjective mapping 
\begin{align}\label{e:mappingclass_SL2Z}
	\kappa:\pi_0\text{Diff}(S^3\times S^3)\to \text{SL}(2,\dZ)\, ,
\end{align}
given by looking at the action of $[\phi]\in \pi_0\text{Diff}(S^3\times S^3)$ on the homology ring $\kappa[\phi] = [\phi_*]:H_3(S^3\times S^3)\to H_3(S^3\times S^3)$.  It is now well understood, see \cite{Kreck},\cite{Krylov}, that the kernel 
\begin{align}
	\cK=\ker\kappa\triangleleft\pi_0\text{Diff}(S^3\times S^3)
\end{align}
is a 2-step nilpotent group and obeys the short exact sequence 
\begin{align}
	0\to \dZ_{28}\to \cK\to \dZ\times\dZ\to 0\, .
\end{align}
This kernel and its $\dZ_{28}$ extension are closely related to the exotic differentiable structures on seven manifolds.  The group $\pi_0\text{Diff}(S^3\times S^3)/\cK=\text{SL}(2,\dZ)$ is generated by the two diffeomorphisms
\begin{align}
	&\phi_1(g_1,g_2) = (g_1,g_1g_2)\, ,\notag\\
	&\phi_2(g_1,g_2) = (g_1g_2^{-1},g_2)\, ,
\end{align}
see \cite{Krylov}.  On the other hand the kernel $\cK$, which is the collection of diffeomorphisms whose induced action on the homology is trivial, can be identified as the nilpotent group
\begin{align}
	\cK = \begin{bmatrix}
    1 & a & c \\        
    0 & 1 & b \\
    0 & 0& 1   
\end{bmatrix}\, ,\;\;a,b\in \dZ\, ,\;\;c\in \dZ_{28}\, .
\end{align}
It is generated by two elements, which are given by the diffeomorphisms
\begin{align}
	&\phi^\cK_1(g_1,g_2) = (g_2g_1g_2^{-1},g_2)\, ,\notag\\
	&\phi^\cK_2(g_1,g_2) = (g_1,g_1g_2g_1^{-1})\, .
\end{align}

In particular, if we consider the diffeomorphisms
\begin{align}
	&\phi_3(g_1,g_2) = (g_1,g_2g_1^{-1})\, ,\notag\\
	&\phi_4(g_1,g_2) = (g_2g_1,g_2)\, ,
\end{align}
then we see that $\{\phi_1,\phi_2,\phi_3,\phi_4\}$ generates $\pi_0\text{Diff}(S^3\times S^3)$ .\\

Now it follows from Theorem \ref{t:equivariant_mapping_class:2} and Remark \ref{r:equivariant_mapping_diffeo} that there exist families of metrics $g_{1,t},g_{2,t},g_{3,t},g_{4,t}$ with $\Ric>0$ such that
\begin{align}
	&g_{j,0} =g_{S^3\times S^3}\, ,\notag\\
	&g_{j,1} =\phi_j^*g_{S^3\times S^3}\, .
\end{align}

To prove the Theorem it is now enough to show for each $[\phi]\in \pi_0\text{Diff}(S^3\times S^3)$ that there exists some representative $\phi\in [\phi]$ for which the Theorem holds, as we can clearly vary the metric within a fixed class by the diffeomorphism action itself.  Thus let $\phi = \phi_{j_k}\circ\cdots\circ\phi_{j_1}$ represent any element of the mapping class group.  If we denote $j_0=0$ with $\phi_0=Id$, then let us define the family of metrics
\begin{align}
	g_t \equiv \phi_{j_\ell}^*\circ\cdots \circ \phi_{j_0}^* g_{j_{\ell+1},k(t-\frac{\ell}{k})}\, \, \, \text{ if }\, \, \, t\in \big[\frac{\ell}{k},\frac{\ell+1}{k}\big ]\, .
\end{align}  
Then we have that $\Ric_t>0$ with $g_0=g_{S^3\times S^3}$ and $g_1 = \phi^*g_{S^3\times S^3}$, as claimed.  $\qed$\\


\end{document}